\documentclass[aos]{imsart2}
    \RequirePackage{amsthm,amsmath,amsfonts,amssymb}
    \RequirePackage[numbers]{natbib}
    \RequirePackage[colorlinks,citecolor=blue,urlcolor=blue]{hyperref}
    \RequirePackage{graphicx}
    
    \usepackage[inline]{enumitem}
    
    \numberwithin{equation}{section}
    \usepackage{dsfont}
    \usepackage{mathtools} 
    \usepackage{color} 
    \usepackage{tikz}
    
    \newcommand{\bbC}{\mathbb{C}}

    \newcommand{\bbN}{\mathbb{N}}
    \newcommand{\bbP}{\mathbb{P}}
    
    \newcommand{\bbR}{\mathbb{R}}
    \newcommand{\bbS}{\mathbb{S}}
    
    \newcommand{\bbZ}{\mathbb{Z}}
    
    \newcommand{\cA}{\mathcal{A}}
    \newcommand{\cB}{\mathcal{B}}
    \newcommand{\cC}{\mathcal{C}}

    \newcommand{\cF}{\mathcal{F}}
    
    \newcommand{\cH}{\mathcal{H}}
    \newcommand{\cI}{\mathcal{I}}
    \newcommand{\cJ}{\mathcal{J}}
    \newcommand{\cK}{\mathcal{K}}
    \newcommand{\cL}{\mathcal{L}}

    \newcommand{\cS}{\mathcal{S}}
    \newcommand{\cT}{\mathcal{T}}
    \newcommand{\cU}{\mathcal{U}}
    \newcommand{\cV}{\mathcal{V}}
    \newcommand{\cW}{\mathcal{W}}
    \newcommand{\cX}{\mathcal{X}}
    
    \newcommand{\cZ}{\mathcal{Z}}
    
    \newcommand{\kvec}{\mathbf{k}} 
    
    \newcommand{\norm}[2]{     \| #1       \|_{ #2 }}
    
    \newcommand{\scalar}[2]{     ( #1       )_{ #2 }}

    \newcommand{\dual}[1]{{#1}^{*}}
    
    \newcommand{\tr}{\operatorname{tr}}
    \newcommand{\rd}{\mathrm{d}}
    \newcommand{\from}{\colon}

    \newcommand{\mv}[1]{{\boldsymbol{\mathrm{#1}}}}
    \newcommand{\trsp}{\ensuremath{\top}}
    \newcommand{\eps}{\varepsilon}
    
    \newcommand{\GP}{Z}
    \newcommand{\gp}{z}
    
    \newcommand{\proper}{\mathsf}
    \newcommand{\pN}{\proper{N}}
    \newcommand{\normal}{\pN}
    
    \newcommand{\MU}{\widetilde{\mu}}
    \newcommand{\muc}{\mu_{\mathrm{c}}}
    \newcommand{\mus}{\mu_{\mathrm{s}}}
    \newcommand{\MUc}{\MU_{\mathrm{c}}}
    \newcommand{\MUs}{\MU_{\mathrm{s}}}
    
    \newcommand{\pE}{\proper{E}}
    \newcommand{\PE}{\widetilde{\pE}}
    \newcommand{\pEc}{\pE_{\mathrm{c}}}
    \newcommand{\pEs}{\pE_{\mathrm{s}}}
    \newcommand{\PEc}{\PE_{\mathrm{c}}}
    \newcommand{\PEs}{\PE_{\mathrm{s}}}
    
    \newcommand{\pV}{\proper{Var}}
    \newcommand{\PV}{\widetilde{\pV}}
    \newcommand{\pVc}{\pV_{\mathrm{c}}}
    \newcommand{\pVs}{\pV_{\mathrm{s}}}
    \newcommand{\PVc}{\PV_{\mathrm{c}}}
    \newcommand{\PVs}{\PV_{\mathrm{s}}}
    
    \newcommand{\Cov}{\proper{Cov}}
    \newcommand{\COV}{\widetilde{\Cov}}

    \newcommand{\CC}{\widetilde{\cC}}
    \newcommand{\CH}{\widetilde{\cH}}
    
    \newcommand{\hn}{h_n}
    \newcommand{\Hn}{\widetilde{h}_n}
    \newcommand{\hnc}{\hn^{\mathrm{c}}}
    \newcommand{\hns}{\hn^{\mathrm{s}}}
    \newcommand{\Hnc}{\Hn^{\mathrm{c}}}
    \newcommand{\Hns}{\Hn^{\mathrm{s}}}
    
    \newcommand{\err}{e}
    \newcommand{\ERR}{\widetilde{e}}
    \newcommand{\errc}{\err_{\mathrm{c}}}
    \newcommand{\errs}{\err_{\mathrm{s}}}
    \newcommand{\ERRc}{\ERR_{\mathrm{c}}}
    \newcommand{\ERRs}{\ERR_{\mathrm{s}}}
    
    \newcommand{\sinc}{\operatorname{sinc}}

    \usepackage{chapterbib}
    \startlocaldefs
    \theoremstyle{plain}

    \newtheorem{theorem}{Theorem}[section]
    \newtheorem{lemma}[theorem]{Lemma}
    \newtheorem{proposition}[theorem]{Proposition}
    \newtheorem{corollary}[theorem]{Corollary}
    \theoremstyle{remark}
    
    \newtheorem{remark}[theorem]{Remark}
    \newtheorem{assumption}[theorem]{Assumption}
    \newtheorem{example}[theorem]{Example}
    \renewcommand{\theenumi}{\normalfont{(\roman{enumi})}}
    
    \endlocaldefs
\usepackage{subfiles}
\begin{document}

\begin{cbunit}
\begin{frontmatter}
			
\title{Necessary and sufficient conditions \\ 
	for asymptotically optimal linear prediction \\ 
	of random fields on compact metric spaces}
\runtitle{Linear prediction for misspecified random fields}
%\thankstext{T1}{A sample additional note to the title.}

\begin{aug}
	%%%%%%%%%%%%%%%%%%%%%%%%%%%%%%%%%%%%%%%%%%%%%%%
	%% Only one address is permitted per author. %%
	%% Only division, organization and e-mail is %%
	%% included in the address.                  %%
	%% Additional information can be included in %%
	\author[A]{\fnms{Kristin} \snm{Kirchner}\ead[label=e1]{k.kirchner@tudelft.nl}}
	\and  
	\author[B]{\fnms{David} \snm{Bolin}\ead[label=e2]{david.bolin@kaust.edu.sa}} 
	%%%%%%%%%%%%%%%%%%%%%%%%%%%%%%%%%%%%%%%%%%%%%%
	%% Addresses                                %%
	%%%%%%%%%%%%%%%%%%%%%%%%%%%%%%%%%%%%%%%%%%%%%%
	\address[A]{Delft Institute of Applied Mathematics, 
		Delft University of Technology,
		\printead{e1}} 
	\address[B]{CEMSE Division, King Abdullah 
		University of 
		Science and Technology, 
		\printead{e2}}
\end{aug}

\begin{abstract}
	Optimal linear prediction 
	(aka.\ kriging) of a 
	random field $\{Z(x)\}_{x\in\mathcal{X}}$ 
	indexed by a compact metric space 
	$(\mathcal{X},d_{\mathcal{X}})$
	can be obtained if the mean value function 
	$m\colon\mathcal{X}\to\mathbb{R}$ and the covariance 
	function 
	$\varrho\colon\mathcal{X}\times\mathcal{X}\to\mathbb{R}$ 
	of~$Z$ are known. 
	We consider the problem of 
	predicting the value of $Z(x^*)$ 
	at some location $x^*\in\mathcal{X}$ based on observations at 
	locations $\{x_j\}_{j=1}^n$ 
	which 
	accumulate at $x^*$  
	as $n\to\infty$ 
	(or, more generally, 
	predicting $\varphi(Z)$ 
	based on $\{ \varphi_j(Z) \}_{j=1}^n$
	for linear functionals $\varphi, \varphi_1, \ldots, \varphi_n$). 
	Our main result characterizes 
	the asymptotic performance 
	of linear predictors 
	(as $n$ increases) 
	based on an incorrect second order 
	structure $(\widetilde{m},\widetilde{\varrho})$, 
	without any restrictive assumptions 
	on $\varrho, \widetilde{\varrho}$ 
	such as stationarity. 
	We, for the first time, 
	provide necessary and sufficient conditions 
	on $(\widetilde{m},\widetilde{\varrho})$ 
	for asymptotic optimality of 
	the corresponding linear predictor  
	holding uniformly 
	with respect to $\varphi$. 
	These general results are illustrated by 
	weakly stationary random fields on $\mathcal{X}\subset\mathbb{R}^d$ 
	with Mat\'ern 
	or periodic covariance functions, 
	and 
	on the sphere $\mathcal{X}=\mathbb{S}^2$ 
	for the case of two   
	isotropic covariance functions. 
\end{abstract}



\end{frontmatter}

%====================================================================
\section{Introduction}\label{section:intro}
%====================================================================

Optimal linear prediction of random fields, 
often also called kriging, 
is an important and widely used technique 
for interpolation of spatial data. 
Consider a random field 
$\{\GP(x) :  x\in\cX  \}$ 
on a compact topological space $\cX$ such as 
a closed and bounded subset of~$\bbR^d$. 
Assume that $\GP$ is  
almost surely continuous on $\cX$ 
and that we want to predict its value  
at a location $x^*\in\cX$ based on a set of observations 
$\{\GP(x_j)\}_{j=1}^n$ for locations 
$x_1,\ldots,x_n\in\cX$ 
all distinct from $x^*$. 
The kriging predictor is the linear predictor 
$\widehat{\GP}(x^*) 
= 
\alpha_0 + \sum_{j=1}^n \alpha_j \GP(x_j)$ 
of $\GP(x^*)$ based on the observations, 
where the coefficients 
$\alpha_0,\ldots,\alpha_n\in\bbR$ 
are chosen such that the variance of the error  
$(\widehat{\GP} - \GP)(x^*)$ is minimized. 
By letting $m(\,\cdot\,)$ and $\varrho(\,\cdot\,, \,\cdot\,)$  
denote the mean and the covariance function of $\GP$,  
we can express $\widehat{\GP}(x^*)$ as 
\begin{equation}\label{eq:kriging_equation}
\widehat{\GP}(x^*) 
= 
m(x^*) 
+ 
\mv{c}_n^\trsp \mv{\Sigma}_n^{-1} (\mv{\GP}_n - \mv{m}_n), 
\end{equation}
where 
$\mv{\GP}_n := (\GP(x_1),\ldots, \GP(x_n))^{\trsp}$, 
$\mv{m}_n := ( m(x_1),\ldots, m(x_n) )^{\trsp}$, 
$\mv{\Sigma}_n \in \bbR^{n\times n}$ 
has elements $[\mv{\Sigma}_n]_{ij} := \varrho(x_i,x_j)$, and  
$\mv{c}_n := ( \varrho(x^*, x_1), \ldots, \varrho(x^*,x_n) )^{\trsp}$. 

In applications the 
mean and covariance functions are rarely known 
and therefore need to be estimated from data. 
It is thus of interest to study the effect which  
a misspecification of the mean  
or the covariance function 
has on the efficiency of the linear predictor. 
Stein~\cite{stein88,stein90} 
considered the situation 
that the sequence 
$\{x_j\}_{j\in\bbN}\subset \bbR^d$ 
has $x^*$ 
as a limiting point and the  
predictor~$\widehat{\GP}$ 
is computed using misspecified mean and covariance functions,  
$\widetilde{m}$ and $\widetilde{\varrho}$. 
His main outcome was that the best linear predictor based on 
$(\widetilde{m}, \widetilde{\varrho})$ is asymptotically efficient, 
as $n\rightarrow\infty$,  
provided that the Gaussian measures corresponding to $(m, \varrho)$ 
and $(\widetilde{m}, \widetilde{\varrho})$ are equivalent  
(see Appendix~A).  
This result in fact holds uniformly 
with respect to~$x^*$ 
and, moreover, uniformly 
for each linear functional $\varphi$ such that 
$\varphi(\GP)$ has finite variance \cite{stein90}. 

For stationary covariance functions, 
there exist simple conditions for verifying 
whether the corresponding Gaussian measures are 
equivalent 
\cite{Arafat2018,Bevilacqua2019,stein93} 
and thus if the linear predictions are asymptotically efficient. 
However, for any constant $c\in(0,\infty)$, 
the linear predictor based on $(m, c\varrho)$ 
is equal to that based on $(m, \varrho)$, 
whereas the Gaussian measures corresponding to $(m, c\varrho)$ 
and $(m, \varrho)$ are orthogonal for all $c\neq 1$. 
This shows that equivalence of the measures is a sufficient 
but not necessary condition for asymptotic efficiency. 
Less restrictive conditions have been derived for some specific cases 
such as periodic processes on $[0,1]^d$ and 
weakly stationary random fields on~$\bbR^d$ 
observed on a lattice \cite{stein97, stein99-paper}. 
With these results in mind, 
an immediate question is if one can find 
\emph{necessary and sufficient} conditions 
for uniform asymptotic efficiency of linear prediction 
using misspecified mean and covariance functions. 
The aim of this work is to show that this indeed is the case. 

We derive necessary and sufficient conditions 
for general second order structures~$(m, \varrho)$ 
and $(\widetilde{m}, \widetilde{\varrho})$, 
without any  
restrictive assumptions  
such as periodicity or stationarity.
These conditions are weaker 
compared to those of the Feldman--H\'ajek theorem 
and thus clearly exhibit 
their fulfillment in the case that 
the Gaussian measures corresponding to  
$(m, \varrho)$ and 
$(\widetilde{m}, \widetilde{\varrho})$
are equivalent 
(see Remark~\ref{rem:comparison-with-FH}).  
Furthermore, our results are formulated for random 
fields on general compact metric spaces, 
which include compact Euclidean domains 
in~$\bbR^d$, 
but also more general domains such as the 
sphere~$\bbS^2$ or metric graphs  
(see Example~\ref{ex:cov-functions}). 
Assuming compactness of the space 
is meaningful, since 
in applications 
the observed locations are always 
contained in some compact subset 
(e.g., a bounded and closed domain in $\bbR^d$).

This general setting is outlined in Section~\ref{section:setting}. 
Our main results are stated in Section~\ref{section:results}  
and proven in Section~\ref{section:proofs}. 
Section~\ref{section:simplified} presents 
simplified necessary and sufficient conditions 
for the two important 
special cases when   
$\varrho,\widetilde{\varrho}$ 
induce the same eigenfunctions, or when
$\varrho,\widetilde{\varrho}$ are 
translation invariant on $\bbR^d$ and 
have spectral densities.  
Section~\ref{section:applications} verifies these 
conditions for weakly stationary  
random fields on $\cX\subset\bbR^d$, 
where $\varrho,\widetilde{\varrho}$ are 
of Mat\'ern type  
or periodic on~$[0,1]^d$.  
We also discuss an example  
on $\cX=\bbS^2$ which, to the best of our knowledge, 
is the first result on 
asymptotically optimal linear prediction on the sphere. 
The Supplementary Material \cite{kbsup} 
contains three appendices 
(Appendix~A/B/C) 
pertaining to this article.

%====================================================================
\section{Setting and problem formulation}\label{section:setting}
%====================================================================

We assume that we are given a 
square-integrable stochastic process 
$\GP\from\cX\times\Omega\to\bbR$ 
defined 
on a complete probability space $(\Omega,\cA,\bbP)$ 
and indexed by a  
connected, compact metric space $(\cX,d_{\cX})$ 
of infinite cardinality.  
In addition, we let ${m\from\cX\to\bbR}$ denote 
the mean value function of~$\GP$   
and assume  
that the covariance function, 
\[
\varrho \from\cX\times\cX \to \bbR, 
\qquad 
\varrho(x,x') 
:= 
\int_\Omega \, (\GP(x,\omega)-m(x)) (\GP(x',\omega)-m(x')) \, \rd\bbP(\omega), 
\] 
is (strictly) positive definite 
and continuous. 
Let $\nu_\cX$ be a strictly positive and finite 
Borel measure on $(\cX,\cB(\cX))$. 
Here and throughout, 
$\cB(\cT)$ denotes 
the Borel $\sigma$-algebra 
on a topological space~$\cT$. 
As the symmetric covariance function $\varrho$ 
is assumed to be 
positive definite   
and continuous, 
the corresponding covariance 
operator, defined by 
\begin{equation}\label{eq:kernel-C}
\textstyle 
\cC\from L_2(\cX,\nu_\cX) \to L_2(\cX,\nu_\cX), 
\qquad 
(\cC w)(x) := \int_{\cX} \varrho(x, x') w(x') \,\rd\nu_\cX(x'), 
\end{equation}
is self-adjoint, 
positive definite, and compact  
on $L_2(\cX,\nu_\cX)$. 
Since $(\cX,d_\cX)$ 
is connected and compact,   
the set $\cX$ is uncountable and 
$L_2(\cX,\nu_\cX)$ is an 
infinite-dimensional separable Hilbert space. 
By compactness of $\cC$, there exists a 
countable system of (equivalence classes of)
eigenfunctions~$\{e_j\}_{j\in\bbN}$ of~$\cC$ 
which can be chosen 
as an orthonormal basis for~$L_2(\cX,\nu_\cX)$.

Moreover, it can be shown that $\cC$ maps into 
the space of continuous functions.  
For this reason, 
we may identify 
the eigenfunctions $\{e_j\}_{j\in\bbN}$ 
with their continuous 
representatives. 
We let $\{\gamma_j\}_{j\in\bbN}$ denote 
the positive eigenvalues corresponding 
to $\{e_j\}_{j\in\bbN}$. 
By Mercer's theorem 
(see e.g.\ \cite{Mercer1909,Steinwart2012})
the covariance function $\varrho$ then 
admits the series representation 
\begin{equation}\label{eq:varrho-series}  
\textstyle 
\varrho(x,x') 
= \sum\limits_{j\in\bbN} \gamma_j e_j(x) e_j(x'),
\qquad
x,x'\in\cX,    
\end{equation} 
where the convergence of this series 
is absolute and uniform. 
In addition, we can express the action 
of the covariance operator by a series,
\[
\textstyle 
\cC w 
=
\sum\limits_{j\in\bbN} 
\gamma_j \scalar{ w, e_j}{L_2(\cX,\nu_\cX)} e_j,  
\qquad
w \in L_2(\cX,\nu_\cX),
\]
converging 
pointwise---i.e., for all $w \in L_2(\cX,\nu_\cX)$---and 
uniformly, i.e., 
in the operator norm.  
Finally, we note that 
square-integrability of the stochastic process 
implies that $\cC$ 
has a finite trace on $L_2(\cX,\nu_\cX)$, 
$\tr(\cC)=\sum_{j\in\bbN} \gamma_j < \infty$. 

\begin{example}\label{ex:cov-functions} 
Examples of   
covariance functions   
on a compact metric space  
$(\cX, d_\cX)$ 
are given by the Mat\'ern class, 
where 
\[
\varrho(x,x') 
:= 
\varrho_{\mathsf{M}}\bigl( d_\cX(x,x') \bigr) 
\quad 
\text{with} 
\quad 
\varrho_{\mathsf{M}}(r)
:= 
\tfrac{\sigma^2}{2^{\nu-1} \Gamma(\nu)} \,
(\kappa r)^\nu 
K_{\nu}(\kappa r), 
\quad 
r\geq 0. 
\]
More precisely, one may consider  
stochastic processes with Mat\'ern covariance functions, 
indexed by one of the following compact metric spaces:   
\begin{enumerate}[label={(\alph*)}, labelsep=5mm, topsep=4pt]
	\item\label{ex:cov-euclid}\hspace*{-2.3mm}  
	$\cX\subset\bbR^d$ is a  
	\emph{connected, compact Euclidean domain}  
	equipped with 
	the Euclidean metric  
	for all parameters 
	$\sigma,\nu,\kappa \in (0,\infty)$, 
	see  
	\cite{matern1960}; 
	\item\label{ex:cov-sphere}\hspace*{-3mm}     
	$\cX:= \bbS^{d}=\left\{x \in \bbR^{d+1} :
	\norm{x}{\bbR^{d+1}}=1\right\}$ 
	is the \emph{$d$-sphere} equipped 
	with the great circle distance 
	$d_{\bbS^d}(x,x') 
	:= 
	\arccos\bigl( \scalar{x,x'}{\bbR^{d+1}} \bigr)$  
	for $\sigma,\kappa \in (0,\infty)$ and 
	$\nu\in(0, 1/2]$, 
	see  
	\cite[][Section~4.5, Example~2]{gneiting2013}; 
	\item\label{ex:cov-manifold}\hspace*{-2.4mm}    
	$\cX \subset \bbR^D$ 
	is a \emph{$d$-dimensional connected, compact    
		manifold} (e.g., the $d$-sphere~$\bbS^d$), 
	embedded in $\bbR^D$ for some $D>d$   
	and equipped 
	with the Euclidean metric 
	on~$\bbR^{D}$  
	for any set of parameters 
	$\sigma,\nu,\kappa \in (0,\infty)$, 
	see e.g.\  
	\cite{Guinness2016}; 
	\item\label{ex:cov-graph}\hspace*{-2.4mm}   
	$\cX$ is a \emph{graph with Euclidean edges} 
	equipped 
	with the resistance metric for   
	$\sigma,\kappa \in (0,\infty)$ and 
	$\nu\in(0, 1/2]$, see  
	\cite[][Definition~1, Section~2.3 and Table~1]{Anderes2017}. 
\end{enumerate} 
We point out that, for $\nu \in (1/2,\infty)$, 
the function $(x,x') \mapsto \varrho_{\mathsf{M}}\bigl( d_\cX(x,x') \bigr)$ 
in~\ref{ex:cov-sphere} and~\ref{ex:cov-graph}  
is not (strictly) positive definite 
and, thus, not a valid covariance function 
for our setting. 
We furthermore emphasize that the Mat\'ern covariance families    
in~\ref{ex:cov-euclid} and~\ref{ex:cov-sphere} 
are stationary on~$\bbR^d$ and isotropic on $\bbS^d$, respectively, 
but we do not require $\varrho$ to have 
these properties. 
\end{example}

Since the kriging predictor in \eqref{eq:kriging_equation}
only depends on the mean value function and the covariance function 
of the process $\GP$, it is identical to 
the kriging predictor for a \emph{Gaussian} process 
with the same first two moments. 
For ease of presentation, 
we therefore from now on assume that 
$\GP$ is a 
Gaussian process on $(\cX,d_\cX)$ 
with mean value function $m\in L_2(\cX,\nu_\cX)$, 
continuous, (strictly) positive definite covariance 
function $\varrho$ 
and corresponding covariance operator~$\cC$. 
Note, however, that all our results 
extend to the case of non-Gaussian processes, 
as their proofs rely only on 
the first two statistical moments.  
We write 
$\mu = \normal(m,\cC)$
for the Gaussian measure on 
the Hilbert space $L_2(\cX,\nu_\cX)$ 
induced by the process $\GP$, i.e., 
for every 
Borel set 
$A \in \cB(L_2(\cX,\nu_\cX))$ we have 
\[
\mu(A)
= 
\bbP(\{\omega\in\Omega : \GP(\,\cdot\,, \omega)\in A\}).
\]   
The operator $\pE[\,\cdot\,]$ will 
denote the expectation operator under $\mu$, 
i.e., 
for 
an $L_2(\cX,\nu_\cX)$-valued random variable $Y$ 
with distribution $\mu$ 
and a Borel measurable mapping 
$\varphi\from L_2(\cX,\nu_\cX) \to \bbR$, 
the expected values 
$\pE[Y]\in L_2(\cX,\nu_\cX)$ and $\pE[\varphi(Y)]\in\bbR$ 
are the Bochner integrals 
\[
\pE[Y]
=
\int_{L_2(\cX,\nu_\cX)} y \, \rd\mu(y)
\quad\;\; 
\text{and} 
\quad\;\; 
\pE[\varphi(Y)]
=
\int_{L_2(\cX,\nu_\cX)} \varphi(y) \, \rd\mu(y),  
\] 
cf.~\cite[][Corollary~5.1]{Kukush2019}. 
Furthermore, we use the following notation 
for the real-valued 
variance and covariance operators 
with respect to $\mu$: 
If $\varphi, \varphi'\from L_2(\cX,\nu_\cX) \to \bbR$ 
are Borel measurable and 
$g:=\varphi(Y)$, $g':=\varphi'(Y)$, then
\[
\pV[g]:=\pE\bigl[ (g-\pE[g])^2 \bigr], 
\qquad 
\Cov[g, g'] 
:= 
\pE\bigl[ (g-\pE[g]) (g'-\pE[g']) \bigr]. 
\]

To present the theoretical setting 
of optimal linear prediction (kriging) 
as well as the necessary notation, 
we proceed in two steps:  
We first consider the centered case 
with $m=0$, 
and then extend it 
to the general case.

%====================================================================
\subsection{Kriging assuming zero mean} 
%====================================================================

Let $\GP^0\from\cX\times\Omega\to\bbR$ be 
a \emph{centered} Gaussian process. 
Assuming that it has a continuous covariance function~$\varrho$, 
we may identify $\GP^0 \from \cX \to L_2(\Omega,\bbP)$ 
with its continuous representative.  
In particular, for each $x\in\cX$, 
the real-valued random variable $\GP^0(x)$ 
is a well-defined element of $L_2(\Omega,\bbP)$.  
Consider the vector space $\cZ^0\subset L_2(\Omega,\bbP)$ 
of finite linear combinations 
of such random variables,  
\begin{equation}\label{eq:def:cZ0}
\textstyle
\cZ^0 := 
\biggl\{ 
\sum\limits_{j=1}^K \alpha_j \GP^0(x_j)
: 
K\in\bbN, 
\
\alpha_1,\ldots, \alpha_K \in\bbR, 
\  
x_1,\ldots, x_K \in \cX
\biggr\}. 
\end{equation} 
We then define the (Gaussian) Hilbert space $\cH^0$ 
(cf.~\cite{Janson1997}) 
as the 
closure of $\cZ^0$ 
with respect to the norm $\norm{\,\cdot\,}{\cH^0}$ 
induced by the $L_2(\Omega,\bbP)$ inner product,  
\begin{equation}\label{eq:def:cH0}
\begin{split}
	\textstyle 
	\biggl( 
	\sum\limits_{i=1}^{K}  \alpha_i  \GP^0(x_i), 
	\sum\limits_{j=1}^{K'} \alpha_j' \GP^0(x_j')
	\biggr)_{\cH^0}
	\textstyle 
	:= 
	\sum\limits_{i=1}^{K}  
	\sum\limits_{j=1}^{K'} \alpha_i \alpha_j' 
	\pE\bigl[ \GP^0(x_i) \GP^0(x_j') \bigr], 
	\\
	\cH^0 
	:= 
	\bigl\{ 
	g \in L_2(\Omega,\bbP)  
	\, \bigl| \, 
	\exists 
	\{g_j\}_{j\in\bbN}\subset\cZ^0  
	:
	\lim\nolimits_{j\to\infty}\norm{g-g_j}{L_2(\Omega,\bbP)}
	=0
	\bigr\}. 
\end{split}
\end{equation} 
Continuity of the covariance kernel $\varrho$ 
on $\cX\times\cX$ implies separability 
of the Hilbert space   
$\cH^0$, 
see 
\cite[][Theorem~32]{Berlinet2004} and  
\cite[][Theorem~2C]{Parzen1959}.

By definition, see e.g.\ \cite[][Section~1.2]{stein99}, 
the kriging predictor~$\hn^0$ of $h^0\in\cH^0$ 
based on a set of observations 
$\bigl\{ y_{n1}^0, \ldots, y_{nn}^0 \bigr\} \subset \cH^0$ 
is~the~best linear predictor in $\cH^0$ or, 
in other words,  
the $\cH^0$-orthogonal projection of~$h^0$ 
onto the linear space 
generated by $y_{n1}^0, \ldots, y_{nn}^0$. 
By recalling the inner product on $\cH^0$ 
from \eqref{eq:def:cH0}, 
the kriging predictor $\hn^0$ is thus the unique element 
in the finite-dimensional subspace 
$\cH_n^0 := 
\operatorname{span}\bigl\{ y_{n1}^0, \ldots, y_{nn}^0 \bigr\}
\subset\cH^0$ 
satisfying 
\begin{equation}\label{eq:def:h0n}
\hn^0 \in \cH_n^0: 
\quad 
\left( \hn^0 - h^0, g_n^0 \right)_{\cH^0} 
= 
\pE\bigl[ \bigl( \hn^0 - h^0 \bigr) g_n^0 \bigr] 
=
0
\quad 
\forall g_n^0 \in \cH_n^0. 
\end{equation} 
Consequently, $\hn^0$  
is the $\cH^0$-best approximation 
of $h^0$ in $\cH_n^0$, i.e., 
\[
\| \hn^0 - h^0 \|_{\cH^0} 
= 
\inf\nolimits_{g_n^0 \in \cH_n^0} 
\| g_n^0 - h^0 \|_{\cH^0}. 
\]

%====================================================================
\subsection{Kriging with general mean}
%====================================================================

Let us next consider the case  
that the Gaussian process~$\GP$ 
has a general mean value function 
$m\from\cX\to\bbR$  
which, for now, we  
assume to be continuous. 
The analytical framework  
for kriging  
then needs to be adjusted, since 
the space of possible predictors 
has to contain 
functions of the form \eqref{eq:kriging_equation}  
including constants.

Evidently, 
every linear combination 
$h = \sum_{j=1}^K \alpha_j \GP(x_j)$ 
has a representation $h = c + h^0$ 
with $c\in\bbR$ and $h^0 \in\cZ^0\subset\cH^0$, 
where the vector spaces 
$\cZ^0,\cH^0\subset L_2(\Omega,\bbP)$ 
are generated by linear combinations 
of the centered process $\GP^0 := \GP - m$ 
as in~\eqref{eq:def:cZ0} and \eqref{eq:def:cH0}; 
namely, 
\begin{equation}\label{eq:h-decomp} 
\textstyle 
h 
= 
\sum\limits_{j=1}^K \alpha_j \GP(x_j)
= 
\sum\limits_{j=1}^K \alpha_j m(x_j) 
+ 
\sum\limits_{j=1}^K \alpha_j \GP^0(x_j)  
=: 
c + h^0, 
\end{equation} 
or, more generally, 
$h = \pE[h] + (h - \pE[h])$. 
Furthermore, note that 
zero is the only constant 
contained in $\cZ^0, \cH^0 \subset L_2(\Omega,\bbP)$. 
This follows from the fact that 
elements in $\cZ^0$ are linear combinations 
of the process $\GP^0$ at locations in~$\cX$, 
see~\eqref{eq:def:cZ0}.  
A constant $c\neq 0$ in $\cZ^0$ 
would thus imply that the corresponding linear combination 
$c=\sum_{j=1}^K \alpha_j \GP^0(x_j)$ 
has zero variance, 
which contradicts the (strict) positive 
definiteness of the covariance function~$\varrho$. 
For this reason, the decomposition in \eqref{eq:h-decomp} 
is unique. 
This motivates 
to define  
the Hilbert space   
containing all possible observations 
and predictors 
for a general second order structure $(m,\varrho)$ 
as the (internal) direct sum 
of vector spaces, given by 
\begin{equation}\label{eq:def:cH}
\cH 
:= 
\bbR\oplus \cH^0 
= 
\bigl\{ 
h \in L_2(\Omega,\bbP) 
: 
\exists c\in\bbR, \, 
\exists h^0 \in\cH^0 
\text{ with }
h = c + h^0 
\bigr\},  
\end{equation}
which is equipped with the graph norm,  
\begin{equation}\label{eq:cH-norm}
\|h\|_{\cH}^2 = |c|^2 + \|h^0\|_{\cH^0}^2
\quad 
\text{if} 
\quad 
h = c + h^0 
\in \bbR\oplus\cH^0 = \cH. 
\end{equation}
Note that, similarly as for $\cH^0$, 
the inner product on $\cH$ equals  
the inner product on $L_2(\Omega,\bbP)$:
\[
\scalar{g,h}{\cH} 
= 
\scalar{\pE[g], \pE[h]}{\bbR} 
+ 
\scalar{g - \pE[g], h - \pE[h]}{\cH^0}
=
\pE[g] \pE[h] 
+ 
\Cov[g, h]
=
\pE[g h]. 
\]

Now suppose that we want to predict $h \in \cH$ 
given a set of observations 
\[
y_{nj}=c_{nj} + y_{nj}^0 \in \cH, 
\quad 
\text{where} 
\quad 
c_{nj}\in\bbR, 
\quad 
y_{nj}^0\in\cH^0, 
\qquad  
j\in\{ 1,\ldots,n\}.
\] 
The kriging predictor of $h= c + h^0 
\in \bbR\oplus\cH^0=\cH$ 
based on the observations $\{ y_{n1},\ldots,y_{nn} \}$ 
is then 
$\hn = c + \hn^0$, 
where $\hn^0$ is the kriging predictor of~$h^0$ 
based on the centered 
observations 
$\bigl\{ y_{n1}^0, \ldots, y_{nn}^0 \bigr\} \subset\cH^0$, 
as defined in \eqref{eq:def:h0n}. 
The definition 
of the norm on $\cH$ 
in \eqref{eq:cH-norm}  
readily implies that 
\[  
\| \hn - h \|_{\cH}^2 
= 
|c-c|^2 + \| \hn^0 - h^0 \|_{\cH^0}^2 
= 
0 + \inf_{g_n^0 \in \cH_n^0} \| g_n^0 - h^0 \|_{\cH^0}^2.  
\]
Hence, if we, for 
$y_{n1}^0, \ldots, y_{nn}^0 \in \cH^0$, 
define the subspace 
$\cH_n \subset \cH$ by  
\begin{equation}\label{eq:def:cHn}
\cH_n 
:= 
\bbR \oplus \cH^0_n, 
\quad\text{where}\quad
\cH^0_n
:= 
\operatorname{span}
\bigl\{ y_{n1}^0, \ldots, y_{nn}^0  \bigr\} 
\subset 
\cH^0, 
\end{equation} 
we have that in either case 
(centered and non-centered) 
the kriging predictor of $h\in\cH$   
based on the observations  
$\bigl\{ y_{n1}=c_{n1}+y_{n1}^0, \ldots, y_{nn}=c_{nn}+y_{nn}^0 \bigr\}$   
is given by the $\cH$-orthogonal projection 
of $h$ onto~$\cH_n$, i.e.,  
\begin{equation}\label{eq:def:hn}
\begin{split}
	\hn\in\cH_n :& 
	\quad 
	\scalar{ \hn - h, g_n}{\cH} 
	= 
	\pE\left[ ( \hn - h ) g_n \right] 
	= 0
	\quad 
	\forall 
	g_n \in \cH_n, 
	\\
	\hn\in\cH_n :& 
	\qquad\;  
	\| \hn - h \|_{\cH}
	= 
	\inf_{g_n \in \mathcal{H}_n} \| g_n - h \|_{\cH}. 
\end{split} 
\end{equation}
For this reason, for every $h\in\cH$, 
the kriging predictor 
$\hn$ is fully determined by the subspace~$\cH_n$ 
and we also call~$\hn$ the  
kriging predictor (or 
best linear predictor)   
\emph{based on~$\cH_n$}  
(instead of 
\emph{based on the set of observations 
$\{y_{n1}, \ldots, y_{nn} \}$}). 

Finally, since 
the definitions 
\eqref{eq:def:cH}, \eqref{eq:def:cHn} and  
\eqref{eq:def:hn} of the spaces $\cH, \cH_n$ and 
the kriging predictor $\hn$ 
are meaningful even if 
the mean value function is not continuous, 
hereafter we only require 
that $m\in L_2(\cX,\nu_\cX)$.  
Note, however,  
that then  
the point evaluation $\GP(x^*)$, $x^*\in\cX$, 
might not be an element of~$\cH$.

%====================================================================
\subsection{Problem formulation}
%====================================================================

We assume without loss of generality 
that the centered observations $y_{n1}^0,\ldots, y_{nn}^0$
are linearly independent in $\cH^0$  
so that 
in~\eqref{eq:def:cHn} we have $\dim( \cH_n^0 ) = n$. 
Furthermore, we suppose that the 
family of 
subspaces~$\{\cH_n \}_{n\in\bbN}$  
generated by the observations,   
see \eqref{eq:def:cHn}, 
is dense in~$\cH$. 
More specifically,   
we require that, 
for any $h\in\cH$, 
the corresponding kriging predictors  
$\{ \hn \}_{n\in\bbN}$ 
defined via \eqref{eq:def:hn} 
are consistent in the sense that 
\begin{equation}\label{eq:ass:Hn-dense}
\lim_{n\to\infty} 
\pE\bigl[ (\hn - h)^2 \bigr]
= 
\lim_{n\to\infty} 
\| \hn - h \|_{\cH}^2   
= 0. 
\end{equation}
For future reference, we introduce the 
set $\cS_{\mathrm{adm}}^\mu$ 
which contains all admissible sequences  
$\{\cH_n\}_{n\in\bbN}$ 
of subspaces of $\cH$ 
generated by observations which 
provide $\mu$-consistent kriging, 
\begin{equation}\label{eq:def:S-adm} 
\begin{split} 
	\cS^\mu_{\mathrm{adm}}  
	:= 
	\bigl\{ \{\cH_n\}_{n\in\bbN} \, \bigl| \, 
	&\forall n \in \bbN:  
	\cH_n \text{ is as in \eqref{eq:def:cHn} with}\,
	\dim(\cH_n^0)=n,  
	\\
	&\forall h\in\cH:  
	\{\hn\}_{n\in\bbN}   
	\text{ as in \eqref{eq:def:hn}		
		satisfy \eqref{eq:ass:Hn-dense}} \bigr\}.  
\end{split} 
\end{equation} 
Since $(\cX, d_\cX)$ is connected and compact,   
we have 
$\cS^\mu_{\mathrm{adm}}\neq\emptyset$. 
Note that we do 
\emph{not assume nestedness} 
of $\{\cH_n\}_{n\in\bbN}$. 
Therefore, we cover situations when  
the observations are not part of a sequence and 
$\{ y_{n1},\ldots,y_{nn} \} \not\subset 
\{ y_{n+1,1},\ldots,y_{n+1,n+1} \}$. 

\begin{example}\label{ex:admissible} 
Suppose that 
$m$ and $\varrho$ are continuous 
and that $\{x_j\}_{j\in\bbN}$ is a sequence in 
$(\cX,d_\cX)$ which  
\emph{accumulates at $x^*\in\cX$}, i.e., 
there exists a subsequence 
$\{\bar{x}_k\}_{k\in\bbN} \subseteq \{ x_j \}_{j\in\bbN}$ 
such that 
$\lim_{k\to\infty} d_\cX( \bar{x}_{k}, x^* ) =0$.
Assume further that 
$\cH_n \supseteq \bbR \oplus \operatorname{span}\bigl\{\GP^0(x_j):j\leq n \bigr\}$ 
for all $n$. 
Then, the kriging predictors $\{\hn\}_{n\in\bbN}$     
for $h := \GP(x^*)$ are consistent:  
For $n$ sufficiently large such that 
$\{ x_j \}_{j=1}^n \cap \{\bar{x}_k\}_{k\in\bbN} 
= \{\bar{x}_k\}_{k=1}^{k^*_n}$ 
is not empty (i.e., $k^*_n \in\bbN$), 
we have 
\begin{align*}
	\pE\bigl[ (\hn - h)^2 \bigr] 
	&\leq \textstyle
	\inf\limits_{\alpha_0,\alpha_1, \ldots, \alpha_n\in\bbR}
	\pE\Bigl[ \bigl( 
	\GP(x^*) - \alpha_0 - \sum\nolimits_{j=1}^n \alpha_j \GP^0(x_j) 
	\bigr)^2 \Bigr] 		
	\\
	&\textstyle 
	= 
	\inf\limits_{\alpha_1, \ldots, \alpha_n\in\bbR}
	\pE\Bigl[ \bigl( 
	\GP^0(x^*) - \sum\nolimits_{j=1}^n \alpha_j \GP^0(x_j) 
	\bigr)^2 \Bigr]
	\leq 
	\pE\Bigl[ \bigl( \GP^0(x^*) - \GP^0( \bar{x}_{k^*_n} ) \bigr)^2 \Bigr] 
	\\
	&= 
	\varrho( x^*,x^* ) 
	+ 
	\varrho( \bar{x}_{k^*_n}, \bar{x}_{k^*_n} ) 
	- 
	2\varrho( \bar{x}_{k^*_n}, x^* ) \rightarrow 0, 
	\quad\;\; \text{as} \quad\;\; 
	n\rightarrow \infty . 
\end{align*}
This shows \eqref{eq:ass:Hn-dense} for 
$h = \GP(x^*)$. Note 
that the kriging predictors based 
on the  
subspaces $\{\cH_n\}_{n\in\bbN}$ 
have to be consistent 
for every $h\in\cH$ so that 
the sequence 
is admissible,  
$\{\cH_n\}_{n\in\bbN} \in \cS^\mu_{\mathrm{adm}}$, 
see \eqref{eq:def:S-adm}. 
Assuming that every     
$\cH_n^0$ 
is generated by centered point observations 
$\GP^0(x_1), \GP^0(x_2), \ldots$,   
the above argument shows that, for any 
$h=\sum_{\ell=1}^L c_\ell \GP(x^*_{\ell})$ in $\bbR\oplus\cZ^0$, 
the kriging predictors $\{\hn\}_{n\in\bbN}$ based on 
$\{\cH_n\}_{n\in\bbN}$, $\cH_n:=\bbR\oplus\cH_n^0$,  
are consistent 
whenever 
the sequence of observation points  
$\{ x_j \}_{j\in\bbN}$
\emph{accumulates at any} $x^*\in\cX$. 
Since $\cZ^0$ is dense in the Hilbert space $\cH^0$, 
the same is true for 
every $h\in\cH=\bbR\oplus\cH^0$. 
\end{example}

Suppose that  
$\MU = \normal(\widetilde{m},\CC)$ 
is a second Gaussian measure on 
$L_2(\cX,\nu_\cX)$   
with mean value 
function~$\widetilde{m}\in L_2(\cX,\nu_\cX)$  
and trace-class covariance operator 
$\CC\from L_2(\cX,\nu_\cX) \to L_2(\cX,\nu_\cX)$. 
Let $\PE[\,\cdot\,]$, $\PV[\,\cdot\,]$ 
and $\COV[\,\cdot\,, \,\cdot\,]$ denote 
the real-valued expectation, variance, 
and covariance operators 
under $\MU$. 
We are now interested in the asymptotic behavior 
of the linear predictor based on~$\MU$. 
That is, what happens if, instead of the kriging predictor $\hn$, 
we use the linear predictor $\Hn$ 
which is the kriging predictor 
if $\MU$ was the correct model?

%====================================================================
\section{General results on compact metric spaces}\label{section:results} 
%====================================================================

We first generalize the results 
of \cite[][Section~3]{stein90} and 
\cite[][Chapter~4, Theorem~10]{stein99} 
to our setting of Gaussian processes 
on a compact metric space $(\cX, d_\cX)$. 
That is, uniformly asymptotically  
optimal linear prediction 
under the assumption that 
the two Gaussian measures 
$\mu$ and $\MU$ are equivalent 
on $L_2(\cX,\nu_\cX)$. 

\begin{theorem}\label{thm:kriging-equiv-measures} 
Let $\mu=\normal(m,\cC)$ 
and  
$\MU=\normal(\widetilde{m},\CC)$ be equivalent.  
Define 
$\cH^0$ and $\CH^0$ as in \eqref{eq:def:cH0} 
with respect to   
$\mu$ and $\MU$, 
respectively. 
Then, 
$\cH^0$ and $\CH^0$ are norm equivalent,  
$\cS^{\mu}_{\mathrm{adm}} = \cS^{\MU}_{\mathrm{adm}}$    
$($see \eqref{eq:def:S-adm}$)$, and  
for all  
$\{\cH_n\}_{n\in\bbN} \in \cS^\mu_{\mathrm{adm}}$ 
the following hold:   
\begin{gather}
	\lim_{n\to\infty}
	\sup_{h\in \cH_{-n}}
	\frac{
		\pE\bigl[ ( \Hn - h)^2 \bigr]
	}{
		\pE\bigl[ ( \hn - h)^2 \bigr]
	} 
	=
	\lim_{n\to\infty}
	\sup_{h\in \cH_{-n}}
	\frac{
		\PE\bigl[ ( \hn - h)^2 \bigr]
	}{
		\PE\bigl[ ( \Hn - h)^2 \bigr]
	} 
	= 1, 
	\label{eq:kriging_equiv-opt} 
	\\	
	\lim_{n\to\infty}
	\sup_{h\in \cH_{-n}}
	\left|
	\frac{
		\PE\bigl[ ( \hn - h)^2 \bigr]
	}{
		\pE\bigl[ ( \hn - h)^2 \bigr]
	} - 1 
	\right| 
	=
	\lim_{n\to\infty}
	\sup_{h\in \cH_{-n}}
	\left|
	\frac{
		\pE\bigl[ ( \Hn - h)^2 \bigr]
	}{
		\PE\bigl[ ( \Hn - h)^2 \bigr]
	} - 1 
	\right| 
	= 0. 
	\label{eq:kriging_equiv-a} 
\end{gather}
Here, $\hn, \Hn$ are 
the best linear predictors of $h\in\cH$ 
based on~$\cH_n$ and the 
measures~$\mu$ and $\MU$, respectively. 
The set 
$\cH_{-n}\subset \cH$ contains all 
elements 
with $\pE\bigl[ (\hn - h)^2\bigr] > 0$.
\end{theorem}

The norm equivalence of $\cH^0$ and $\widetilde{\cH}^0$ 
guarantees that the best 
linear predictors $\{\Hn\}_{n\in\bbN}$ of $h\in\cH$ based on 
$\{\cH_n \}_{n\in\bbN} \in\cS^\mu_{\mathrm{adm}}$ 
and $\MU$ are well-defined. 
Furthermore, 
it corresponds to 
equivalence of $\pV[\,\cdot\,]$ 
and $\PV[\,\cdot\,]$ on $\cH$, see Proposition~\ref{prop:AssI} below, 
so that 
in combination with the restriction $h\in\cH_{-n}$ 
it ensures that  
the case~$0/0$ 
in \eqref{eq:kriging_equiv-opt}, \eqref{eq:kriging_equiv-a} 
is evaded:  
For $h\in\cH_{-n}$,
we obtain      
$\PE\bigl[ ( \Hn - h)^2 \bigr] 
= 
\PV\bigl[  \Hn - h \bigr] 
\geq 
c\, 
\pV\bigl[  \Hn - h \bigr] 
\geq  
c \, \pV\bigl[  \hn - h \bigr] 
=  
c \,
\pE\bigl[ ( \hn - h)^2 \bigr] 
>0$,  
with $c\in(0,\infty)$ 
independent of $n$ and $h$. \pagebreak   

Equivalence of the Gaussian measures 
$\mu=\normal(m,\cC)$ and   
$\MU=\normal(\widetilde{m},\CC)$ 
implies that  
${m - \widetilde{m}}$ is in the 
Cameron--Martin space 
$H^*:=\cC^{1/2}(L_2(\cX,\nu_\cX))$ 
which is a Hilbert space 
with respect to  
$\scalar{\,\cdot\,, \,\cdot\,}{H^*}
:= \scalar{\cC^{-1}\,\cdot\,, \,\cdot\,}{L_2(\cX,\nu_\cX)}$,  
see Appendix~A 
in the Supplementary Material~\cite{kbsup}.
However, 
$m$ and~$\widetilde{m}$ are 
not necessarily 
both elements of $H^*$. 
Thus, Theorem~\ref{thm:kriging-equiv-measures} 
generalizes \cite[][Chapter~4, Theorem~10]{stein99}, 
where $m=0$ is assumed, 
even on Euclidean domains. 

These results for equivalent measures $\mu$ and $\MU$ also apply  
when considering the variances of the prediction errors. 
This is subject of the next corollary.

\begin{corollary}\label{cor:kriging-equiv-measures-var} 
The statements of Theorem~\ref{thm:kriging-equiv-measures} 
remain true if we replace 
each second moment 
in \eqref{eq:kriging_equiv-opt} and \eqref{eq:kriging_equiv-a}
by the corresponding variance.
That is, 
for all $\{\cH_n\}_{n\in\bbN}\in\cS^\mu_{\mathrm{adm}}$, 
the following hold: 	
\begin{gather}
	\lim_{n\to\infty}
	\sup_{h\in \cH_{-n}} 
	\frac{
		\pV\bigl[ \Hn - h \bigr]
	}{
		\pV\bigl[ \hn - h \bigr]
	}  
	=
	\lim_{n\to\infty}
	\sup_{h\in \cH_{-n}}
	\frac{
		\PV\bigl[ \hn - h \bigr]
	}{
		\PV\bigl[ \Hn - h \bigr]
	} 
	= 1, 
	\label{eq:kriging_equiv_var-opt} 
	\\ 
	\lim_{n\to\infty}
	\sup_{h\in \cH_{-n}}
	\left|
	\frac{
		\PV\bigl[ \hn - h \bigr]
	}{
		\pV\bigl[ \hn - h \bigr]
	} - 1 
	\right| 
	=
	\lim_{n\to\infty}
	\sup_{h\in \cH_{-n}}
	\left|
	\frac{
		\pV\bigl[ \Hn - h \bigr]
	}{
		\PV\bigl[ \Hn - h \bigr]
	} - 1 
	\right| 
	= 0. 
	\label{eq:kriging_equiv_var-a}
\end{gather}
\end{corollary}

Theorem~\ref{thm:kriging-equiv-measures} 
shows that equivalence of $\mu$ and $\MU$   
is \emph{sufficient} for 
uniformly asymptotically optimal linear prediction. 
The following (less restrictive) assumptions will subsequently 
be shown to be \emph{necessary and sufficient}.

\begin{assumption}\label{ass:kriging}  
Let $\varrho,\widetilde{\varrho}\from\cX\times\cX \to \bbR$ 
be two continuous,   
(strictly) positive definite covariance functions 
with 
corresponding covariance operators 
$\cC, \CC$, defined on $L_2(\cX,\nu_\cX)$  
via \eqref{eq:kernel-C}. 
Assume that $\cC, \CC$
and  
$m,\widetilde{m} \in L_2(\cX,\nu_\cX)$ 
are such that: 
\begin{enumerate}[label=\Roman*., labelsep=5mm, topsep=4pt]
	\item\hspace*{-1.5mm}  
	The Cameron--Martin spaces 
	$H^* = \cC^{1/2}(L_2(\cX,\nu_\cX))$  
	and 
	$\widetilde{H}^* =
	\CC^{1/2}(L_2(\cX,\nu_\cX))$ 
	are norm equivalent Hilbert spaces. 
	\item\hspace*{-2mm} 
	The difference between the  
	mean value functions 
	$m,\widetilde{m} \in L_2(\cX,\nu_\cX)$  
	is an element of the 
	Cameron--Martin space, i.e.,  
	$m-\widetilde{m} \in H^*$. 
	\item\hspace*{-3mm}  
	There exists 
	a positive real number $a\in(0,\infty)$ 
	such that the operator 
	\begin{equation}\label{eq:ass:Ta}
		T_a \from L_2(\cX,\nu_\cX) \to L_2(\cX,\nu_\cX), 
		\qquad 
		T_a := \cC^{-1/2} \CC \cC^{-1/2} - a \cI  
	\end{equation}
	is compact. 
	Here and below, 
	$\cI$ denotes the identity on $L_2(\cX,\nu_\cX)$.  
\end{enumerate}  
\end{assumption}

\begin{remark}\label{rem:comparison-with-FH} 
We briefly comment on the similarities 
and differences between Assumptions~\ref{ass:kriging}.I--III
and the set of assumptions 
in the Feldman--H\'ajek theorem 
(Theorem~A.1 
for $E=L_2(\cX,\nu_\cX)$, 
see Appendix~A 
in the Supplementary Material \cite{kbsup}) 
that are necessary and sufficient 
for equivalence of two Gaussian measures 
$\mu=\normal(m, \cC)$ and 
$\MU=\normal(\widetilde{m},\CC)$. 
Note that all assumptions are identical 
except for the third. For equivalence of 
the measures $\mu$ and~$\MU$, 
the operator $T_1 = \cC^{-1/2} \CC \cC^{-1/2} - \cI$ 
has to be Hilbert--Schmidt on $L_2(\cX,\nu_\cX)$. 
Since every Hilbert--Schmidt 
operator is compact, 
this in particular implies that 
Assumption~\ref{ass:kriging}.III  
holds for $a=1$. 
This shows the greater 
generality of our Assumptions~\ref{ass:kriging}.I--III 
compared to the assumption that 
the two Gaussian measures 
$\mu$ and $\MU$ are equivalent. 
\end{remark}

\begin{proposition}\label{prop:AssI}
Let $\mu=\normal(m,\cC)$,      
$\MU=\normal(\widetilde{m},\CC)$, and 
define $\cH^0, \CH^0$ as in \eqref{eq:def:cH0} 
with respect to the measures $\mu$ and $\MU$, respectively. 
The following are equivalent: 
\begin{enumerate}[label=(\roman*), labelsep=5mm, topsep=4pt]
	\item\label{prop:AssI-i}\hspace*{-2mm}    
	Assumption~\ref{ass:kriging}.I 
	is satisfied.  
	\item\label{prop:AssI-ii}\hspace*{-3.3mm}   
	The linear operator 
	$\CC^{1/2} \cC^{-1/2}\from L_2(\cX,\nu_\cX) \to L_2(\cX,\nu_\cX)$ 
	is an isomorphism, i.e., 
	it is bounded and has     
	a bounded inverse. 
	\item\label{prop:AssI-iii}\hspace*{-3.9mm}    
	The Hilbert spaces $\cH^0,\CH^0$ 
	are norm equivalent. In particular,  
	there exist $k_0, k_1\in(0,\infty)$ such that 
	$k_0 \pV[ h ] \leq \PV[ h ] \leq k_1 \pV[ h ]$
	holds for all $h\in\cH$,  
	with $\cH$ as in~\eqref{eq:def:cH}.  
	\item\label{prop:AssI-iv}\hspace*{-4mm}     
	There exist $0<k\leq K <\infty$ 
	such that, 
	for all 
	$\{\cH_n\}_{n\in\bbN}\in\cS^{\mu}_{\mathrm{adm}}$, 
	any of the following fractions  
	is bounded from below by $k>0$ 
	and from above by $K<\infty$, 
	uniformly with respect to  
	$n\in\bbN$ and $h\in\cH_{-n}$: 
	\begin{equation}\label{eq:prop:AssI-iv}
		\frac{\PV\bigl[ \hn - h \bigr]}{\pV\bigl[ \hn - h\bigr]}, 
		\quad 
		\frac{\pV\bigl[ \Hn - h \bigr]}{\PV\bigl[ \Hn - h\bigr]}, 
		\quad 
		\frac{\pV\bigl[ \Hn - h \bigr]}{\pV\bigl[ \hn - h\bigr]}, 
		\quad 
		\frac{\PV\bigl[ \hn - h \bigr]}{\PV\bigl[ \Hn - h\bigr]}.    
	\end{equation}
	Here, $\hn, \Hn$ are the best linear predictors of $h$ 
	based on~$\cH_n$ and $\mu$ resp.~$\MU$. 
\end{enumerate}  
\end{proposition}

Proposition~\ref{prop:AssI} 
elucidates the role 
of Assumption~\ref{ass:kriging}.I: 
As previously noted, 
the norm equivalence 
of the spaces $\cH^0$ and $\CH^0$ 
in~\ref{prop:AssI-iii} 
ensures that, for any $h\in\cH$, the 
best linear predictors 
$\{\Hn \}_{n\in\bbN}$   
based on  
$\{\cH_n\}_{n\in\bbN}\in\cS^\mu_{\mathrm{adm}}$ 
and the measure $\MU$ 
are well-defined. 
Furthermore, uniform 
boundedness of the fractions in \ref{prop:AssI-iv}         
guarantees that  
the sequence 
$\{\Hn\}_{n\in\bbN}$  
is $\mu$-consistent, 
\[
\lim_{n\to\infty}
\pV\bigl[ \Hn - h \bigr] 
\leq 
\sup_{\ell\in\bbN} \sup_{g\in\cH_{-\ell}} 
\frac{\pV\bigl[ \widetilde{g}_\ell - g \bigr]}{
\pV\bigl[ g_\ell - g \bigr]}
\lim_{n\to \infty} 
\pE\bigl[ (\hn - h)^2 \bigr] 	
= 
0, 
\] 
which clearly is necessary for 
asymptotically optimal linear prediction.

Including one more assumption, 
namely Assumption~\ref{ass:kriging}.III, 
yields necessary and
sufficient conditions for 
uniform asymptotic optimality 
of linear predictions, 
when the quality of the linear predictors is 
measured by the variance of the error. 
This result is formulated in 
the following theorem. 

\begin{theorem}\label{thm:kriging_general_var} 
Let $\mu=\normal(m,\cC)$ and   
$\MU=\normal(\widetilde{m},\CC)$. 
In addition, let $\hn,\Hn$ denote the best linear predictors of $h$ 
based on~$\cH_n$ and the measures~$\mu$ and~$\MU$, 
respectively. 
Then, any of the assertions,  
\begin{align} 
	\lim_{n\to\infty}
	\sup_{h\in \cH_{-n}}
	\frac{
		\pV\bigl[ \Hn - h \bigr]
	}{
		\pV\bigl[ \hn - h \bigr]
	} 
	&= 1,
	\label{eq:kriging_gen_var1} 
	\\
	\lim_{n\to\infty}
	\sup_{h\in \cH_{-n}}
	\frac{
		\PV\bigl[ \hn - h \bigr]
	}{
		\PV\bigl[ \Hn - h \bigr]
	} 
	&= 1, 
	\label{eq:kriging_gen_var2}
	\\ 
	\lim_{n\to\infty} 
	\sup_{h\in \cH_{-n}}
	\left|
	\frac{
		\PV\bigl[ \hn - h \bigr]}{
		\pV\bigl[ \hn - h \bigr]
	}
	-a \right| 
	&= 0, 
	\label{eq:kriging_gen_var3} 
	\\
	\lim_{n\to\infty} 
	\sup_{h\in \cH_{-n}}
	\left|
	\frac{
		\pV\bigl[ \Hn - h \bigr]}{
		\PV\bigl[ \Hn - h \bigr] 
	}
	- \frac{1}{a} \right| 
	&= 0,   
	\label{eq:kriging_gen_var4}
\end{align} 
holds for all 
$\{\cH_n\}_{n\in\bbN} \in \cS^\mu_{\mathrm{adm}}$
if and only if 
Assumptions~\ref{ass:kriging}.I 
and~\ref{ass:kriging}.III are fulfilled.
The constant $a\in(0,\infty)$  
in \eqref{eq:kriging_gen_var3} and \eqref{eq:kriging_gen_var4} 
is the same as 
that in~\eqref{eq:ass:Ta} 
of Assumption~\ref{ass:kriging}.III. 
\end{theorem}

\begin{remark} 
Theorem~\ref{thm:kriging_general_var} 
shows, in particular, that either all of the four asymptotic statements 
\eqref{eq:kriging_gen_var1}--\eqref{eq:kriging_gen_var4} 
hold simultaneously or none of them 
are true. 
\end{remark} 

Finally, 
when measuring 
the quality of the linear predictors 
in terms of the mean squared error, 
additionally the behavior of 
the difference $m-\widetilde{m}$ 
between the mean value functions 
matters,   
and all three 
of Assumptions~\ref{ass:kriging}.I--III 
are necessary and sufficient 
for uniform asymptotic optimality 
in this sense. 
This characterization is formulated 
in Theorem~\ref{thm:kriging_general} 
which is our main result.

\begin{theorem}\label{thm:kriging_general} 
Let $\mu=\normal(m,\cC)$ and   
$\MU=\normal(\widetilde{m},\CC)$. 
In addition, let $\hn,\Hn$ denote the best linear predictors of $h$ 
based on~$\cH_n$ and the measures~$\mu$ and~$\MU$, 
respectively. 
Then, any of the assertions,  
\begin{align}
	\lim_{n\to\infty}
	\sup_{h\in \cH_{-n}}
	\frac{
		\pE\bigl[ ( \Hn - h)^2 \bigr]
	}{
		\pE\bigl[ ( \hn - h)^2 \bigr]
	} 
	&=1,  
	\label{eq:kriging_gen1} 
	\\
	\lim_{n\to\infty}
	\sup_{h\in \cH_{-n}}
	\frac{
		\PE\bigl[ ( \hn - h)^2 \bigr]
	}{
		\PE\bigl[ ( \Hn - h)^2 \bigr]
	} 
	&= 1, 
	\label{eq:kriging_gen2} 
	\\	
	\lim_{n\to\infty}
	\sup_{h\in \cH_{-n}}
	\left|
	\frac{
		\PE\bigl[ ( \hn - h)^2 \bigr]
	}{
		\pE\bigl[ ( \hn - h)^2 \bigr]
	} - a 
	\right| 
	&=0, 
	\label{eq:kriging_gen3} 
	\\ 
	\lim_{n\to\infty}
	\sup_{h\in \cH_{-n}}
	\left|
	\frac{
		\pE\bigl[ ( \Hn - h)^2 \bigr]
	}{
		\PE\bigl[ ( \Hn - h)^2 \bigr]
	} - \frac{1}{a} 
	\right| 
	&= 0, 
	\label{eq:kriging_gen4} 
\end{align}
holds for all 
$\{\cH_n\}_{n\in\bbN} \in \cS^\mu_{\mathrm{adm}}$
if and only if 
Assumptions~\ref{ass:kriging}.I--III 
are satisfied.
The constant $a\in(0,\infty)$  
in \eqref{eq:kriging_gen3} and \eqref{eq:kriging_gen4} 
is the same as 
that in~\eqref{eq:ass:Ta} 
of Assumption~\ref{ass:kriging}.III. 
\end{theorem}

%====================================================================
\section{Proofs of the results}\label{section:proofs}
%====================================================================
%
Throughout this section,  
we abbreviate $L_2(\cX,\nu_\cX)$ by $L_2$, 
$\cL(L_2)$ is 
the space of bounded linear operators 
on $L_2$ and  
the subspaces $\cK(L_2)\subset\cL(L_2)$ as well as 
$\cL_2(L_2)\subset\cL(L_2)$  
contain all compact and  
Hilbert--Schmidt operators, respectively  
(see Appendix~A in the Supplementary Material~\cite{kbsup}). 

Recall that $\GP^0=\GP-m$, 
where $\GP$ is a Gaussian process 
on $(\cX,d_\cX)$ with corresponding  
Gaussian measure $\mu=\normal(m,\cC)$, and 
that $\{e_j\}_{j\in\bbN}$ 
is an orthonormal basis for~$L_2$ 
consisting of (the continuous representatives of)  
eigenfunctions of the covariance operator~$\cC$, 
with corresponding positive eigenvalues 
$\{\gamma_j\}_{j\in\bbN}$. 
In the next lemma a relation   
between the (dual of the)  
Cameron--Martin space for~$\mu$ 
and the Hilbert space $\cH^0$ in \eqref{eq:def:cH0} 
is established, 
similarly as in 
\cite[][Theorem~5D]{Parzen1959} 
or 
\cite[][Theorem~35]{Berlinet2004}. 
This relation will be crucial for proving~all~results. 

\begin{lemma}\label{lem:cH-ONB}
For each $j\in\bbN$, 
define $v_j := \tfrac{1}{\sqrt{\gamma_j}} e_j$ 
as well as the 
real-valued 
random variable 
$\gp_j := \scalar{\GP^0, v_j}{L_2}$. 
Then, the following hold: 
\begin{enumerate}[label=(\roman*), labelsep=5mm, topsep=4pt] 
	\item\label{lem:cH-ONB-i}\hspace*{-2.5mm}  
	$\{v_j\}_{j\in\bbN}$ 
	is an orthonormal basis for  
	$H  := \cC^{-1/2}(L_2)$, 
	with 
	$\scalar{v, v'}{H} 
	:= 
	\scalar{\cC v, v'}{L_2}$, 
	which is 
	the dual of the Cameron--Martin space
	$\dual{H} = \cC^{1/2}(L_2)$. 
	\item\label{lem:cH-ONB-ii}\hspace*{-3.4mm}   
	$\{\gp_j\}_{j\in\bbN}$  
	is an orthonormal basis 
	for the Hilbert space~$\cH^0$  
	equipped with the inner product 
	$\scalar{\,\cdot\,, \,\cdot\,}{\cH^0} 
	=
	\pE[\,\cdot\,  \,\cdot\,]
	=
	\Cov[\,\cdot\,, \,\cdot\,]$,  
	see \eqref{eq:def:cH0}.  
	\item\label{lem:cH-ONB-iii}\hspace*{-3.8mm}    
	The linear operator  
	\begin{equation}\label{eq:cJ-iso:def}     
		\cJ \from H \to \cH^0,
		\qquad 
		v \mapsto \scalar{\GP^0 , v}{L_2}, 
	\end{equation}
	is a well-defined isometric isomorphism and, 
	for all $v,v' \in H$, we have 
	\begin{equation}\label{eq:cJ-iso:scalar} 
		\scalar{v, v'}{H}  
		= 
		\scalar{\cC v, v'}{L_2}
		= 
		\Cov\bigl[ \cJ v, \cJ v' \bigr] 
		= 
		\pE\bigl[ \cJ v \cJ v' \bigr] 
		= 
		\scalar{\cJ v, \cJ v'}{\cH^0} 
		.
	\end{equation}
\end{enumerate}
\end{lemma} 

\begin{proof}
Since $L_2\subset H$ and   
$v_j=\cC^{-1/2}e_j$, 
claim~\ref{lem:cH-ONB-i}---orthonormality and the basis property of 
$\{v_j\}_{j\in\bbN}$ in $H$---follows directly 
from the corresponding properties of
$\{e_j\}_{j\in\bbN}$ in $L_2$. 

To prove the second assertion~\ref{lem:cH-ONB-ii}, 
we first note that  
$v_j  \from \cX\to\bbR$ 
is continuous for every $j\in\bbN$. 
Thus, by 
Lemma~B.3
(see Appendix~B in the Supplementary Material~\citep{kbsup})
we find that  
$\gp_j\in \cH^0$ 
for all $j\in\bbN$. 
Next, we prove that $\{\gp_j\}_{j\in\bbN}$ 
constitutes an orthonormal 
basis for $\cH^0$. 
For this, we need 
to show that $\scalar{\gp_i, \gp_j}{\cH^0}=\delta_{ij}$ 
(\emph{orthonormality})
and 
\[
\scalar{h, \gp_j}{\cH^0}=0 
\quad \forall j\in\bbN 
\quad 
\Rightarrow 
\quad 
h = 0 
\qquad 
\text{(\emph{basis poperty})}. 
\]
Due to the identities  
$\pE[\gp_i \gp_j] 
= 
\Cov[ \scalar{\GP^0, v_i}{L_2}, \scalar{\GP^0, v_j}{L_2}] 
= 
\scalar{\cC v_i, v_j}{L_2}
= 
\scalar{v_i, v_j}{H}$, 
ortho- normality follows from~\ref{lem:cH-ONB-i}.
Now let $h\in\cH^0$ be such that  
$\scalar{h, \gp_j}{\cH^0}=0$ vanishes 
for all~$j\in\bbN$.\pagebreak   

\noindent 
By Fubini's 	
theorem we then obtain that, 
for all $j\in\bbN$, 
\[ 
0 = 
\pE[h \scalar{\GP^0,e_j}{L_2}] 
= 
\int_\cX 
\pE[h \GP^0(x)] e_j(x) \, \rd\nu_\cX(x) 
= 
\scalar{\pE[h \GP^0(\,\cdot\,)], e_j}{L_2}. 
\] 
Since $\{e_j\}_{j\in\bbN}$ is an orthonormal 
basis for $L_2$  
and 
the mapping $\cX\ni x\mapsto \pE[h \GP^0(x)] \in\bbR$ 
is continuous, 
this shows that 
$\pE[h \GP^0(x)]=0$
for all $x\in\cX$, 
which implies 
(due to strict positive definiteness 
of~$\varrho$)
that $h\in\cH^0$ 
has to vanish. 
We conclude~\ref{lem:cH-ONB-ii}, 
$\{\gp_j\}_{j\in\bbN}$ 
is an orthonormal basis 
for $\cH^0$.  

It remains to prove~\ref{lem:cH-ONB-iii}. 
Clearly, $\cJ v_j = \gp_j$ for all $j\in\bbN$. 
Thus, the linear mapping $\cJ\from H \to \cH^0$ is  
well-defined and an 
isometry, since by~\ref{lem:cH-ONB-i} and~\ref{lem:cH-ONB-ii} 
$\{v_j\}_{j\in\bbN}$ and $\{\gp_j\}_{j\in\bbN}$ 
are orthonormal 
bases for $H$ and $\cH^0$, respectively. 
Furthermore, 
$\scalar{v, v'}{H} 
= 
\scalar{\cC v, v'}{L_2} 
=
\Cov[ \scalar{\GP^0, v}{L_2}, \scalar{\GP^0, v'}{L_2} ]
=
\Cov[ \cJ v, \cJ v' ]
=
\pE[ \cJ v \cJ v' ] 
=
\scalar{\cJ v, \cJ v'}{\cH^0}$ 
for all $v,v'\in H$, 
completing the proof of~\ref{lem:cH-ONB-iii}. 
\end{proof}

\begin{proof}[Proof of Proposition~\ref{prop:AssI}] 
We first show   
\ref{prop:AssI-i}$\,\Rightarrow\,$\ref{prop:AssI-ii}$\,\Rightarrow\,$\ref{prop:AssI-iii}$\,\Rightarrow\,$\ref{prop:AssI-i},  
followed by the proof of the equivalence  
\ref{prop:AssI-iii}$\,\Leftrightarrow\,$\ref{prop:AssI-iv}:   	

\ref{prop:AssI-i}$\,\Rightarrow\,$\ref{prop:AssI-ii}: 
Under Assumption~\ref{ass:kriging}.I, 
the norms on $H^*$ and $\widetilde{H}^*$ 
are equivalent, i.e., 
there are $c_0,c_1\in(0,\infty)$ 
such that 
$\norm{\cC^{-1/2} u}{L_2} 
\leq 
c_0 \norm{\CC^{-1/2} u}{L_2}$ 
and 
$\norm{\CC^{-1/2} u}{L_2} 
\leq c_1 \norm{\cC^{-1/2} u}{L_2}$ 
for all $u \in H^*=\widetilde{H}^*$. 
Thus, for any $w\in L_2$,  
$\CC^{1/2}w \in H^*$ 
with 
$\norm{\cC^{-1/2}\CC^{1/2} w}{L_2}
\leq 
c_0 \norm{w}{L_2}$ 
and, in addition, 
$\cC^{1/2}w \in \widetilde{H}^*$ with  
$\norm{\CC^{-1/2}\cC^{1/2} w}{L_2}
\leq 
c_1 \norm{w}{L_2}$. 
This shows 
that $\cC^{-1/2}\CC^{1/2}$ and, thus, also its 
adjoint $\CC^{1/2}\cC^{-1/2}$ are isomorphisms 
on $L_2$, 
and \ref{prop:AssI-ii} follows. 

\ref{prop:AssI-ii}$\,\Rightarrow\,$\ref{prop:AssI-iii}: 
Let the Hilbert space $H$ and the isometry 
$\cJ \from H \to \cH^0$ 
be defined as in Lemma~\ref{lem:cH-ONB}\ref{lem:cH-ONB-iii}. 
In addition, define 
$\widetilde{H}:=\CC^{-1/2}(L_2)$  
and the isometry 
$\widetilde{\cJ}\from \widetilde{H} \to \CH^0$  
in the obvious analogous way. 
If 
$c_0 := \norm{\CC^{1/2}\cC^{-1/2}}{\cL(L_2)}$ 
and  
$c_1 := \norm{\cC^{1/2}\CC^{-1/2}}{\cL(L_2)}$  
are finite, 
then $H$ and $\widetilde{H}$ are norm equivalent, 
i.e., 
$c_1^{-1} \norm{v}{H}
\leq \norm{v}{\widetilde{H}} 
\leq c_0 \norm{v}{H}$ 
holds for all $v\in H=\widetilde{H}$, 
and 
the inclusion mapping $\cI_{H\to\widetilde{H}}$ 
of $H$ in~$\widetilde{H}$
is continuous. 
Thus, 
we obtain 
$\| \widetilde{\cJ}\cI_{H\to\widetilde{H}}\cJ^{-1} h^0 \|_{\CH^0} 
= 
\| \cI_{H\to\widetilde{H}}\cJ^{-1} h^0 \|_{\widetilde{H}} 
\leq 
c_0 
\| \cJ^{-1} h^0 \|_{H} 
= 
c_0 
\| h^0 \|_{\cH^0}$ 
for every $h^0\in\cH^0$. 
Next, let $h^0\in\cH^0$ and set 
$v^h:=\cJ^{-1}h^0\in H$.  
Then, we observe the identities  
$\| \widetilde{\cJ}\cI_{H\to\widetilde{H}}\cJ^{-1} h^0\|_{\CH^0}^2 
= 
\| \widetilde{\cJ} v^h \|_{\CH^0}^2 
= 
\PV[ \scalar{\GP, v^h}{L_2} ] 
= 
\PV[ \scalar{\GP^0, v^h}{L_2} ]
= 
\PV[h^0]$, 
which combined with the above show that   
$\PV[h] = \PV[h^0] \leq c_0^2 \norm{h^0}{\cH^0}^2 = c_0^2 \, \pV[h]$ 
for all $h\in\cH$, where 
we set $h^0 := h-\pE[h]\in\cH^0$.  
Similarly, we derive   
${\| \cJ \cI_{\widetilde{H}\to H} \widetilde{\cJ}^{-1} \widetilde{h}^0\|_{\cH^0} 
	\leq c_1 \| \widetilde{h}^0\|_{\CH^0}}$ 
for all $\widetilde{h}^0\in\CH^0$, 
and we may change the roles of $H$ and $\widetilde{H}$ 
(respectively, of $\cH^0$ and $\CH^0$)
to conclude that also the relation   
$\pV[h]\leq c_1^2 \, \PV[h]$ holds 
for all $h\in\CH=\bbR\oplus\CH^0$. 

\ref{prop:AssI-iii}$\,\Rightarrow\,$\ref{prop:AssI-i}: 
We prove that the dual spaces, 
$H=\cC^{-1/2}(L_2)$, 
$\widetilde{H}=\CC^{-1/2}(L_2)$, 
are norm equivalent, which implies the result for 
$\dual{H}$ and~$\dual{\widetilde{H}}$. 
Norm equivalence of $\cH^0$ and $\CH^0$ 
implies continuity of  
the inclusion maps 
$\cI_{\cH^0\to\CH^0}, \cI_{\CH^0\to\cH^0}$ 
which, similarly as above, 
yields continuity of  
$\widetilde{\cJ}^{-1} \cI_{\cH^0\to\CH^0} \cJ \from H\to\widetilde{H}$ 
and of 
$\cJ^{-1} \cI_{\CH^0\to\cH^0} \widetilde{\cJ} \from \widetilde{H}\to H$. 
Thus, it follows that $\norm{v}{\widetilde{H}}\leq c_0\norm{v}{H}$ 
and 
$\norm{\widetilde{v}}{H}\leq c_1\norm{\widetilde{v}}{\widetilde{H}}$  
hold for all $v\in H$, $\widetilde{v}\in\widetilde{H}$ 
with some constants $c_0,c_1\in(0,\infty)$, since 
$\| \widetilde{\cJ}^{-1} \cI_{\cH^0\to\CH^0} \cJ v \|_{\widetilde{H}} 
= \norm{v}{\widetilde{H}}$ 
and  
$\| \cJ^{-1} \cI_{\CH^0\to\cH^0} \widetilde{\cJ} \widetilde{v} \|_{H} 
= 
\norm{ \widetilde{v} }{H}$. 

\ref{prop:AssI-iii}$\,\Rightarrow\,$\ref{prop:AssI-iv}: 
Suppose that  
$k_0 \pV[h] \leq \PV[h] \leq k_1 \pV[h]$
holds    
for every $h\in\cH=\CH$   
and 
let $\{\cH_n\}_{n\in\bbN} \in \cS^\mu_{\mathrm{adm}}$.  
Then, for every  
$n\in\bbN$ and all $h\in\cH_{-n}$, 	
$k_0 
\leq 
\frac{\PV[\hn-h]}{\pV[\hn-h]}
\leq k_1$ 
as well as  
$k_1^{-1}
\leq 
\frac{\pV[\Hn-h]}{\PV[\Hn-h]}
\leq k_0^{-1}$ 
readily follow.  
Subsequently,   
we find that 
\begin{equation}\label{eq:trick-var}
	1\leq 
	\frac{\pV\bigl[ \Hn - h\bigr]}{\pV\bigl[ \hn - h\bigr]} 
	= 
	\frac{\pV\bigl[ \Hn - h\bigr]}{\PV\bigl[ \Hn - h\bigr]}
	\frac{\PV\bigl[ \Hn - h\bigr]}{\PV\bigl[ \hn - h\bigr]}
	\frac{\PV\bigl[ \hn - h\bigr]}{\pV\bigl[ \hn - h\bigr]}
	\leq 
	k_0^{-1} k_1, 
\end{equation} 
and a similar trick shows that also 
$\frac{\PV[ \hn - h ]}{\PV[ \Hn - h ]} 
\in  
\bigl[ 1, k_1 k_0^{-1} \bigr]$ 
for all $n$ and $h$. 

\ref{prop:AssI-iv}$\,\Rightarrow\,$\ref{prop:AssI-iii}: 
We first show necessity of~\ref{prop:AssI-iii}
for uniform boundedness (from above and below) 
of the first two fractions in~\eqref{eq:prop:AssI-iv}. 
For this, let $h\in\cZ^0 \setminus\{0\}$.  
By positive definiteness of~$\varrho$, 
there exists $\phi\in\cZ^0$ so that $\{h,\phi\}$ are linearly independent. 
Define 
$\psi'_1 := \phi - \frac{\scalar{\phi,h}{\cH^0}}{\scalar{h,h}{\cH^0}} h$, 
$\psi_1 := \frac{1}{\norm{\psi'_1}{\cH^0}} \psi'_1$
and 
$\widetilde{\psi}'_1 := \phi - \frac{\scalar{\phi,h}{\CH^0}}{\scalar{h,h}{\CH^0}} h$, 
$\widetilde{\psi}_1 := \frac{1}{\norm{\widetilde{\psi}'_1}{\cH^0}} \widetilde{\psi}'_1$, 
and note that $h\in\cZ^0$ is orthogonal 
to $\psi_1\in\cZ^0$ in $\cH^0$ and 
to $\widetilde{\psi}_1\in\cZ^0$ in $\CH^0$. 
By separability of $\cH^0$ there exist 
sequences $\{\psi_j\}_{j\geq 2}$ and $\{\widetilde{\psi}_j\}_{j\geq 2}$ 
such that $\{\psi_j\}_{j\in\bbN}$ and $\{\widetilde{\psi}_j\}_{j\in\bbN}$ 
are orthonormal bases for $\cH^0$. 
For $n\in\bbN$, define the spaces 
$\cH_n^\star := \bbR\oplus\operatorname{span}\{\psi_1,\ldots,\psi_n\}$, 
$\cH_n^\divideontimes := 
\bbR\oplus\operatorname{span}\{\widetilde{\psi}_1,\ldots,\widetilde{\psi}_n\}$. 
Then, 
${\{\cH_n^\star\}_{n\in\bbN}, \{\cH_n^\divideontimes\}_{n\in\bbN} \in\cS^\mu_{\mathrm{adm}}}$ 
and, if  
$h_1$ denotes the best linear predictor 
of $h$ based on $\cH_1^\star$ and $\mu$, 
and $\widetilde{h}_1$ denotes the best linear predictor of $h$ 
based on $\cH_1^\divideontimes$ and $\MU$, 
then $h_1=\widetilde{h}_1=0$ follows. 
By boundedness of 
the first or second fraction in \eqref{eq:prop:AssI-iv} 
(with $n=1$),  
we have that 
$\frac{\PV[ h ]}{\pV[ h ] } 
=
\frac{\PV[ h_1 - h ]}{\pV[ h_1 - h ] } 
\in 
[k,K]$ or  
$\frac{\pV[ h ]}{\PV[ h ]} 
=
\frac{\pV\left[ \widetilde{h}_1 - h \right]}{\PV\left[ \widetilde{h}_1 - h \right]} 
\in 
[k,K]$. 
Thus, in both cases 
$k_0\pV[h]\leq\PV[h]\leq k_1\pV[h]$ holds, 
where $k_0 := \min\{k,K^{-1}\}$ and 
$k_1 := \max\{K,k^{-1}\}$. 
Since $h\in\cZ^0$ 
was arbitrary and since 
the constants $k,K\in(0,\infty)$ do not depend on $h$ 
(as the fractions in~\eqref{eq:prop:AssI-iv} are bounded \emph{uniformly}  
in $\{\cH_n\}_{n\in\bbN}$, $n$ and $h$),  
assertion \ref{prop:AssI-iii} follows by density 
of $\cZ^0$ in $\cH^0$ and in $\CH^0$. 

Assume next that the third fraction 
in \eqref{eq:prop:AssI-iv} is bounded, uniformly with respect to 
$\{\cH_n\}_{n\in\bbN}\in\cS^\mu_{\mathrm{adm}}$, 
$n\in\bbN$ and $h\in\cH_{-n}$. 
In either of the two cases, 
$\alpha_0=0$ or $\alpha_1=\infty$, 
where 
$\alpha_0:=\inf_{h\in\cH^0\cap\CH^0} \frac{\PV[h]}{\pV[h]}$,  
$\alpha_1:=\sup_{h\in\cH^0\cap\CH^0} \frac{\PV[h]}{\pV[h]}$, 
it follows as in \cite[][Proof of Theorem~5]{CLEVELAND71}
that there exist   
sequences $\bigl\{ h^{(\ell)} \bigr\}_{\ell\in\bbN}, 
\bigl\{\psi_1^{(\ell)}\bigr\}_{\ell\in\bbN} \subset\cH^0\cap\CH^0$, 
normalized in $\cH^0$, with 
$\tfrac{\pV\bigl[ \widetilde{h}^{(\ell)}_1 - h^{(\ell)} \bigr]}{\pV\bigl[ h^{(\ell)}_1 - h^{(\ell)} \bigr]} 
\geq \ell$ for all $\ell\in\bbN$, 
where $h^{(\ell)}_1, \widetilde{h}^{(\ell)}_1$ are the best linear predictors 
of $h^{(\ell)}$ based on $\mu$ resp.\ $\MU$ and 
$\cH_1^{(\ell)}:=\bbR\oplus\operatorname{span}\bigl\{ \psi_1^{(\ell)} \bigr\}$. 
By separability of~$\cH^0$, for each $\ell\in\bbN$, 
we may complement $\psi_1^{(\ell)}$ 
to an orthonormal basis 
$\bigl\{ \psi_j^{(\ell)} \bigr\}_{j\in\bbN}$ for~$\cH^0$. 
Thus,~for all $\ell\in\bbN$, 
$\bigl\{ \cH_n^{(\ell)} \bigr\}_{n\in\bbN} \in \cS^\mu_{\mathrm{adm}}$  
holds, 
where 
$\cH_n^{(\ell)} := \bbR\oplus\operatorname{span} \bigl\{\psi_1^{(\ell)},\ldots,\psi_n^{(\ell)} \bigr\}$ 
and 
\[
\sup_{\ell\in\bbN} \; \sup_{n\in\bbN} \, \sup_{h\in\cH_{-n}^{(\ell)}} \; 
\frac{\pV\bigl[ \widetilde{h}^{(\ell)}_n - h \bigr]}{\pV\bigl[ h^{(\ell)}_n - h \bigr]} 
\geq 
\sup_{\ell\in\bbN}  \;  
\frac{\pV\bigl[\widetilde{h}^{(\ell)}_1 - h^{(\ell)} \bigr]}{\pV\bigl[ h^{(\ell)}_1 - h^{(\ell)} \bigr]} 
= 
\infty, 
\]
contradicting uniform boundedness of the third 
fraction in \eqref{eq:prop:AssI-iv}. 
We therefore conclude that 
$\alpha_0,\alpha_1\in(0,\infty)$, 
$\cH^0 \cap \CH^0 = \cH^0 = \CH^0$ 
and \ref{prop:AssI-iii} follows. 

Finally, assuming uniform boundedness 
of the last fraction in \eqref{eq:prop:AssI-iv}, 
analogous arguments show that 
$\widetilde{\alpha}_0:=\inf_{h\in\cH^0\cap\CH^0} \frac{\pV[h]}{\PV[h]}$,  
$\widetilde{\alpha}_1:=\sup_{h\in\cH^0\cap\CH^0} \frac{\pV[h]}{\PV[h]}$ 
satisfy 
$\widetilde{\alpha}_0,\widetilde{\alpha}_1\in(0,\infty)$, 
again yielding \ref{prop:AssI-iii}. 
\end{proof} 

\begin{remark}\label{rem:COV} 
The arguments in the proof of Proposition~\ref{prop:AssI}
imply, in particular, that 
under Assumption~\ref{ass:kriging}.I  
we have, for all $v, v' \in H=\widetilde{H}$, that 
\begin{equation}\label{eq:COV} 
	\scalar{v, v'}{\widetilde{H}} 
	= 
	\scalar{\CC v, v'}{L_2} 
	= 
	\COV\bigl[ \scalar{\GP^0, v}{L_2}, \scalar{\GP^0, v'}{L_2} \bigr]
	=
	\COV\bigl[ \cJ v, \cJ v' \bigr]. 
\end{equation}
\end{remark} 	

\begin{lemma}\label{lem:S-mu}
Suppose Assumption~\ref{ass:kriging}.I 
is satisfied and 
let $\{\cH_n \}_{n\in\bbN}$ be a sequence 
of subspaces of $\cH$ such that, 
for all $n\in\bbN$,  
$\cH_n$ is of the form \eqref{eq:def:cHn}. 
Then, for every $h\in \cH$, 
the kriging predictors $\{\hn \}_{n\in\bbN}$ 
based on $\{\cH_n \}_{n\in\bbN}$  
and the measure $\mu$ are $\mu$-consistent if and only if 
the kriging predictors $\{\Hn \}_{n\in\bbN}$ 
based on $\{\cH_n \}_{n\in\bbN}$  
and $\MU$ are $\MU$-consistent. 
In particular, 
$\cS^{\mu}_{\mathrm{adm}}=\cS^{\MU}_{\mathrm{adm}}$. 
\end{lemma} 

\begin{proof}
Let $\{\cH_n\}_{n\in\bbN}\in\cS^\mu_{\mathrm{adm}}$.
By Proposition~\ref{prop:AssI}\ref{prop:AssI-i}$\,\Leftrightarrow\,$\ref{prop:AssI-iii},  
$\cH^0$ and $\CH^0$ are norm equivalent. 
Thus, 
$
\PE\bigl[ (\Hn - h)^2 \bigr] 
= 
\PV\bigl[ \Hn - h \bigr] 
\leq 
\PV\bigl[ \hn - h \bigr] 
\leq 
k_1  
\pV\bigl[ \hn - h \bigr] 
= 
k_1 
\pE\bigl[ (\hn - h)^2 \bigr] 
$
holds, for any $h\in\cH=\CH$,   
where $k_1\in(0,\infty)$ is independent 
of $n$ and $h$. 
This shows that   
$\cS^\mu_{\mathrm{adm}} \subseteq \cS^{\MU}_{\mathrm{adm}}$. 
Analogously, 
$\pE\bigl[ (\hn - h)^2 \bigr] 
\leq k_0^{-1}  
\PE\bigl[ (\Hn - h)^2 \bigr]$ 
follows for 
$\{\cH_n\}_{n\in\bbN}\in\cS^{\MU}_{\mathrm{adm}}$, 	
with $k_0\in(0,\infty)$ 
independent of $n$ and $h$,  
showing the reverse inclusion 
$\cS^{\MU}_{\mathrm{adm}} \subseteq \cS^{\mu}_{\mathrm{adm}}$. 
\end{proof} 

\begin{proposition}\label{prop:sufficiency} 
Let $\mu=\normal(m,\cC)$
and 
$\MU=\normal(\widetilde{m},\CC)$. 
Suppose that 
Assumptions~\ref{ass:kriging}.I 
and~\ref{ass:kriging}.III 
are satisfied  
and let $a\in(0,\infty)$ 
be the constant in \eqref{eq:ass:Ta}  
of Assumption~\ref{ass:kriging}.III.  
In addition, let $\hn,\Hn$ denote the best linear predictors of~$h$ 
based on~$\cH_n$ and the measures~$\mu$ and $\MU$, 
respectively. 
Then, \eqref{eq:kriging_gen_var1}--\eqref{eq:kriging_gen_var4}
hold for all 
$\{\cH_n\}_{n\in\bbN}\in\cS_{\mathrm{adm}}^\mu$. 
If, in addition,  
Assumption~\ref{ass:kriging}.II 
is fulfilled, then  
\eqref{eq:kriging_gen1}--\eqref{eq:kriging_gen4}
hold for all 
$\{\cH_n\}_{n\in\bbN}\in\cS_{\mathrm{adm}}^\mu$. 
\end{proposition} 

\begin{proof} 
Let $\{\cH_n\}_{n\in\bbN}\in\cS_{\mathrm{adm}}^\mu$. 
As shown in Proposition~C.2 
(see Appendix~C in the Supplementary Material~\cite{kbsup}), 
we can without loss of generality 
assume that $\mu$ has zero mean and that 
$\MU$ has mean $\widetilde{m} - m$. 
We first show that 
Assumptions~\ref{ass:kriging}.I 
and~\ref{ass:kriging}.III 
imply that \eqref{eq:kriging_gen_var3} and \eqref{eq:kriging_gen_var4} 
hold. 
To this end, let $n\in\bbN$ and  
recall that $\hn$ is the kriging predictor of~$h$ 
based on 
$\cH_n 
=
\bbR\oplus
\cH_n^0$ 
and $\mu$.  
We let 
$\bigl\{\psi_1^{(n)},\ldots,\psi_n^{(n)} \bigr\}$ 
be an $\cH^0$-orthonormal basis for $\cH^0_n$, i.e.,  
$\pE\bigl[ \psi_k^{(n)} \psi_\ell^{(n)} \bigr] = \delta_{k\ell}$. 
Since $\cH^0$ is a separable Hilbert space 
there exists a countable orthonormal 
basis of the orthogonal complement  
of $\cH_n^0$ in $\cH^0$, which will be denoted 
by $\bigl\{ \psi_k^{(n)} \bigr\}_{k>n}$.  
Then, by construction  
$\bigl\{ \psi_k^{(n)} \bigr\}_{k\in\bbN}$ 
is an orthonormal 
basis for $\cH^0$.
We identify $\psi_k^{(n)}\in\cH^0$ with 
${v_k^{(n)} := \cJ^{-1} \psi_k^{(n)} \in H}$, 
where $\cJ \from H \to \cH^0$ is the 
isometric isomorphism in \eqref{eq:cJ-iso:def} 
from Lemma~\ref{lem:cH-ONB}\ref{lem:cH-ONB-iii}. 
Due to \eqref{eq:cJ-iso:scalar}, 
$\bigl\{ v_k^{(n)} \bigr\}_{k\in\bbN}$ 
is then an orthonormal basis for $H=\cC^{-1/2}(L_2)$. 
Furthermore, we note that, for every $h\in \cH_{-n}$, the vector 
$\hn - h \in \cH^0$ can be written as a linear combination 
of $\bigl\{ \psi_k^{(n)} \bigr\}_{k>n}$, i.e.,  
$\hn - h = \sum_{k=n+1}^\infty c_k^{(n)} \psi_k^{(n)}$ 
with $\sum_{k=n+1}^\infty | c_k^{(n)} |^2 < \infty$.  

We recall the identities 
in \eqref{eq:cJ-iso:scalar} and \eqref{eq:COV}
from Lemma~\ref{lem:cH-ONB}\ref{lem:cH-ONB-iii} 
and Remark~\ref{rem:COV} 
and rewrite the term 
$\text{(A)}:= 
\bigl| 
\PV[ \hn - h ] 
- 
a \pV[ \hn - h ]  
\bigr|$ 
as follows, 
\begin{align*} 
	\text{(A)} 
	&\textstyle 
	= 
	\biggl| 
	\sum\limits_{k,\ell=n+1}^\infty
	c_k^{(n)} c_\ell^{(n)} \left( 
	\COV\bigl[ \psi_k^{(n)}, \psi_\ell^{(n)} \bigr] 
	- 
	a \Cov\bigl[ \psi_k^{(n)}, \psi_\ell^{(n)} \bigr]  \right) 
	\biggr|	 
	\\
	&\textstyle 
	= 
	\biggl| 
	\sum\limits_{k,\ell=n+1}^\infty  
	c_k^{(n)} c_\ell^{(n)}  
	\left( 
	\bigl( \CC v_k^{(n)}, v_\ell^{(n)} \bigr)_{L_2} 
	- 
	a \bigl( \cC v_k^{(n)}, v_\ell^{(n)} \bigr)_{L_2} 
	\right) 
	\biggr|.   
\end{align*} 
Since $\bigl\{ v_k^{(n)} \bigr\}_{k\in\bbN}$ 
is an orthonormal basis for $H=\cC^{-1/2}(L_2)$, 
so is 
$\bigl\{ w_k^{(n)} \bigr\}_{k\in\bbN}$ 
for $L_2$, where 
$w_k^{(n)} := \cC^{1/2} v_k^{(n)}$. 
We set     
$w_n^h := \sum_{k = n+1}^\infty c_k^{(n)} w_k^{(n)}$ 
and obtain 
\[
\text{(A)} 
= 
\bigl| 
\bigl( \bigl( \cC^{-1/2} \CC \cC^{-1/2} -a \cI \bigr) 
Q_n^\perp w_n^h , 
Q_n^\perp w_n^h \bigr)_{L_2} \bigr|, 
\]  
where   
$Q_n^\perp := \cI- Q_n$ 
and 
$Q_n \from L_2 \to W_n$ 
denotes the $L_2$-orthogonal projection 
onto the subspace 
$W_n := 
\operatorname{span}\bigl\{w_1^{(n)}, \ldots, 
w_n^{(n)}\bigr\}$. 
By Assumption~\ref{ass:kriging}.III, 
$T_a = \cC^{-1/2} \CC \cC^{-1/2} -a \cI$ 
is compact on $L_2$.
For this reason, there exists an 
orthonormal basis $\{b_j\}_{j\in\bbN}$ for $L_2$ 
consisting of eigenvectors of $T_a$ with corresponding 
eigenvalues $\{\tau_j\}_{j\in\bbN}\subset\bbR$ 
accumulating only at zero. 
For $J\in\bbN$,  
we define $V_J:=\operatorname{span}\{b_1, \ldots, b_J\}$.  
We write $P_J\from L_2 \to V_J$ 
for the corresponding $L_2$-orthogonal projection   
and set $P_J^\perp := \cI-P_J$.  
Then, by invoking the chain of identities   
$\| w_n^h \|_{L_2}^2   
= \sum_{k=n+1}^\infty \bigl| c_k^{(n)} \bigr|^2 
= \pE\bigl[ (\hn - h)^2 \bigr] 
= \pV[ \hn -h ]$, 
we estimate 
\[ 
\text{(A)} 
\leq 
\pV[ \hn -h ] 
\sup_{\norm{w}{L_2}=1} 
\bigl| \bigl( T_a Q_n^\perp w, Q_n^\perp w \bigr)_{L_2} \bigr| . 
\]
Clearly, if $T_a = 0$, we obtain that $\text{(A)} = 0$. 
Thus, from now on we assume that 
$\norm{T_a}{\cL(L_2)} > 0$. 
Since $\cI = P_J + P_J^\perp$ 
and $P_J^\perp T_a P_J = P_J T_a P_J^\perp = 0$ 
we find   
\begin{align} 
	\frac{ \text{(A)} }{ \pV[ \hn -h ]  }  
	&\leq 
	\sup_{\norm{w}{L_2}=1} 
	\left| 
	\bigl( P_J^\perp T_a P_J^\perp Q_n^\perp w, Q_n^\perp w \bigr)_{L_2}   
	+ \bigl( P_J T_a P_J Q_n^\perp w, Q_n^\perp w \bigr)_{L_2} 
	\right| 
	\notag\\ 
	&\leq 
	\sup_{\norm{w}{L_2}=1} 
	\bigl| \bigl( T_a P_J^\perp w, P_J^\perp w \bigr)_{L_2} \bigr| 
	+ 
	\sup_{\norm{w}{L_2}=1}  
	\bigl\| Q_n^\perp P_J T_a P_J Q_n^\perp w \bigr\|_{L_2} . 
	\label{eq:proof:sufficiency-1}
\end{align} 
Here, we have used self-adjointness 
of $T_a, P_J, P_J^\perp$ and $Q_n^\perp$  
on $L_2$ in the last 
step.  
Now fix $\eps\in(0,\infty)$.  
Since $\lim_{j\to\infty}\tau_j = 0$, 
there exists  
$J_\eps\in\bbN$ with  
\begin{equation}\label{eq:proof:sufficiency-2} 
	\textstyle 
	\sup\limits_{\norm{w}{L_2}=1} 
	\bigl| 
	\bigl( T_a
	P_{J_\eps}^\perp w, P_{J_\eps}^\perp w \bigr)_{L_2} \bigr| 
	= \sup\limits_{j > J_\eps} |\tau_j| 
	< \frac{\eps}{2}. 
\end{equation} 
In addition,     
for 
$w := \sum_{k\in\bbN} \alpha_k^{(n)} w_k^{(n)} \in L_2$
and 
$h^w := \sum_{k\in\bbN} \alpha_k^{(n)} \psi_k^{(n)} \in \cH^0$, 
with some square-summable coefficients 
$\bigl\{ \alpha_k^{(n)} \bigr\}_{k\in\bbN}$, 
we find that 
\begin{equation}\label{eq:Q-n-dense}
	\textstyle 
	\norm{ Q_n^\perp w }{L_2}^2
	= 
	\sum\limits_{k=n+1}^\infty \bigl| \alpha_k^{(n)} \bigr|^2 
	= 
	\Bigl\| 
	\sum\limits_{k=n+1}^\infty \alpha_k^{(n)} \psi_k^{(n)} 
	\Bigr\|_{\cH^0}^2 
	=
	\| \hn^w - h^w \|_{\cH}^2.  
\end{equation} 
Because of this relation 
and thanks to the assumption that 
$\{\cH_n\}_{n\in\bbN}\in\cS^\mu_{\mathrm{adm}}$,  
there exists   
$n_\eps\in\bbN$ such that 
$\max_{1\leq j \leq J_\eps}
\norm{Q_n^\perp b_j }{L_2} < 
\frac{\eps}{2 \norm{T_a}{\cL(L_2)} \sqrt{J_\eps}}$ 
holds for every $n\geq n_\eps$, 
cf.~\eqref{eq:def:S-adm}. 
Therefore, for all $n\geq n_\eps$,  
we obtain that 
\[
\textstyle 
\norm{ Q_n^\perp P_{J_\eps} w }{L_2}
<  
\frac{\eps}{2 \norm{T_a}{\cL(L_2)} \sqrt{J_\eps}}  
\sum\limits_{j=1}^{J_\eps} 
| \scalar{w, b_j}{L_2} | 
\leq  
\frac{\eps}{2 \norm{T_a}{\cL(L_2)} } \norm{P_{J_\eps} w}{L_2} 
\quad  
\forall w \in L_2. 
\] 
The norm identities   
$\norm{P_{J_\eps}}{\cL(L_2)} 
= 
\norm{Q_n^\perp}{\cL(L_2)} = 1$  
thus imply that, 
for every $n\geq n_\eps$, and 
for all $w \in L_2$, 
\begin{equation}\label{eq:proof:sufficiency-3}
	\textstyle 
	\norm{ Q_n^\perp P_{J_\eps} T_a P_{J_\eps} Q_n^\perp w }{L_2} 
	<  
	\frac{\eps}{2 \norm{T_a}{\cL(L_2)}} 
	\norm{ P_{J_\eps} T_a P_{J_\eps} Q_n^\perp w }{L_2}
	\leq 
	\frac{\eps}{2} \norm{w}{L_2}  
\end{equation} 
holds. 
Combining 
\eqref{eq:proof:sufficiency-1}, 
\eqref{eq:proof:sufficiency-2} 
and \eqref{eq:proof:sufficiency-3}
shows that 
$\sup_{h\in\cH_{-n}}
\frac{\text{(A)}}{\pV[ \hn - h ]}
<\eps$ 
for every $n\geq n_\eps$ 
and, 
since $\eps\in(0,\infty)$ was arbitrary, 
\[ 
\lim_{n\to \infty} 
\sup_{h\in\cH_{-n}}
\frac{\text{(A)}}{\pV[ \hn - h ]}
= 
\lim_{n\to \infty} 
\sup_{h\in\cH_{-n}}
\left| 
\frac{\PV[\hn - h]}{\pV[\hn-h]} - a 
\right| 
=0, 
\]
i.e., \eqref{eq:kriging_gen_var3}  
follows. 
Furthermore, 
$\CC^{-1/2}\cC\CC^{-1/2} - a^{-1}\cI$  
is compact on $L_2$ 
by 
Lemma~B.1 
(see Appendix~B in the Supplementary Material~\cite{kbsup}) 
and  
$\cS^{\mu}_{\mathrm{adm}} = \cS^{\MU}_{\mathrm{adm}}$
by Lemma~\ref{lem:S-mu}  
so that, after changing the roles
of the measures $\mu$ and $\MU$,  
\eqref{eq:kriging_gen_var3}   
implies \eqref{eq:kriging_gen_var4}. 

Next, we show validity of 
\eqref{eq:kriging_gen3} under 
Assumptions~\ref{ass:kriging}.I--III. 
To this end, we first split  
$\bigl|
\PE\bigl[ (\hn - h)^2 \bigr] 
- 
a \pE\bigl[ (\hn - h)^2 \bigr] 
\bigr|
\leq 
\text{(A)} 
+ 
\text{(B)}$ 
in  
term (A), which is defined as above, and 
term 
$\text{(B)}
:= 
\bigl| 
\PE [ \hn - h ] 
\bigr|^2$.
By the Cauchy--Schwarz inequality, 
\[
\textstyle 
\text{(B)} 
= 
\Bigl| 
\sum\limits_{k=n+1}^\infty
c_k^{(n)} 
\PE\bigl[ \psi_k^{(n)} \bigr]  
\Bigr|^2
\leq 
\pE\bigl[ ( \hn - h )^2 \bigr]
\sum\limits_{k=n+1}^{\infty} 
\bigl| \PE\bigl[ \psi_k^{(n)} \bigr] \bigr|^2.  
\]
For each $k\geq n+1$,  
we let $\bigl\{ \psi_{kj}^{(n)} \bigr\}_{j\in\bbN}$ 
be the coefficients of $\psi_k^{(n)}$ 
when represented with respect to 
the orthonormal basis $\{\gp_j\}_{j\in\bbN}$ 
from Lemma~\ref{lem:cH-ONB}\ref{lem:cH-ONB-ii}. We then find  
(recall that we have centered $\mu$ so that $\MU$ 
has mean $\widetilde{m}-m$): 
{\allowdisplaybreaks 
	\begin{align*}
		\textstyle
		\sum\limits_{k=n+1}^{\infty} 
		&\textstyle 
		\bigl| \PE\bigl[ \psi_k^{(n)} \bigr] \bigr|^2 
		= 
		\sum\limits_{k=n+1}^{\infty} 
		\Bigl| 
		\sum\limits_{j\in\bbN} 
		\psi_{kj}^{(n)}  
		\PE[ \gp_j ] 
		\Bigr|^2
		= 
		\sum\limits_{k=n+1}^{\infty} 
		\Bigl| 
		\sum\limits_{j\in\bbN} 
		\psi_{kj}^{(n)}  
		\PE\bigl[ \scalar{\GP^0, v_j }{L_2} \bigr] 
		\Bigr|^2 
		\\
		&\textstyle 
		= 
		\sum\limits_{k=n+1}^{\infty} 
		\Bigl|  
		\sum\limits_{j\in\bbN} 
		\psi_{kj}^{(n)}  
		\bigl( \widetilde{m} - m, \cC^{-1/2} e_j \bigr)_{L_2}  
		\Bigr|^2 
		= 
		\sum\limits_{k=n+1}^{\infty} 
		\bigl( \cC^{-1/2} (\widetilde{m} - m) , w_k^{(n)} \bigr)_{L_2}^2,    
	\end{align*}
	since} 
$w_k^{(n)} = \cC^{1/2} v_k^{(n)} 
= \cC^{1/2} \cJ^{-1} \psi_k^{(n)} 
= \sum_{j\in\bbN} \psi_{kj}^{(n)} e_j$ 
and this series converges in $L_2$. 
Therefore, 
$\sum_{k=n+1}^{\infty} 
\bigl| \PE\bigl[ \psi_k^{(n)} \bigr] \bigr|^2 
= 
\bigl\| Q_n^\perp \cC^{-1/2} (m - \widetilde{m} ) \bigr\|_{L_2}^2$ 
follows.   
By Assumption~\ref{ass:kriging}.II 
the difference of the means $m - \widetilde{m}$ is an element 
of the Cameron--Martin space $H^*=\cC^{1/2}(L_2)$.  
Consequently, 
$\cC^{-1/2}( m - \widetilde{m} ) \in L_2$ 
and the norm on the right-hand side converges to zero 
as $n\to\infty$ by \eqref{eq:Q-n-dense} 
and \eqref{eq:ass:Hn-dense}.  
This shows that also 
\begin{equation}\label{eq:term-B} 
	\lim_{n\to\infty} 
	\sup_{h\in\cH_{-n}}
	\frac{\text{(B)}}{\pE\bigl[ ( \hn - h )^2 \bigr]} 
	= 
	\lim_{n\to\infty} 
	\sup_{h\in\cH_{-n}}
	\frac{ \bigl| \PE[ \hn - h ] \bigr|^2 }{ 
		\pE\bigl[ ( \hn - h )^2 \bigr]} 
	= 0.  
\end{equation}
We thus conclude with \eqref{eq:kriging_gen_var3} 
and \eqref{eq:term-B} that, 
uniformly in $h$, 
\[
\left| 
\frac{ \PE\bigl[ ( \hn - h )^2 \bigr] }{ 
	\pE\bigl[ ( \hn - h )^2 \bigr]} 
- a 
\right| 
\leq 
\left| 
\frac{ \PV[ \hn - h ] }{ 
	\pV[ \hn - h ]} 
- a 
\right| 
+ 
\frac{ \bigl| \PE[ \hn - h ] \bigr|^2 }{ 
	\pE\bigl[ ( \hn - h )^2 \bigr]} 
\to
0 
\quad 
\text{as} 
\quad 
n\to\infty, 
\]
and 
\eqref{eq:kriging_gen3} follows. 
Again, by virtue of Lemma~\ref{lem:S-mu} and 
Lemma~B.1 (see Appendix~B in the Supplementary Material~\cite{kbsup}) 
we may   
change the roles of $\mu$ and $\MU$
which gives \eqref{eq:kriging_gen4}. 

To derive \eqref{eq:kriging_gen1}, note that 
$\pE\bigl[ ( \hn - h )^2 \bigr] 
\leq 
\pE\bigl[ ( \Hn - h )^2 \bigr]$ 
as $\hn$ is the $\mu$-best linear predictor. 
For the same reason, we obtain   
$\PE\bigl[ ( \Hn - h )^2 \bigr] 
\leq 
\PE\bigl[ ( \hn - h )^2 \bigr]$,   
and the estimates 
$1 \leq 
\frac{\pE[ ( \Hn - h )^2 ]}{\pE[ ( \hn - h )^2 ]}
\leq 
\frac{\pE[ ( \Hn - h )^2 ]}{\PE[ ( \Hn - h )^2 ]}
\frac{\PE[ ( \hn - h )^2 ]}{\pE[ ( \hn - h )^2 ]}$
follow 
similarly as in \eqref{eq:trick-var}. 
By \eqref{eq:kriging_gen4} and \eqref{eq:kriging_gen3} 
the last two fractions converge 
to $a^{-1}$ and to $a$, 
uniformly~in~$h$,  
as $n\to\infty$ 
and \eqref{eq:kriging_gen1} follows. 
Changing the roles of $\mu$ and $\MU$  
implies~\eqref{eq:kriging_gen2}. 

Finally, note that if 
$\mu=\normal(m,\cC)$ and 
$\MU=\normal(\widetilde{m},\CC)$ 
are such that 
Assumptions~\ref{ass:kriging}.I 
and~\ref{ass:kriging}.III 
are satisfied, 
then the centered measures 
$\muc=\normal(0,\cC)$ and 
$\MUc=\normal(0,\CC)$ 
satisfy Assumptions~\ref{ass:kriging}.I--III 
so that 
\eqref{eq:kriging_gen1}, \eqref{eq:kriging_gen2} 
hold for the pair $\muc,\MUc$
and  
\eqref{eq:kriging_gen_var1}, \eqref{eq:kriging_gen_var2} 
follow from the identities 
in (C.2) and (C.3), see Appendix~C 
in the Supplementary Material~\cite{kbsup}. 
\end{proof} 

\begin{proof}[Proof of Theorem~\ref{thm:kriging-equiv-measures} 
\& Corollary~\ref{cor:kriging-equiv-measures-var}]
If the measures $\mu$ and~$\MU$ are equivalent, 
then by the Feldman--H\'ajek theorem  
(see Theorem~A.1 in Appendix~A of the Supplementary Material~\cite{kbsup})  
Assumptions~\ref{ass:kriging}.I--II  
hold and  
$T_1 = \cC^{-1/2}\CC\cC^{-1/2}-\cI \in \cL_2(L_2)$. 
Since every Hilbert--Schmidt operator is compact, 
this implies that also Assumption~\ref{ass:kriging}.III 
is fulfilled for $a=1$. 
Therefore, for every   
$\{\cH_n\}_{n\in\bbN}\in\cS^\mu_{\mathrm{adm}}$,  
all assertions    
in \eqref{eq:kriging_equiv-opt}, \eqref{eq:kriging_equiv-a},  
\eqref{eq:kriging_equiv_var-opt}, \eqref{eq:kriging_equiv_var-a} 
follow from 
Proposition~\ref{prop:sufficiency}. 
\end{proof} 

\begin{lemma}\label{lem:nec-Ass-I+III}
Let $\mu=\normal(m,\cC)$,  
$\MU=\normal(\widetilde{m},\CC)$.   
In \eqref{eq:kriging_gen_var1}--\eqref{eq:kriging_gen_var4}
let $\hn,\Hn$ denote the best linear predictors of $h$ 
based on~$\cH_n$ and the measures~$\mu,\MU$.  
Then, validity of 
any of the statements  
\eqref{eq:kriging_gen_var1}, 
\eqref{eq:kriging_gen_var2},
\eqref{eq:kriging_gen_var3},  
or \eqref{eq:kriging_gen_var4}
for all 
$\{\cH_n\}_{n\in\bbN}\in\cS_{\mathrm{adm}}^\mu$  
implies that the Assumptions~\ref{ass:kriging}.I 
and~\ref{ass:kriging}.III are satisfied, 
and the constant $a\in(0,\infty)$ 
in \eqref{eq:kriging_gen_var3}, 
\eqref{eq:kriging_gen_var4} is the 
same as in \eqref{eq:ass:Ta}. 
\end{lemma} 

\begin{proof} 
By the identities~(C.2)--(C.5) 
(see Appendix~C in the Supplementary Material~\cite{kbsup}) 
we can without loss of generality 
assume that $m=\widetilde{m}=0$. Then,  
$\pV\bigl[ \Hn - h \bigr] 
= 
\pE\bigl[ (\Hn - h)^2 \bigr]$  
and 
$\PV\bigl[ \hn - h \bigr] 
= 
\PE\bigl[ (\hn - h)^2 \bigr]$ 
follow. 
Furthermore,  
$\pV\bigl[ \hn - h \bigr] 
= 
\pE\bigl[ (\hn - h)^2 \bigr]$ 
and 
$\PV\bigl[ \Hn - h \bigr] 
= 
\PE\bigl[ (\Hn - h)^2 \bigr]$ 
always hold 
by unbiasedness of the kriging predictor. 
Recall from Lemma~\ref{lem:cH-ONB} 
the orthonormal bases 
$\{e_j\}_{j\in\bbN}$ for $L_2$, 
$\{ v_j \}_{j\in\bbN}$ 
for $H=\cC^{-1/2}(L_2)$, 
and 
$\{ \gp_j \}_{j\in\bbN}$ 
for~$\cH^0$  
as well as 
the isometry~$\cJ\from H \to \cH^0$ 
which identifies $v_j$ with $\gp_j$.  

If any of the statements  
\eqref{eq:kriging_gen_var1}, 
\eqref{eq:kriging_gen_var2}, 
\eqref{eq:kriging_gen_var3}, 
or
\eqref{eq:kriging_gen_var4} 
holds for every  
$\{\cH_n\}_{n\in\bbN}\in\cS^\mu_{\mathrm{adm}}$,  
then by Lemma~B.4  
(see Appendix~B in the Supplementary Material~\cite{kbsup}), 
all four assertions of Proposition~\ref{prop:AssI}  
and, in particular, 
Proposition~\ref{prop:AssI}\ref{prop:AssI-i} hold, i.e., 
Assumption~\ref{ass:kriging}.I  
is satisfied. 

Next, we prove that validity of 
\eqref{eq:kriging_gen_var2} 
for all 
$\{\cH_n\}_{n\in\bbN} \in\cS^\mu_{\mathrm{adm}}$ 
implies  
Assumption~\ref{ass:kriging}.III. 
For $n\in\bbN$, 
define 
$E_n := \operatorname{span}\{e_1,\ldots,e_n\}\subset L_2$ 
and 
$H_n := \operatorname{span}\{v_1,\ldots,v_n\}\subset H$,   
and let 
$E_n^\perp = \operatorname{span}\{e_j\}_{j>n}$ 
as well as 
$H_n^\perp = \operatorname{span}\{v_j\}_{j>n}$   
be their orthogonal complements 
in $L_2$ and $H$, respectively. 
Note that $E_n = H_n$ 
and $E_n^\perp \subset H_n^\perp$. 
Now suppose that, for all $a\in(0,\infty)$, 
the linear operator 
$T_a = \cC^{-1/2} \CC \cC^{-1/2} - a\cI$ 
is not compact on $L_2$, and define    
$\underline{\alpha} := \norm{\cC^{1/2}\CC^{-1/2}}{\cL(L_2)}^{-2}$, 
$\overline{\alpha} := \norm{\CC^{1/2}\cC^{-1/2}}{\cL(L_2)}^2$. 
Then, by Lemma~B.2 
(see Appendix~B in the Supplementary Material~\cite{kbsup}) 
there exist $\delta\in(0,\infty)$ 
and, for every $n\in\bbN$, 
$\underline{a}_n, \overline{a}_n \in [\underline{\alpha},\overline{\alpha}]$ and 
$\underline{w}_n,\overline{w}_n 
\in E_n^\perp\setminus\{0\}$ 
such that, for all $n\in\bbN$, we have 
$\overline{a}_n - \underline{a}_n\geq \delta$ 
and 
\[
\biggl| 
\frac{ \scalar{ \cC^{-1/2}\CC\cC^{-1/2} \underline{w}_n, 
		\underline{w}_n }{L_2} }{ 
	\scalar{ \underline{w}_n, 
		\underline{w}_n }{L_2}} 
- 
\underline{a}_n 
\biggr| 
< \frac{\delta \underline{\alpha}^2}{3 \overline{\alpha}^2},  
\qquad   
\biggl| 
\frac{ \scalar{ \cC^{-1/2}\CC\cC^{-1/2} \overline{w}_n, 
		\overline{w}_n }{L_2} }{ 
	\scalar{ \overline{w}_n, 
		\overline{w}_n }{L_2}} 
- 
\overline{a}_n    
\biggr| 
< \frac{\delta \underline{\alpha}^2}{3 \overline{\alpha}^2}.  
\]  
We set 
$\underline{c}_n := \overline{a}_n^{\,-1}$, 
$\overline{c}_n := \underline{a}_n^{-1}$, and 
$\underline{v}_n := \cC^{-1/2} \overline{w}_n$,  
$\overline{v}_n := \cC^{-1/2} \underline{w}_n$. 
Then, we obtain that, for all $n\in\bbN$,  
$\underline{c}_n, \overline{c}_n \in 
\bigl[ \overline{\alpha}^{\,-1}, \underline{\alpha}^{-1} \bigr]$, 
and 
$\overline{c}_n - \underline{c}_n 
\geq \delta' := \delta \overline{\alpha}^{\,-2}$. 
The vectors 
$\underline{v}_n, \overline{v}_n \in H_n^\perp$ 
satisfy  
$\Bigl| 
\frac{ \scalar{\cC \underline{v}_n, \underline{v}_n }{L_2} }{ 
	\scalar{   \CC \underline{v}_n, \underline{v}_n }{L_2}} 
- 
\underline{c}_n \,     
\Bigr| 
< \frac{\delta'}{3}$, 
$\Bigl| 
\frac{ \scalar{\cC \overline{v}_n, \overline{v}_n }{L_2} }{ 
	\scalar{   \CC \overline{v}_n, \overline{v}_n }{L_2}} 
- 
\overline{c}_n \, 
\Bigr| 
< \frac{\delta'}{3}$. 
We then define 
$\underline{\phi}_n := \cJ \underline{v}_n \in\cH^0$ 
and 
$\bar{\phi}_n  := \cJ \overline{v}_n \in\cH^0$
and find that 
\[ 
\frac{\pE\bigl[ \underline{\phi}_n^2 \bigr]}{ 
	\PE\bigl[ \underline{\phi}_n^2 \bigr]} 
= 
\frac{\scalar{\cC \underline{v}_n, \underline{v}_n }{L_2} }{
	\scalar{  \CC \underline{v}_n, \underline{v}_n}{L_2} } 
\in 
\bigl( 
\underline{c}_n - \tfrac{\delta'}{3}, \, 
\underline{c}_n + \tfrac{\delta'}{3} \bigr) , 
\quad 
\frac{\pE\bigl[ \bar{\phi}_n^2 \bigr]}{ 
	\PE\bigl[ \bar{\phi}_n^2 \bigr]} 
= 
\frac{\scalar{\cC \overline{v}_n, \overline{v}_n }{L_2} }{
	\scalar{  \CC \overline{v}_n, \overline{v}_n}{L_2} } 
\in 
\bigl( 
\overline{c}_n - \tfrac{\delta'}{3}, \, 
\overline{c}_n + \tfrac{\delta'}{3} \bigr) . 
\]  
As in 
\cite[][Proof of Theorem~5]{CLEVELAND71} 
it follows that there exist 
$h^{(n)}, \psi_n \in 
\operatorname{span}\{ \underline{\phi}_n, \bar{\phi}_n \}$ 
such that 
\begin{equation}\label{eq:tilde-theta-equality}  
	\frac{ \PE\bigl[ (h^{(n)}_1 - h^{(n)} )^2 \bigr] }{
		\PE\bigl[ ( \widetilde{h}^{(n)}_1 - h^{(n)} )^2 \bigr] } 
	= 
	\frac{ (\widetilde{\theta}_n + \widetilde{\Theta}_n )^2 }{
		4 \widetilde{\theta}_n \widetilde{\Theta}_n} 
\end{equation}
holds,  
where $h^{(n)}_1$ and $\widetilde{h}^{(n)}_1$ are  
the best linear predictors of~$h^{(n)}$ based 
on the subspace  
${\cV_{n} := \bbR\oplus\operatorname{span}\{ \psi_n \}
	\subset\cH}$ 
and the measures   
$\mu$ and $\MU$, respectively. 
Moreover,  
\[
\widetilde{\theta}_n  
:= 
\min\bigl\{ 
\pE \bigl[ h^2 \bigr] / \PE \bigl[ h^2 \bigr] : 
h \in\operatorname{span}\bigl\{ \underline{\phi}_n, \bar{\phi}_n \bigr\}, \, 
h \neq 0 
\bigr\} , 
\]
and $\widetilde{\Theta}_n\in(0,\infty)$ is defined 
as $\widetilde{\theta}_n$ with $\min$ 
replaced by $\max$. 
Clearly, these definitions yield that    
$\widetilde{\theta}_n 
\leq 
\pE \bigl[ \underline{\phi}_n^2 \bigr] /
\PE \bigl[ \underline{\phi}_n^2 \bigr]  
< 
\underline{c}_n + \tfrac{\delta'}{3}$ 
and 
$\widetilde{\Theta}_n
\geq 
\pE\bigl[ \bar{\phi}_n^2 \bigr] / 
\PE\bigl[ \bar{\phi}_n^2 \bigr] 
> 
\overline{c}_n - \tfrac{\delta'}{3}$, 
which implies that  
$\widetilde{\Theta}_n - \widetilde{\theta}_n 
> 
\overline{c}_n - \underline{c}_n - \tfrac{2\delta'}{3} 
= \tfrac{\delta'}{3}$.  
As we have already derived 
fulfillment of Assumption~\ref{ass:kriging}.I, 
Proposition~\ref{prop:AssI}\ref{prop:AssI-i}$\,\Leftrightarrow\,$\ref{prop:AssI-ii}
shows that   
$\widetilde{\Theta}_n 
\leq 
\sup_{h\in \cH^0 \setminus\{0\} }
\frac{ \pE[ h^2 ] }{ \PE[ h^2 ] } 
\leq 
\norm{\cC^{1/2} \CC^{-1/2} }{\cL(L_2)}^2 
< \infty$. 
Define 
$\cH^\star_1 :=  \cV_1$ and, 
for $n\geq 2$, set $\cH^\star_n := 
\bbR\oplus 
\operatorname{span}\{ 
\gp_1, \ldots, \gp_{n-1}, \psi_n \}$. 
By the basis property of $\{\gp_j\}_{j\in\bbN}$ in $\cH^0$, 
see Lemma~\ref{lem:cH-ONB}\ref{lem:cH-ONB-ii}, 
the so constructed subspaces  
are admissible, i.e.,  
$\{\cH^\star_n\}_{n\in\bbN} \in\cS^\mu_{\mathrm{adm}}$. 
Since $h^{(n)} , \psi_n 
\in \operatorname{span}\{ \underline{\phi}_n, \bar{\phi}_n \}$ 
and since 
$\underline{\phi}_n, \bar{\phi}_n$   
are $\cH^0$-orthogonal 
to $\gp_1,
\ldots, \gp_{n-1}$, 
we obtain $h^{(n)}_n=h^{(n)}_1$, 
where~$h^{(n)}_n$ is the best linear predictor of $h^{(n)}$ 
based on~$\cH^\star_n$  
and~$\mu$. 
Thus, by using \eqref{eq:tilde-theta-equality} we obtain, for all $n\in\bbN$, 
\[   
\frac{\PE\bigl[ ( h_n^{(n)} - h^{(n)} )^2 \bigr] }{
	\PE\bigl[ ( \widetilde{h}^{(n)}_n - h^{(n)} )^2 \bigr] } 
- 
1 
\geq 
\frac{ (\widetilde{\Theta}_n - \widetilde{\theta}_n)^2 }{
	4 \widetilde{\theta}_n \widetilde{\Theta}_n } 
> 
\frac{ \delta^{\prime \, 2} }{ 36 \, \widetilde{\Theta}_n^2 }
\geq 
\frac{ \delta^{\prime \, 2} }{ 
	36 \norm{\cC^{1/2} \CC^{-1/2} }{\cL(L_2)}^4 }, 
\] 
a contradiction to \eqref{eq:kriging_gen_var2} 
for the sequence  
$\{\cH^\star_n\}_{n\in\bbN} \in\cS^\mu_{\mathrm{adm}}$,  
which proves  
that 
\eqref{eq:kriging_gen_var2} 
holding for all 
$\{\cH_n\}_{n\in\bbN} \in\cS^\mu_{\mathrm{adm}}$
implies
Assumption~\ref{ass:kriging}.III.

Next, we show that validity of \eqref{eq:kriging_gen_var3} 
for all $\{\cH_n\}_{n\in\bbN}\in\cS^\mu_{\mathrm{adm}}$
also implies that 
Assumption~\ref{ass:kriging}.III 
is satisfied. 
To this end, suppose that this 
assumption does not hold. 
It then again follows from Lemma~B.2
(see Appendix~B in the Supplementary Material~\cite{kbsup}) 
that there are $\delta\in(0,\infty)$ and, for all $n\in\bbN$, 
$\underline{a}_n, \overline{a}_n \in[\underline{\alpha},\overline{\alpha}]$ and 
$\underline{v}_n, 
\overline{v}_n 
\in H_{n}^\perp$, 
linearly independent,   
such that 
$\Bigl| 
\frac{ \scalar{\CC \underline{v}_n, \underline{v}_n }{L_2} }{ 
	\scalar{   \cC \underline{v}_n, \underline{v}_n }{L_2}} 
- 
\underline{a}_n     
\Bigr| 
< \frac{\delta}{3}$, 
$\Bigl| 
\frac{ \scalar{\CC \overline{v}_n, \overline{v}_n }{L_2} }{ 
	\scalar{   \cC \overline{v}_n, \overline{v}_n }{L_2}} 
- 
\overline{a}_n
\Bigr| 
< \frac{\delta}{3}$, 
and 
$\overline{a}_n - \underline{a}_n\geq \delta$ for all $n\in\bbN$. 
Define  
$\bar{h}^{(n)}  := \cJ  \overline{v}_n\in\cH^0$ 
and 
$\underline{h}^{(n)} := \cJ \underline{v}_n\in\cH^0$. 
Then, for 
$\cH^\star_{n}:=\bbR\oplus\operatorname{span}\{\gp_1,\ldots,\gp_n\} 
\subset \cH$ 
and all $n\in\bbN$, 
\begin{equation}\label{eq:proof:nec-Ass-I+III}
	\frac{\PV\bigl[ \bar{h}^{(n)}_n - \bar{h}^{(n)} \bigr]}{
		\pV\bigl[ \bar{h}^{(n)}_n - \bar{h}^{(n)} \bigr]} 
	- 
	\frac{\PV\bigl[ \underline{h}^{(n)}_n - \underline{h}^{(n)} \bigr]}{
		\pV\bigl[ \underline{h}^{(n)}_n - \underline{h}^{(n)} \bigr]} 
	= 
	\frac{\PV\bigl[ \bar{h}^{(n)} \bigr]}{
		\pV\bigl[ \bar{h}^{(n)} \bigr]} 
	- 
	\frac{\PV\bigl[ \underline{h}^{(n)} \bigr]}{
		\pV\bigl[ \underline{h}^{(n)} \bigr]} 
	\geq 
	\frac{\delta}{3}. 
\end{equation} 
Note that, if \eqref{eq:kriging_gen_var3} 
holds for all sequences  
$\{\cH_n \}_{n\in\bbN}\in\cS^\mu_{\mathrm{adm}}$, 
then, in particular, 
\[
\forall 
\{\cH_n \}_{n\in\bbN}\in\cS^\mu_{\mathrm{adm}}
: 
\qquad 
\lim\limits_{n\to\infty}
\biggl( 
\sup\limits_{g\in\cH_{-n}}
\frac{\PV[g_n - g]}{\pV[g_n - g]} 
- 
\inf\limits_{h\in\cH_{-n}}
\frac{\PV[\hn - h]}{\pV[\hn - h]} 
\biggr) = 0
\]
follows. 
Therefore, \eqref{eq:proof:nec-Ass-I+III} 
contradicts \eqref{eq:kriging_gen_var3}  
for the sequence 
$\{\cH^\star_n\}_{n\in\bbN}\in\cS^\mu_{\mathrm{adm}}$ 
and, thus, 
Assumption~\ref{ass:kriging}.III is satisfied 
if \eqref{eq:kriging_gen_var3} holds for 
all $\{\cH_n\}_{n\in\bbN}\in\cS^\mu_{\mathrm{adm}}$.

Finally, necessity of Assumption~\ref{ass:kriging}.III 
for validity of 
\eqref{eq:kriging_gen_var1} (or \eqref{eq:kriging_gen_var4}) 
holding 
for all $\{\cH_n\}_{n\in\bbN}\in\cS^\mu_{\mathrm{adm}}$
follows from changing the roles of $\mu$ and~$\MU$: 
If $\mu$ and $\MU$ are such that 
\eqref{eq:kriging_gen_var1} (or \eqref{eq:kriging_gen_var4}) 
holds for every 
$\{\cH_n\}_{n\in\bbN}\in\cS^\mu_{\mathrm{adm}}$, 
then \eqref{eq:kriging_gen_var2} 
(or \eqref{eq:kriging_gen_var3}) 
is true for the pair $\MU,\mu$ 
and every $\{\cH_n\}_{n\in\bbN}\in
\cS^{\mu}_{\mathrm{adm}}$. 
Since necessity of  
Assumption~\ref{ass:kriging}.I 
has already been derived, 
by Lemma~\ref{lem:S-mu}
we have 
$\cS^{\mu}_{\mathrm{adm}}=\cS^{\MU}_{\mathrm{adm}}$ 
so that the above arguments 
combined 
with Lemma~B.1 
in Appendix~B of the Supplementary Material~\cite{kbsup} 
show that Assumption~\ref{ass:kriging}.III 
also holds.
\end{proof}

\begin{proof}[Proof of Theorem~\ref{thm:kriging_general_var}]
Sufficiency and necessity 
of Assumptions~\ref{ass:kriging}.I 
and~\ref{ass:kriging}.III 
for \eqref{eq:kriging_gen_var1}--\eqref{eq:kriging_gen_var4} 
to hold for all 
$\{\cH_n\}_{n\in\bbN}\in\cS^\mu_{\mathrm{adm}}$ 
have been proven in Proposition~\ref{prop:sufficiency} 
and Lemma~\ref{lem:nec-Ass-I+III}, respectively. 
\end{proof}

\begin{lemma}\label{lem:mean_way_back}
Define the Gaussian measures $\muc=\normal(0, \cC)$ and  
$\mus=\normal(\widetilde{m}-m, \cC)$, 
with corresponding expectation operators $\pEc, \pEs$. 
Let $\hnc,\hns$ denote the best linear predictors of~$h$ 
based on~$\cH_n \in\{\cH_n \}_{n\in\bbN}\in\cS^{\mu}_{\mathrm{adm}}$ 
and the measures $\muc$ resp.\ $\mus$. 
For $n\in\bbN$ and $h\in\cH_{-n}$, 
consider the errors of the predictors, 
$\errc = \errc(h,n) := \hnc-h$,  
$\errs = \errs(h,n) := \hns-h$. Then,
\begin{equation}\label{eq:errors-m}
	\frac{
		\pEc\bigl[ \errs^2 \bigr]
	}{
		\pEc\bigl[ \errc^2 \bigr]
	} 
	- 1
	= 
	\frac{
		\pEs\bigl[ \errc^2 \bigr]
	}{
		\pEs\bigl[ \errs^2 \bigr]
	}
	-1 
	= 
	\left| 
	\frac{
		\pEs\bigl[ \errc^2 \bigr]
	}{
		\pEc\bigl[ \errc^2 \bigr]
	} 
	- 1
	\right| 
	= 
	\left|
	\frac{
		\pEc\bigl[ \errs^2 \bigr]
	}{
		\pEs\bigl[ \errs^2 \bigr]
	} 
	- 1
	\right| 	
	= 
	\frac{ \bigl| \pEs[ \errc ] \bigr|^2 }{
		\pEc\bigl[ \errc^2 \bigr]}.  
\end{equation}
Furthermore, for all 
$\{\cH_n \}_{n\in\bbN}\in\cS^{\mu}_{\mathrm{adm}}$,  
this term 
is bounded, uniformly 
with respect to $n\in\bbN$  
and $h\in\cH_{-n}$,  
if and only if 
Assumption~\ref{ass:kriging}.II 
is satisfied.  	
Under Assumption~\ref{ass:kriging}.II, 
$\lim\limits_{n\to\infty} 
\sup\nolimits_{h\in\cH_{-n}} 
\frac{ | \pEs[ \errc(h,n) ] |^2 }{
	\pEc[ \errc(h,n)^2 ]}=0$ 
holds for all
$\{\cH_n \}_{n\in\bbN}\in\cS^{\mu}_{\mathrm{adm}}$. 
\end{lemma}

\begin{proof} 
Let $n\in\bbN$ and $h\in\cH_{-n}$. 
By  
$\cH$-orthogonality of 
$\errc =  \hnc - h$ to $\cH_n$ 
we obtain  
\[
\pEc\bigl[ \errs^2 \bigr]
- 
\pEc\bigl[ \errc^2 \bigr] 
=
\pEc\bigl[ \errs \errc \bigr] 
+ 
\pEc\bigl[ \errs ( \hns - \hnc )  \bigr]
- 
\pEc\bigl[ \errc  \errs \bigr] 
= 
\pEc\bigl[ (\hns - \hnc)^2  \bigr].   
\]  
Since $\muc$ and 
$\mus$ have the same covariance 
operator, we can combine the 
above equality with (C.1)
from Lemma~C.1 
(see Appendix~C in the Supplementary Material~\cite{kbsup}) 
which gives  
$\frac{ \pEc[ \errs^2 ]}{
	\pEc[ \errc^2 ]} - 1
= 
\frac{ \pEc[ (\hns - \hnc)^2 ]}{
	\pEc[ \errc^2 ]}
=  
\frac{ | \pEs[ \errc ] |^2 }{
	\pEc[ \errc^2 ]}$. 
Noting that 
$\bigl| \pEc[ \errs ] \bigr| = \bigl| \pEs[ \errc ] \bigr|$ 
and 
$\pEs\bigl[ \errs^2 \bigr] = \pEc\bigl[ \errc^2 \bigr]$ 
due to the identical covariance operators of
$\muc, \mus$  
yields the relation   
$\frac{
	\pEs[ \errc^2 ]
}{
	\pEs[ \errs^2 ]
}
-1 
= 
\frac{ | \pEs[ \errc ] |^2 }{
	\pEc[ \errc^2 ]}$. 
Next, again by equality of the covariance operators, we
find that 
$\pEs\bigl[ \errc^2 \bigr] - \pEc\bigl[ \errc^2 \bigr]
= \bigl| \pEs[ \errc ] \bigr|^2$ 
and  
$\pEc\bigl[ \errs^2 \bigr] - \pEs\bigl[ \errs^2 \bigr]
= \bigl| \pEc[ \errs ] \bigr|^2 
= \bigl| \pEs[ \errc ] \bigr|^2$ 
which completes the proof 
of \eqref{eq:errors-m}. 

Now suppose that 
Assumption~\ref{ass:kriging}.II 
is satisfied  
and let 
$\{\cH_n\}_{n\in\bbN} \in\cS^\mu_{\mathrm{adm}}$. 
Then, we obtain 
$\lim\limits_{n\to\infty} 
\sup\nolimits_{h\in\cH_{-n}} 
\frac{ | \pEs[ \errc(h,n) ] |^2 }{
	\pEc[ \errc(h,n)^2 ]}=0$ 
as in \eqref{eq:term-B} 
with $\cC=\CC$.
In particular,  
there exists $K\in(0,\infty)$ 
such that 
$\sup\nolimits_{n\in\bbN} 
\sup\nolimits_{h\in\cH_{-n}} 
\frac{ | \pEs[ \errc(h,n) ] |^2 }{
	\pEc[ \errc(h,n)^2 ]} \leq K$.

Finally, assume that
$m-\widetilde{m}\notin H^* = \cC^{1/2}(L_2)$, 
i.e., Assumption~\ref{ass:kriging}.II 
is not satisfied. 
For $n\in\bbN$, 
define $H_n = \operatorname{span}\{v_1,\ldots,v_n\}$, 
where $\{v_j\}_{j\in\bbN}$ is the orthonormal basis 
of $H$ from Lemma~\ref{lem:cH-ONB}\ref{lem:cH-ONB-i}, 
and let $H_n^\perp$ be the 
$H$-orthogonal complement of $H_n$.  
Since $m-\widetilde{m}\notin H^*$ and $L_2$ is dense in $H$, 
we can find  
$\{\overline{v}_n\}_{n\in\bbN}\subset L_2\setminus\{0\}$ 
such that  
$\scalar{m-\widetilde{m},\overline{v}_n}{L_2} 
\geq 
n \norm{\overline{v}_n}{H}$. 
Furthermore, we may pick 
$\overline{v}_n$ in 
$H_n^\perp\subset H$, 
since $\dim(H_n)<\infty$. 
In summary, 
\[
\forall n \in \bbN  
\quad 
\exists \overline{v}_n \in L_2\cap H_n^\perp, \; 
\overline{v}_n\neq 0 : 
\quad 
\scalar{ m-\widetilde{m}, \overline{v}_n }{L_2}  
\geq 
n \sqrt{\scalar{\cC \overline{v}_n, \overline{v}_n }{L_2}} . 
\]
By \eqref{eq:cJ-iso:scalar} 
$h^{(n)} := \cJ \overline{v}_n \in \cH^0$ 
is $\cH^0$-orthogonal 
to $\cH^0_n := \operatorname{span}\{\gp_1, \ldots, \gp_n\}$ 
if $\{ \gp_j \}_{j\in\bbN}$ 
is the orthonormal basis for $\cH^0$  
from Lemma~\ref{lem:cH-ONB}\ref{lem:cH-ONB-ii}. 
Therefore, the kriging predictor of $h^{(n)}$ 
based on $\cH^\star_n:=\bbR\oplus\cH_n^0$  
and $\muc = \normal(0,\cC)$ vanishes, 
${h_n^{(n),c} = 0}$. 
Thus, 
there exist 
square-summable coefficients $\bigl\{ c_j^{(n)} \bigr\}_{j>n}$ 
such that   
${h^{(n)} = \sum_{j>n} c_j^{(n)} \gp_j}$ and  
$\overline{v}_n = \sum_{j>n} c_j^{(n)} v_j$, 
and we find that 
\[
\frac{ \bigl| \pEs\bigl[ h_n^{(n),c} - h^{(n)} \bigr] \bigr|^2 }{
	\pEc\bigl[  ( h_n^{(n),c} - h^{(n)} )^2 \bigr]}  
=
\frac{ \bigl| \pEs\bigl[ h^{(n)} \bigr] \bigr|^2 }{
	\pEc\bigl[  | h^{(n)} |^2 \bigr]} 
= 
\frac{ \bigl| 
	\pEs\bigl[ 
	\sum\nolimits_{j>n} c_j^{(n)} \scalar{\GP^0, v_j}{L_2} 
	\bigr] 
	\bigr|^2 }{
	\scalar{\cC \overline{v}_n, \overline{v}_n}{L_2} } 
= 
\frac{ \scalar{\widetilde{m}-m, \overline{v}_n}{L_2}^2 }{
	\scalar{\cC \overline{v}_n, \overline{v}_n}{L_2} } 
\geq 
n^2.    
\]
Furthermore, 
$\{\cH^\star_n\}_{n\in\bbN}\in\cS^{\mu}_{\mathrm{adm}}$ 
so that this yields a contradiction. 
\end{proof}

\begin{proof}[Proof of Theorem~\ref{thm:kriging_general}]
In this proof, all references starting with ``C'' 
are referring to Appendix~C 
in the Supplementary Material~\cite{kbsup}. 
As shown in 
(C.2)--(C.5), see
Proposition~C.2, 
we can equivalently prove the claim for 
the pair of measures 
$\muc=\normal(0,\cC)$,  
$\MUs=\normal(\widetilde{m}-m, \CC)$  
in place of $\mu=\normal(m,\cC)$ and 
$\MU=\normal(\widetilde{m},\CC)$.  
Sufficiency of Assumptions~\ref{ass:kriging}.I--III  
for each of the assertions 
\eqref{eq:kriging_gen1}--\eqref{eq:kriging_gen4} 
to hold for all 
$\{\cH_n\}_{n\in\bbN} \in \cS^\mu_{\mathrm{adm}}$
is shown in Proposition~\ref{prop:sufficiency}. 

Conversely, if \eqref{eq:kriging_gen1} 
(or \eqref{eq:kriging_gen2})
holds for $\muc, \MUs$ and all 
$\{\cH_n\}_{n\in\bbN} \in \cS^\mu_{\mathrm{adm}}$, 
then 
by (C.7)
the relation 
\eqref{eq:kriging_gen_var1} 
(or \eqref{eq:kriging_gen_var2}) 
holds for  
the pair $\mu,\MU$ 
and all 
$\{\cH_n\}_{n\in\bbN} \in \cS^\mu_{\mathrm{adm}}$.  
By Lemma~\ref{lem:nec-Ass-I+III}
Assumptions~\ref{ass:kriging}.I 
and~\ref{ass:kriging}.III  
have to be satisfied. 
Subsequently, necessity of Assumption~\ref{ass:kriging}.II  
for \eqref{eq:kriging_gen2} 
follows from (C.8) combined with 
Lemma~\ref{lem:mean_way_back}. 
Since we have already derived Assumption~\ref{ass:kriging}.I, 
we may also combine (C.8)
with Lemma~\ref{lem:mean_way_back} 
applied for $\MUc$ and $\MUs$, 
showing necessity of Assumption~\ref{ass:kriging}.II 
for \eqref{eq:kriging_gen1}. 

If \eqref{eq:kriging_gen3} 
(or \eqref{eq:kriging_gen4})
holds for $\muc,\MUs$ 
and all 
$\{\cH_n\}_{n\in\bbN} \in \cS^\mu_{\mathrm{adm}}$, 
then (C.6) combined 
with 
Proposition~\ref{prop:AssI}\ref{prop:AssI-i}$\,\Leftrightarrow\,$\ref{prop:AssI-iv}, 
see also Remark~C.3, 
show that Assumption~\ref{ass:kriging}.I  
has to be satisfied. 
Thereafter, (C.6) and 
Lemma~\ref{lem:mean_way_back}, 
(using $\muc,\mus$ for \eqref{eq:kriging_gen3} and 
$\MUc,\MUs$ for \eqref{eq:kriging_gen4})  
show that also Assumption~\ref{ass:kriging}.II  
holds. 
Finally,  
again the identity in (C.6) 
implies \eqref{eq:kriging_gen_var3} 
(or \eqref{eq:kriging_gen_var4})
for the pair $\mu,\MU$ and 
all $\{\cH_n\}_{n\in\bbN} \in \cS^\mu_{\mathrm{adm}}$  
and Assumption~\ref{ass:kriging}.III 
follows from Lemma~\ref{lem:nec-Ass-I+III}. 
\end{proof}

%====================================================================
\section{Simplified necessary and sufficient conditions}
\label{section:simplified} 
%====================================================================

In order to exploit Theorem~\ref{thm:kriging_general} 
to check if two models provide uniformly asymptotically 
equivalent linear predictions, 
one has to verify 
Assumptions~\ref{ass:kriging}.I--III. 
Depending on the form of the covariance operators, 
this may be difficult. 
In this section we provide equivalent formulations of 
Assumptions~\ref{ass:kriging}.I and~III    
for two important cases: 
\begin{enumerate*}[label=\arabic*.]  
\item the two covariance operators 
diagonalize with respect to the same eigenbasis, 
and   
\item $\varrho,\widetilde{\varrho}\from\cX\times\cX\to\bbR$ 
are covariance functions  
of weakly stationary random fields on $\cX\subset\bbR^d$, 
a-priori defined on all of $\bbR^d$ and with spectral densities 
$f,\widetilde{f}\from\bbR^d\to[0,\infty)$. 
\end{enumerate*} 

%====================================================================
\subsection{Common eigenbasis}
%====================================================================

In the case that the two covariance operators 
diagonalize with respect to the same eigenbasis, 
conditions~I and~III of Assumption~\ref{ass:kriging} 
can be formulated as conditions on  
the ratios of the eigenvalues. 
We consider this scenario in the next corollary. 

\begin{corollary}\label{cor:equal_ef}
Suppose that $\cC,\CC$ are self-adjoint, positive definite, 
compact operators on $L_2(\cX,\nu_\cX)$ which 
diagonalize with respect to the same orthonormal basis  
$\{e_j\}_{j\in\bbN}$ for $L_2(\cX,\nu_\cX)$, i.e., 
there exist corresponding eigenvalues 
$\gamma_j, \widetilde{\gamma}_j \in (0,\infty)$, $j\in\bbN$, 
accumulating only at zero such that 
$\cC e_j = \gamma_j e_j$ and $\CC e_j = \widetilde{\gamma}_j e_j$ 
for all $j\in\bbN$. 
Then Assumptions~\ref{ass:kriging}.I and~\ref{ass:kriging}.III 
are satisfied if and only if 
there exists   
$a\in(0,\infty)$ 
such that $\lim_{j\to\infty} 
\widetilde{\gamma}_j / \gamma_j = a$. 
\end{corollary}

\begin{proof} 
We start by showing that 
$\lim_{j\to\infty} 
\widetilde{\gamma}_j / \gamma_j = a\in(0,\infty)$
is sufficient for 
Assumptions~\ref{ass:kriging}.I and~\ref{ass:kriging}.III. 
By Proposition~\ref{prop:AssI}, 
Assumption~\ref{ass:kriging}.I is equivalent to 
requiring that $\CC^{1/2} \cC^{-1/2}$ is 
an isomorphism on $L_2$.  
If $\cC$ and $\CC$ admit the 
same eigenbasis~$\{e_j\}_{j\in\bbN}$, 
then these are also eigenvectors of 
the self-adjoint, positive definite linear   
operator $\cC^{-1/2} \CC \cC^{-1/2}$
with corresponding eigenvalues 
$\{\widetilde{\gamma}_j / \gamma_j\}_{j\in\bbN}$. 
By assumption  
this sequence 
converges. Hence,  
$\norm{\CC^{1/2} \cC^{-1/2}}{\cL(L_2)}^2 
= 
\sup_{\norm{v}{L_2}=1} 
\scalar{\cC^{-1/2} \CC \cC^{-1/2} v, v}{L_2} 
= 
\sup_{j\in\bbN}  \widetilde{\gamma}_j / \gamma_j  
\in(0,\infty)$ 
follows,  
and 
$\lim_{j\to\infty} \gamma_j / \widetilde{\gamma}_j  = 1/a$ 
implies that 
$\norm{\cC^{1/2} \CC^{-1/2}}{\cL(L_2)}^2 
=\sup_{j\in\bbN} \gamma_j / \widetilde{\gamma}_j 
\in(0,\infty)$  
by the same argument. 
Thus, 
Assumption~\ref{ass:kriging}.I 
is satisfied. 
Furthermore, also Assumption~\ref{ass:kriging}.III 
follows, since $T_a = \cC^{-1/2} \CC \cC^{-1/2} - a\cI$
diagonalizes with respect to $\{e_j\}_{j\in\bbN}$ 
with corresponding eigenvalues 
$\{\widetilde{\gamma}_j / \gamma_j - a\}_{j\in\bbN}$ 
which by assumption accumulate only   
at zero and, hence, $T_a$ is compact 
on $L_2(\cX,\nu_\cX)$. 
Conversely, if Assumptions~\ref{ass:kriging}.I 
and~\ref{ass:kriging}.III are satisfied, then 
by the latter there exists $a\in(0,\infty)$
such that $T_a$ is compact and 
$\{\widetilde{\gamma}_j / \gamma_j - a\}_{j\in\bbN}$ 
is a null sequence, i.e., 
$\{\widetilde{\gamma}_j / \gamma_j\}_{j\in\bbN}$ 
converges to $a\in(0,\infty)$. 
\end{proof}

%====================================================================
\subsection{Weakly stationary random fields}\label{subsec:spectral-density} 
%====================================================================

We consider 
a connected, compact subset $\cX$ of $\bbR^d$ 
equipped with the Euclidean metric and the Lebesgue measure~$\lambda_d$. 
For brevity we omit $\lambda_d$ 
in the notations $L_2(\cX), L_2(\bbR^d)$. 
We assume that 
the operators $\cC,\CC\from L_2(\cX) \to L_2(\cX)$~are 
induced by continuous, (strictly) positive definite kernels  
$\varrho|_{\cX\times\cX}$ and $\widetilde{\varrho}|_{\cX\times\cX}$
which are restrictions of 
translation invariant covariance functions 
$\varrho, \widetilde{\varrho}\from \bbR^d \times \bbR^d \to \bbR$. 
Translation invariance of  $\varrho$ implies that there 
exists an even 
function $\varrho_0 \from \bbR^d \to \bbR$ such 
that $\varrho(x,x') = \varrho_0(x-x')$ for all $x,x'\in\bbR^d$, 
and similarly~$\widetilde{\varrho}_0$ is defined for $\widetilde{\varrho}$. 
We assume $\varrho_0, \widetilde{\varrho}_0 \in L_1(\bbR^d)$,  
such that the corresponding 
spectral densities $f,\widetilde{f}\from\bbR^d\to[0,\infty)$ 
exist. 
Recall that the spectral density $f$ 
and $\varrho_0$ relate via 
the inversion formula (see e.g.\ \cite[][p.~25]{stein99}): 
For all $\omega\in\bbR^d$,  
we have 
\begin{equation}\label{eq:rho0-f} 
f(\omega) 
= 
\frac{1}{(2\pi)^d}
(\cF\varrho_0)(\omega), 
\qquad 
(\cF\varrho_0)(\omega) 
:=  
\int_{\bbR^d} 
\exp(-i \omega\cdot x) 
\varrho_0(x) \, \rd x. 
\end{equation}
Using this convention for the Fourier transform $\cF$, 
its inverse becomes 
\[
\bigl(\cF^{-1} \hat{v} \bigr)(x) 
= 
\frac{1}{(2\pi)^d} 
\int_{\bbR^d} 
\exp(i \omega\cdot x) 
\hat{v}(\omega) \, \rd \omega, 
\qquad 
x\in\bbR^d. 
\]
Let the linear operator 
$\cF_\cX \from L_2(\cX) \to L_2(\bbR^d;\bbC)$ be 
the composition 
$\cF_\cX := \cF \circ E^0_\cX$, 
where $E_\cX^0$ is the zero extension 
$L_2(\cX) \ni w \mapsto E_\cX^0 w \in L_1(\bbR^d)\cap L_2(\bbR^d)$ 
that sets  
$(E^0_\cX w)(x) = 0$ for all $x\in\bbR^d\setminus\cX$. 
We then consider the following subset 
of the space of complex-valued square-integrable functions $L_2(\bbR^d;\bbC)$,  
which itself is a
vector space over $\bbR$, 
\[ 
\cF_\cX(L_2(\cX)) 
= 
\bigl\{ \hat{w} \from\bbR^d \to \bbC  \;\big|\; \exists w \in L_2(\cX) : \hat{w} = \cF_\cX w \bigr\} 
\subset L_2(\bbR^d;\bbC), 
\] 
and define the Hilbert space 
$H_f$ (over $\bbR$) as the closure of $\cF_\cX(L_2(\cX))$ 
with respect to norm induced by the 
weighted $L_2(\bbR^d;\bbC)$-inner product 
with weight $f$, 
\begin{equation}\label{eq:def:Hf}    
\scalar{\hat{v}_1, \hat{v}_2 }{H_f} 
:= 
\int_{\bbR^d} f(\omega) \hat{v}_1(\omega) \overline{\hat{v}_2(\omega)} \, \rd \omega, 
\qquad 
H_f := \overline{\cF_\cX(L_2(\cX))}^{\norm{\,\cdot\,}{H_f}}. 
\end{equation} 
We recall the Hilbert space~$H=\cC^{-1/2}(L_2(\cX))$ 
with inner product $\scalar{\,\cdot\,, \,\cdot\,}{H} = \scalar{\cC\,\cdot\,, \,\cdot\,}{L_2(\cX)}$ 
from Lemma~\ref{lem:cH-ONB}\ref{lem:cH-ONB-i}
and find by invoking \eqref{eq:rho0-f} that, for all $v_1,v_2\in L_2(\cX)$,  
\begin{align} 
( \cF_\cX v_1, \cF_\cX v_2 )_{H_f} 
& = 
\frac{1}{(2\pi)^d} 
\int_{\bbR^d} (\cF \varrho_0)(\omega) 
( \cF_\cX v_1 )(\omega) \overline{ ( \cF_\cX v_2 )(\omega)} \, \rd \omega 
\notag 
\\
& = 
\frac{1}{(2\pi)^d} 
\int_{\bbR^d} \cF \bigl(\varrho_0 * \bigl( E^0_\cX v_1 \bigr)\bigr) (\omega) 
\overline{ \cF \bigl(E^0_\cX v_2\bigr) (\omega)} \, \rd \omega 
\notag 
\\
& = 
\bigl(\varrho_0 * \bigl( E^0_\cX v_1 \bigr), E^0_\cX v_2 \bigr)_{L_2(\bbR^d)} 
= 
\scalar{\cC v_1, v_2}{L_2(\cX)} 
= 
\scalar{ v_1, v_2}{H}. 
\label{eq:Hf-H-ip} 
\end{align} 
By density of $L_2(\cX)$ in $H$ and of $\cF_\cX(L_2(\cX))$ in $H_f$, 
$\cF_\cX$ thus admits a unique continuous linear extension to  
an inner product preserving isometric isomorphism 
between $H$ and~$H_f$. 
Its inverse $\cF_\cX^{-1}\from H_f \to H$ is  
the unique continuous linear extension of 
${R_\cX\circ \cF^{-1}\from}$ $\cF_\cX (L_2(\cX)) \to L_2(\cX)\subset H$, 
where $R_\cX\from L_2(\bbR^d) \to L_2(\cX)$ denotes the restriction 
to~$\cX$. 

\begin{proposition}\label{prop:spectral-density} 
Suppose that the 
self-adjoint, positive definite, compact 
operators 
$\cC,\CC\from L_2(\cX) \to L_2(\cX)$  
are induced by restrictions (to $\cX\times\cX$) of 
translation invariant covariance functions 
$\varrho, \widetilde{\varrho}\from \bbR^d \times \bbR^d \to \bbR$, 
which have spectral densities 
$f, \widetilde{f}\from\bbR^d\to[0,\infty)$ 
defined via \eqref{eq:rho0-f}. 
Then, Assumptions~\ref{ass:kriging}.I and~\ref{ass:kriging}.III 
are satisfied if and only if: 
\begin{enumerate}[label={\normalfont(\Roman*')}, labelsep=7mm, topsep=4pt]
	\item\label{prop:spectral-density-i}\hspace*{-3.4mm} 
	The spaces $H_f$ and $H_{\widetilde{f}}$ 
	are isomorphic with equivalent norms, i.e., 
	there exist constants $0<k\leq K <\infty$ such that 
	\begin{equation}\label{eq:spectral-density-i}  
		k \norm{\hat{v}}{H_f}^2 \leq 
		\int_{\bbR^d} \widetilde{f}(\omega) |\hat{v}(\omega)|^2 \, \rd \omega 
		\leq K \norm{\hat{v}}{H_f}^2 
		\quad  
		\forall \hat{v}\in \cF_\cX(L_2(\cX)). 
	\end{equation} 
	\setcounter{enumi}{2}
	\item\label{prop:spectral-density-ii}\hspace*{-6.1mm}  
	There exists  
	$a\in(0,\infty)$ 
	such that the linear operator $\widehat{T}_a := S - a\cI_{H_f}$ 
	is compact on~$H_f$, where $\cI_{H_f}$ denotes the identity on $H_f$ and 
	$S\from H_f \to H_f$ is defined by 
	\begin{equation}\label{eq:spectral-density-ii}  
		\scalar{S \hat{v}_1, \hat{v}_2}{H_f}  
		=
		\int_{\bbR^d} \widetilde{f}(\omega) \hat{v}_1(\omega) \overline{\hat{v}_2(\omega)} \, \rd \omega 
		\qquad 
		\forall \hat{v}_1, \hat{v}_2\in H_f. 
	\end{equation}
\end{enumerate} 
\end{proposition} 

\begin{proof} 
\ref{prop:spectral-density-i}
Let $\hat{v}_1, \hat{v}_2 \in \cF_\cX(L_2(\cX))$ 
and $v_1, v_2\in L_2(\cX)$ be such that $\hat{v}_1 = \cF_\cX v_1$ 
and $\hat{v}_2 = \cF_\cX v_2$.  
Applying the inversion formula \eqref{eq:rho0-f} 
for $\widetilde{f}$ gives, 
similarly as in~\eqref{eq:Hf-H-ip}, 
\begin{equation}\label{eq:Ctilde-ftilde}  
	\scalar{\CC v_1, v_2}{L_2(\cX)} 
	= 
	\int_{\bbR^d} 
	\frac{( \cF \widetilde{\varrho}_0 )(\omega )}{(2\pi)^d} 
	\hat{v}_1(\omega) 
	\overline{\hat{v}_2 (\omega)}  \, \rd \omega 
	=
	\int_{\bbR^d} \widetilde{f}(\omega ) \hat{v}_1(\omega) 
	\overline{\hat{v}_2 (\omega)}  \, \rd \omega . 
\end{equation} 	
Therefore, 
\eqref{eq:spectral-density-i} is equivalent to 
$k \scalar{\cC v, v}{L_2(\cX)}
\leq 
\scalar{\CC v,v}{L_2(\cX)}
\leq 
K \scalar{\cC v, v}{L_2(\cX)}$
holding for all $v\in L_2(\cX)$, 
which by density of $L_2(\cX)$ in $H$ can be reformulated as 
the relation  
$\norm{ \CC^{1/2}\cC^{-1/2} w}{L_2(\cX)}^2 \in[k, K]$ 
for all $w\in L_2(\cX)$ with $\norm{w}{L_2(\cX)}=1$, i.e., 
$\CC^{1/2}\cC^{-1/2}$ is an isomorphism on $L_2(\cX)$.  
By Proposition~\ref{prop:AssI}\ref{prop:AssI-i}$\,\Leftrightarrow\,$\ref{prop:AssI-ii}
this is equivalent to Assumption~\ref{ass:kriging}.I. 

\ref{prop:spectral-density-ii} 
We proceed as illustrated below: We prove that 
$T_a$ can be expressed as the composition 
$T_a=\cC^{1/2} \cF_\cX^{-1} \widehat{T}_a \cF_\cX \cC^{-1/2}$. 
Since $\cC^{-1/2}\from L_2(\cX) \to H$ and 
$\cF_\cX\from H \to H_f$ are inner product preserving 
isometric isomorphisms, this shows that  
$T_a\in\cK( L_2(\cX) )$ is equivalent to compactness 
of $\widehat{T}_a$ on $H_f$. 

\begin{center}
	{\footnotesize   
		\begin{tikzpicture}[scale=1] 
		% nodes
		\node (L2l) at (0, 1.5) {$\bigl( L_2(\cX), \scalar{\,\cdot\,, \,\cdot\,}{L_2(\cX)}\bigr)$};
		\node (L2r) at (7, 1.5) {$\bigl( L_2(\cX), \scalar{\,\cdot\,, \,\cdot\,}{L_2(\cX)}\bigr)$};
		\node (Hl) at (0, 0) {$\bigl( H , \scalar{\cC \,\cdot\,, \,\cdot\,}{L_2(\cX)} \bigr)$};
		\node (Hr) at (7, 0) {$\bigl( H , \scalar{\cC \,\cdot\,, \,\cdot\,}{L_2(\cX)} \bigr)$}; 
		\node (FHl) at (0, -1.5) {$\bigl( H_f , \scalar{\,\cdot\,, \,\cdot\,}{H_f} \bigr)$};
		\node (FHr) at (7, -1.5) {$\bigl( H_f , \scalar{\,\cdot\,, \,\cdot\,}{H_f} \bigr)$};
		
		% arrows
		\draw[->, thick] (L2l) -- (L2r) node[midway,above] {$T_a=\cC^{-1/2}\CC\cC^{-1/2} - a \cI$}; 
		\draw[->, thick] (Hl) -- (Hr) node[midway,above] {$\cC^{-1}\CC - a \cI_{H}$}; 
		\draw[->, thick] (FHl) -- (FHr) node[midway,above] {$\widehat{T}_a = S - a \cI_{H_f}$};  
		\draw[->, thick] (L2l) -- (Hl) node[midway,left] {$\cC^{-1/2}$}; 
		\draw[->, thick] (Hr) -- (L2r) node[midway,right] {$\cC^{1/2}$}; 
		\draw[->, thick] (Hl) -- (FHl) node[midway,left] {$\cF \circ E_\cX^0$}; 
		\draw[->, thick] (FHr) -- (Hr) node[midway,right] {$R_\cX \circ \cF^{-1}$}; 
		\end{tikzpicture}} 
\end{center}

\noindent Part \ref{prop:spectral-density-i} implies that 
$\mathfrak{b}(\hat{v}_1, \hat{v}_2) 
:= 
\int_{\bbR^d} 
\widetilde{f}(\omega) \hat{v}_1(\omega) \overline{\hat{v}_2(\omega)} 
\, \rd \omega$ 
defines a continuous, coercive bilinear form 
on the real Hilbert space $H_f$. 
Thus, for every $\hat{v}_1\in H_f$, 
existence and uniqueness of $S \hat{v}_1$ 
satisfying \eqref{eq:spectral-density-ii} 
follows from the Riesz representation theorem, 
and $S\from H_f\to H_f$ is well-defined, linear and bounded. 
For $v_1,v_2\in L_2(\cX)$ and  
$\hat{v}_1 := \cF_\cX v_1$, $\hat{v}_2 := \cF_\cX v_2$, 
we have   
\begin{align*} 
	\scalar{ (\cC^{-1}\CC - a\cI_H) v_1,  v_2}{H} 
	&= 
	\scalar{(\CC-a\cC) v_1, v_2}{L_2(\cX)} 
	= 
	\int_{\bbR^d} (\widetilde{f}(\omega)  - a f(\omega)) 
	\hat{v}_1(\omega) \overline{\hat{v}_2(\omega)} \, \rd \omega  
	\\ 
	&=  
	\scalar{ (S - a \cI_{H_f}) \hat{v}_1, \hat{v}_2}{H_f}
	= 
	\scalar{ \widehat{T}_a \hat{v}_1, \hat{v}_2}{H_f} 
	= 
	\scalar{ \cF_\cX^{-1} \widehat{T}_a \cF_\cX v_1, v_2}{H}, 
\end{align*} 
where we used \eqref{eq:def:Hf}, 
\eqref{eq:Hf-H-ip} and \eqref{eq:Ctilde-ftilde}. 
By density of $L_2(\cX)$ in $H$ and continuity 
of $\cF_\cX^{-1} \widehat{T}_a \cF_\cX$ on $H$, 
this equality holds also for all 
$v_1,v_2\in H$. 
Consequently, we obtain the chain of identities  
$T_a = \cC^{-1/2} \CC \cC^{-1/2} - a\cI 
= \cC^{1/2} \bigl(  \cC^{-1} \CC - a \cI_H \bigr)\cC^{-1/2} 
= \cC^{1/2} \cF_\cX^{-1}\widehat{T}_a\cF_\cX \cC^{-1/2}$. 
\end{proof} 

\begin{remark}
We emphasize that the Hilbert spaces $\cH^0$ in \eqref{eq:def:cH0}, 
$H$ from Lemma~\ref{lem:cH-ONB}\ref{lem:cH-ONB-i} 
and $H_f$ in \eqref{eq:def:Hf}  
are mutually isomorphic, with inner product preserving isomorphisms 
$\cJ \from H \to \cH^0$ and $\cF_\cX\from H \to H_f$. 
\end{remark} 

For two continuous functions 
$g, \widetilde{g} \from \bbR^d \to [0,\infty)$, 
the notation $g \asymp \widetilde{g}$ indicates that 
there exist $k,K\in(0,\infty)$ such that 
the relations 
$k g(\omega) \leq \widetilde{g}(\omega) \leq K g(\omega)$ 
hold for all $\omega\in\bbR^d$. 

\begin{corollary}\label{cor:spectral-density-I} 
Suppose the setting of Proposition~\ref{prop:spectral-density}.
\begin{enumerate}[label={\normalfont(\roman*')}, labelsep=6mm, topsep=4pt]
	\item\label{cor:spectral-density-I-pos}\hspace*{-3.1mm} 
	Assumption~\ref{ass:kriging}.I is satisfied 
	whenever $f\asymp \widetilde{f}$.  
	\item\label{cor:spectral-density-I-neg}\hspace*{-4.5mm}    
	Suppose that  
	$\varrho_0\from\bbR^d\to \bbR$ 
	related to $f\from\bbR^d\to[0,\infty)$ 
	via \eqref{eq:rho0-f} 
	is not infinitely differentiable 
	in at least one 
	Cartesian coordinate direction.  
	Then, in either of the cases 
	$\tfrac{ \widetilde{f}(\omega) }{ f(\omega) } \rightarrow 0$ 
	or 
	$\tfrac{\widetilde{f}(\omega)}{f(\omega)} \rightarrow \infty$ 
	as $\norm{\omega}{\bbR^d}\to\infty$,  
	Assumption~\ref{ass:kriging}.I is not satisfied. 
\end{enumerate}
\end{corollary}

\begin{proof} 
\ref{cor:spectral-density-I-pos} 
Clearly, the relation $f\asymp \widetilde{f}$ in \ref{cor:spectral-density-I-pos} 
implies that \eqref{eq:spectral-density-i} holds. 
Therefore, 
by Proposition~\ref{prop:spectral-density}\ref{prop:spectral-density-i} 
Assumption~\ref{ass:kriging}.I is satisfied. 

\ref{cor:spectral-density-I-neg}  
Without loss of generality, we may assume that $\cX$ 
contains an open subset containing the origin and 
pick $L\in(0,\infty)$ with ${[-L^{-1}, L^{-1}]^d \subseteq \cX}$. 
By assumption on the differentiability of $\varrho_0$, 
there exist $j \in \{1,\ldots,d\}$ and $p\in\bbN$ such that 
$\int_{\bbR} f_j(\omega_j) \omega_j^{2p}  \,\rd \omega_j 
= \infty$, 
where 
\[ 
f_j\from\bbR\to[0,\infty), 
\;\;\, 
f_j(\omega_j) 
:= 
\int_{\bbR^{d-1}}  
f(\omega_1, \ldots, \omega_j, \ldots, \omega_d) \, 
\rd \lambda_{d-1}(\omega_1,\ldots, \omega_{j-1}, \omega_{j+1},\ldots, \omega_d). 
\]
Let $r(\,\cdot\,)$ be the rectangular function, 
defined as $r(x) = 1$ for $|x| \leq \tfrac{1}{2}$ 
and $r(x) = 0$ for~$|x| > \tfrac{1}{2}$. 
For $n\in\bbN$ such that $n>pL$, 
consider the forward difference operator 
$\Delta_n\from L_2(\bbR) \to L_2(\bbR)$, 
$\Delta_n g(x) := n(g(x+1/n)-g(x))$, 
and set  
\[
\textstyle 
v_n(x) 
:= 
n^d 
\Delta_n^p \, r\bigl(n x_j - \tfrac{1}{2} \bigr) 
\prod\limits_{k \neq j} r\bigl( n x_k - \tfrac{1}{2} \bigr), 
\qquad 
n\in\bbN, 
\quad 
n > pL. 
\] 
Each function $v_n$ 
has compact support 
in $[-L^{-1}, L^{-1}]^d \subseteq \cX$ 
and $v_n\in L_2(\cX)$. 
Furthermore, its Fourier transform is
$\hat{v}_n(\omega) 
:= 
( \cF_\cX v_n )(\omega)
= 
n^p \bigl( e^{i\omega_j/n} - 1 \bigr)^p 
\prod_{k=1}^d e^{-i\omega_k/2}\sinc\bigl( \frac{\omega_k}{2\pi n} \bigr)$, 
where $\sinc(x) := \frac{\sin(\pi x)}{\pi x}$ for all $x\in\bbR$, 
and by basic trigonometric identities 
\begin{equation}\label{eq:cor:spectral-density-I-neg} 
	\textstyle 
	|\hat{v}_n(\omega)|^2 
	= 
	\left[ 2n \sin\bigl(\frac{\omega_j}{2n}\bigr) \right]^{2p}
	\prod\limits_{k=1}^d  \sinc^2 \bigl( \frac{\omega_k}{2\pi n} \bigr) 
\end{equation} 
follows. 
Fix $\ell\in\bbN$. 
By assumption there exists a constant $M_\ell \in (0,\infty)$ 
such that one of the following holds 
for all $\omega$ with $\|\omega\|_{\bbR^d} > M_\ell$: 
\textbf{a)} 
$\widetilde{f}(\omega) < \tfrac{1}{2\ell} f(\omega)$ 
or \textbf{b)} 
$\widetilde{f}(\omega) > 2\ell \, f(\omega)$. 
Next, define the infinite strip 
$A_\ell := \{ \omega\in\bbR^d : |\omega_j| \leq M_\ell \}$. 
Then, for every $n > pL$,  
we find by \eqref{eq:cor:spectral-density-I-neg} 
\[ 
\int_{A_\ell} f(\omega) |\hat{v}_n(\omega)|^2  \, \rd \omega  
\leq 
\int_{A_\ell}  
f(\omega)
\omega_j^{2p} \,\rd \omega 
\leq 
M_\ell^{2p} \int_{\bbR^d} 	f(\omega) \,\rd \omega  
= 
M_\ell^{2p} \varrho_0(0), 
\]
since 
$2n\sin\bigl( \frac{\omega_j}{2n} \bigr) 
=
\sinc\bigl( \frac{\omega_j}{2\pi n} \bigr) 
\omega_j$ 
and $|\sinc(x)|\leq 1$ for all $x\in\bbR$. 
By the same arguments, 
$\int_{A_\ell} \widetilde{f}(\omega) |\hat{v}_n(\omega)|^2 \, \rd \omega 
\leq M_\ell^{2p} \widetilde{\varrho}_0(0)$ 
holds for all $n > pL$. 
In addition, for every 
$n\in\bbN$ with $n > \max\{ pL, M_\ell/\pi \}$, define the set 
\[ 
B_{\ell}^{n}  
:= 
\bigl\{ \omega \in \bbR^d : 
M_\ell < |\omega_j| < n \pi, \
|\omega_k| < n \pi \;\; \forall k\neq j \bigr\} 
\subset 
A_\ell^c 
:= 
\bbR^d\setminus A_\ell. 
\] 
Since
$\sinc^2(\theta/(2\pi n)) > (2/\pi)^2$ 
for $\theta\in(- \pi n, \pi n)$, 
we obtain again by~\eqref{eq:cor:spectral-density-I-neg} that 
\begin{align*}
	\int_{B_\ell^n} f(\omega)  |\hat{v}_n(\omega)|^2  \, \rd \omega 
	>  
	\frac{ 4^{d} }{ \pi^{2d} }  
	\int_{B_\ell^n} f(\omega)  
	\left[2n \sin\left(\frac{\omega_j}{2n}\right) \right]^{2p} 
	\, \rd \omega 
	>  
	\frac{ 4^{d+p} }{ \pi^{2(d+p)} }  
	\int_{B_\ell^n} f(\omega) \omega_j^{2p} \, \rd \omega .   
\end{align*} 
Furthermore, note that 
$\lim_{n\to\infty} 
\int_{B_\ell^n} f(\omega) \omega_j^{2p} \, \rd \omega
=
\int_{A_\ell^c} f(\omega) \omega_j^{2p} \, \rd \omega 
=\infty$, 
since the integral   
$\int_{\bbR^d} f(\omega) \omega_j^{2p} \, \rd \omega 
=
\int_{\bbR} f_j(\omega_j) \omega_j^{2p} \, \rd \omega_j 
= \infty$ 
diverges and  
$\int_{A_\ell} f(\omega) \omega_j^{2p} \, \rd \omega 
\leq M_\ell^{2p} \varrho_0(0)$ is finite.
For this reason, there exists an integer
$n_0 = n_0(\ell) > \max\{ pL, M_\ell/\pi \}$ 
such that 
\[  
\int_{ A_\ell^c } 
f(\omega) 
|\hat{v}_n(\omega)|^2 \, \rd \omega  
\geq 
\int_{B_\ell^n} f(\omega)  |\hat{v}_n(\omega)|^2  \, \rd \omega 
> 
2\ell M_\ell^{2p} \max\{ \varrho_0(0), \widetilde{\varrho}_0(0)\} , 
\] 
for all $n\geq n_0$. 
We then obtain, for every $n\geq n_0$, 
in case \textbf{a)} the estimate 
\[
\frac{\int_{\bbR^d}\widetilde{f}(\omega)|\hat{v}_n(\omega)|^2\,\rd\omega}{ 
	\norm{\hat{v}_n}{H_f}^2}
\leq 
\frac{\int_{A_\ell} \widetilde{f}(\omega)|\hat{v}_n(\omega)|^2\,\rd\omega}{ 
	\int_{ A_\ell^c } f(\omega)|\hat{v}_n(\omega)|^2\,\rd\omega}  
+
\frac{
	\int_{ A_\ell^c } \widetilde{f}(\omega)|\hat{v}_n(\omega)|^2\,\rd\omega}{ 
	\int_{ A_\ell^c } f(\omega)|\hat{v}_n(\omega)|^2\,\rd\omega}  
< \ell^{-1}, 
\]
and in case \textbf{b)} we have, 
for all $n\geq n_0$,  
\begin{align*}
	&\frac{\int_{\bbR^d}\widetilde{f}(\omega)|\hat{v}_n(\omega)|^2\,\rd\omega}{ 
		\norm{\hat{v}_n}{H_f}^2}
	\geq
	\frac{\int_{ A_\ell^c }
		\widetilde{f}(\omega)|\hat{v}_n(\omega)|^2\,\rd\omega}{ 
		\int_{ A_\ell } f(\omega)|\hat{v}_n(\omega)|^2\,\rd\omega +
		\int_{ A_\ell^c } f(\omega)|\hat{v}_n(\omega)|^2\,\rd\omega}  
	\\
	&\qquad\qquad 
	=  
	\frac{\int_{ A_\ell^c }
		\widetilde{f}(\omega)|\hat{v}_n(\omega)|^2\,\rd\omega}{ 
		\int_{ A_\ell^c } f(\omega)|\hat{v}_n(\omega)|^2\,\rd\omega}  
	\Biggl( 	\frac{ \int_{A_\ell} f(\omega)|\hat{v}_n(\omega)|^2\,\rd\omega}{ 
		\int_{ A_\ell^c } f(\omega)|\hat{v}_n(\omega)|^2\,\rd\omega}  + 1 \Biggr)^{-1}
	\geq 
	2\ell  \bigl( \tfrac{1}{2\ell}+1 \bigr)^{-1}  
	\geq 
	\ell. 
\end{align*}
Since $\ell\in\bbN$ was arbitrary, in case \textbf{a)} 
there is no constant $k\in(0,\infty)$ 
such that the lower bound in \eqref{eq:spectral-density-i} holds 
and 
in case \textbf{b)} 
we cannot find $K\in(0,\infty)$ 
for the upper bound in \eqref{eq:spectral-density-i}.  
Thus, the result follows by 
Proposition~\ref{prop:spectral-density}\ref{prop:spectral-density-i}.  
\end{proof}

In what follows, we let $\cW_{\Pi_{\mathbf{R}}}$ denote the class 
of Fourier transforms $\cF v$ of square-integrable 
functions $v\in L_2(\bbR^d;\bbC)$ with support 
$\operatorname{supp}(v) \subseteq \Pi_{\mathbf{R}}$ 
inside the bounded parallelepiped  
$\Pi_{\mathbf{R}} = \{x\in \bbR^d : -R_j \leq x_j \leq R_j, \, j=1,\ldots,d \}$. 

\begin{corollary}\label{cor:spectral-density-III} 
Suppose the setting of Proposition~\ref{prop:spectral-density},  
$f\asymp \widetilde{f}$, 
and furthermore that there exists $\varphi_0\in\cW_{\Pi_{\mathbf{R}}}$ 
such that 
$f\asymp |\varphi_0|^2$. 
Then Assumptions~\ref{ass:kriging}.I and~\ref{ass:kriging}.III are satisfied 
whenever 
there exists a constant $a\in(0,\infty)$ 
such that 
$\frac{\widetilde{f}(\omega)}{f(\omega)} \to a$ as 
$\| \omega\|_{\bbR^d} \to \infty$. 
\end{corollary} 

\begin{proof} 
By Corollary~\ref{cor:spectral-density-I}\ref{cor:spectral-density-I-pos}  
$f\asymp \widetilde{f}$ implies that Assumption~\ref{ass:kriging}.I 
holds. 

Next, recall the bounded linear operator $S\from H_f \to H_f$  
from \eqref{eq:spectral-density-ii} and, for $\ell\in\bbN$, 
define 
the self-adjoint linear 
operator $\widehat{T}_a^\ell\from H_f \to H_f$ 
similarly via 
\[ 
\bigl( \widehat{T}_a^\ell \hat{v}_1, \hat{v}_2 \bigr)_{H_f}  
=
\int_{B_\ell} 
\bigl( \widetilde{f}(\omega) - a f(\omega) \bigr) 
\hat{v}_1(\omega) \overline{\hat{v}_2(\omega)} \, \rd \omega 
\quad 
\forall \hat{v}_1, \hat{v}_2\in H_f, 
\]
where 
$B_\ell := \{ w\in\bbR^d: \| \omega\|_{\bbR^d} < \ell \}$ 
is the 
ball around the origin with radius~$\ell$. 
By Proposition~\ref{prop:spectral-density}\ref{prop:spectral-density-i} 
$\widehat{T}_a^\ell$ is bounded with 
$\norm{\widehat{T}_a^\ell}{\cL(H_f)}\leq a + \norm{S}{\cL(H_f)}$ 
for all $\ell\in\bbN$. 
We now proceed in two steps: 
We first show that, for every $\ell\in\bbN$, 
$\widehat{T}_a^\ell$ is compact on $H_f$. 
Secondly, 
we prove convergence  
$\lim_{\ell\to\infty} \| \widehat{T}_a^\ell - \widehat{T}_a \|_{\cL(H_f)} = 0$,  
which implies 
that $\widehat{T}_a = S - a\cI_{H_f}$ is compact on~$H_f$, 
since $\cK(H_f)$ is closed in $\cL(H_f)$. 
Then, 
Assumption~\ref{ass:kriging}.III 
holds 
by Proposition~\ref{prop:spectral-density}\ref{prop:spectral-density-ii}. 

By Proposition~\ref{prop:spectral-density}\ref{prop:spectral-density-i} 
$H_f$ and $H_{|\varphi_0|^2}$ are isomorphic 
and  
$c \| \hat{v} \|_{H_f}^2 
\leq  
\| \hat{v} \|_{H_{|\varphi_0|^2}}^2 
\leq  
C  \| \hat{v} \|_{H_f}^2$ 
for some constants $c,C\in(0,\infty)$ 
independent of $\hat{v}\in H_f$.  
For this reason, the operator 
$\widehat{T}_a^\ell$ is compact on $H_f$ 
if and only if  
it is compact on $H_{|\varphi_0|^2}$. 
To see that, for every $\ell\in\bbN$, the operator 
$\widehat{T}_a^\ell$ is compact on $H_{|\varphi_0|^2}$, 
we prove the stronger result that 
$\widehat{T}_a^\ell$ is Hilbert--Schmidt~on~$H_{|\varphi_0|^2}$. 
Let $\{ \hat{v}_j \}_{j\in\bbN}$ 
be an orthonormal basis 
for $H_{|\varphi_0|^2}$. 
Then, we note that the relations 
$f\asymp \widetilde{f}$ and $f\asymp |\varphi_0|^2$ 
imply equality of the supports,  
$\operatorname{supp}(\widetilde{f}) 
= 
\operatorname{supp}(f) 
= 
\operatorname{supp}(\varphi_0)$, 
and estimate 
\begin{align*} 
	\textstyle 
	C^{-1} 
	\sum\limits_{j\in\bbN} 
	\| \widehat{T}_a^\ell \hat{v}_j 
	&\textstyle  
	\|_{H_{|\varphi_0|^2}}^2 
	\leq   
	\sum\limits_{j\in\bbN} 
	\| \widehat{T}_a^\ell \hat{v}_j \|_{H_f}^2 
	=
	\sum\limits_{j \in\bbN} \, 
	\sup\limits_{ \| \hat{w} \|_{H_f}=1} 
	\bigl| \scalar{ \widehat{T}_a^\ell \hat{v}_j, \hat{w} }{H_f} \bigr|^2 
	\\
	&\textstyle  
	= 
	\sum\limits_{j \in\bbN} \, 
	\sup\limits_{ \| \hat{w} \|_{H_f}=1} 
	\Bigl| 
	\int\nolimits_{ \operatorname{supp}(\varphi_0) \cap B_\ell }   
	\frac{\widetilde{f}(\omega) - a f(\omega)}{|\varphi_0(\omega)|^2} 
	|\varphi_0(\omega)|^2 
	\hat{v}_j (\omega) \overline{\hat{w} (\omega)} \, \rd \omega \Bigr|^2 
	\\
	&\textstyle 
	\leq 
	\sum\limits_{j \in\bbN}  \, 
	\sup\limits_{ \| \hat{w} \|_{H_f}=1} 
	\| \hat{w} \|_{H_{|\varphi_0|^2}}^2 
	\int_{  \operatorname{supp}(\varphi_0) \cap B_\ell }   
	\frac{| \widetilde{f}(\omega) - a f(\omega) |^2}{|\varphi_0(\omega)|^4} 
	|\varphi_0(\omega)|^2 
	|\hat{v}_j (\omega)|^2 \, \rd \omega  
	\\
	&\textstyle 
	\leq 
	C \, 
	\Bigl| 
	\sup\limits_{\omega\in\operatorname{supp}(\varphi_0)} 
	\frac{ \widetilde{f}(\omega) + a f(\omega) }{|\varphi_0(\omega)|^2} 
	\Bigr|^2  
	\int_{ \operatorname{supp}(\varphi_0) \cap B_\ell } 
	|\varphi_0(\omega)|^2 \sum\limits_{j\in\bbN} |\hat{v}_j(\omega)|^2 \, \rd \omega. 
\end{align*} 
Since $f\asymp |\varphi_0|^2$ and $f\asymp \widetilde{f}$, 
the supremum in this bound is finite. 
Furthermore, as $\varphi_0\in\cW_{\Pi_{\mathbf{R}}}$ 
by \cite[][Lemma on p.~34]{Yadrenko1973} we obtain the bound 
\[
\textstyle 
|\varphi_0(\omega)|^{2}
\sum\limits_{j\in\bbN} |\hat{v}_j(\omega)|^2 
\leq 
C_{\mathbf{R},\mathbf{K}}, 
\quad 
\text{where} 
\quad 
C_{\mathbf{R},\mathbf{K}} 
:= 
\prod_{j=1}^d \frac{(R_j + K_j)}{\pi }, 
\]
and $\Pi_{\mathbf{K}}$ is a parallelepiped 
enclosing the compact set $\cX\subseteq\Pi_{\mathbf{K}}$.  
Consequently, 
\[ 
\textstyle 
\sum\limits_{j \in \bbN}  
\| \widehat{T}_a^\ell \hat{v}_j \|_{H_{|\varphi_0|^2}}^2
\leq 
C^2 
(2\ell)^d \,
C_{\mathbf{R},\mathbf{K}} \, 
\Bigl| 
\sup\limits_{\omega\in\operatorname{supp}(\varphi_0)} 
\frac{ \widetilde{f}(\omega) + a f(\omega) }{|\varphi_0(\omega)|^2} 
\Bigr|^2 <\infty, 
\]
i.e., for every $\ell\in\bbN$, 
the operator $\widehat{T}_a^\ell$ is Hilbert--Schmidt on $H_{|\varphi_0|^2}$ 
and, thus, compact on $H_f$.

Finally, let $\eps\in(0,\infty)$ and $\ell_\eps\in\bbN$ be such that 
$\sup\nolimits_{\omega\in B_\ell^c } 
\Bigl| \frac{\widetilde{f}(\omega)}{f(\omega)} - a \Bigr| < \eps$ 
for all $\ell\geq\ell_\eps$, where $B_\ell^c := \bbR^d\setminus B_\ell$. 
Then, for every $\ell\geq\ell_\eps$ and all $\hat{v}_1, \hat{v}_2 \in H_f$,  
\[ 
\textstyle 
\bigl( \bigl( \widehat{T}_a - \widehat{T}^\ell_a \bigr) \hat{v}_1, \hat{v}_2 \bigr)_{H_f} 
= 
\int_{B_\ell^c} 
\bigl( \widetilde{f}(\omega) - a f(\omega) \bigr) 
\hat{v}_1(\omega) \overline{\hat{v}_2(\omega)} \, \rd \omega 
\leq 
\eps 
\norm{\hat{v}_1}{H_f} \norm{\hat{v}_2}{H_f}. 
\] 
Thus, $\lim_{\ell\to\infty} \| \widehat{T}^\ell_a - \widehat{T}_a \|_{\cL(H_f)} = 0$ 
and $\widehat{T}_a$ is compact on $H_f$. 
\end{proof}

%====================================================================
\section{Applications}
\label{section:applications}
%====================================================================

In the following we exemplify  
the results of Section~\ref{section:results} 
and Section~\ref{section:simplified} 
by three specific applications. 
Corollary~\ref{cor:spectral-density-I} and 
Corollary~\ref{cor:spectral-density-III}
can be used to check for  
uniformly asymptotically optimal linear prediction
in the case of weakly stationary processes 
on compact subsets of $\bbR^d$, 
using their spectral densities. 
As an explicit example we consider the Mat\'ern 
covariance family in Section~\ref{subsec:matern}. 
Corollary~\ref{cor:equal_ef} is applicable, e.g., to 
periodic random fields on $\cX=[0,1]^d$ 
as considered by Stein~\cite{stein97}, 
see Section~\ref{subsec:periodic}.  
Moreover, as Theorem~\ref{thm:kriging_general} 
it also holds for random fields 
on more general domains. 
As a further illustration we consider 
an application on the sphere~$\cX=\bbS^2$ 
in Section~\ref{subsec:sphere}.  

%====================================================================
\subsection{The Mat\'ern covariance family}
\label{subsec:matern}
%====================================================================

The Mat\'ern covariance function 
$\varrho|_{\cX\times\cX}$ on $\cX\subset\bbR^d$  
with parameters $\sigma,\nu,\kappa\in(0,\infty)$, 
see Example~\ref{ex:cov-functions}\ref{ex:cov-euclid},  
has the spectral density 
\begin{equation}\label{eq:matern-spectral-density} 
f_{\mathsf{M}}(\omega) 
= 
\frac{1}{(2\pi)^d}
(\cF\varrho_{\mathsf{M}})(\omega)
=
\frac{ \Gamma(\nu + d/2) }{ \Gamma(\nu) \pi^{d/2} } 
\frac{ \sigma^2 \kappa^{2\nu} }
{ ( \kappa^2 + \| \omega \|_{\bbR^d}^2 )^{\nu + d/2} }, 
\qquad  
\omega\in\bbR^d, 
\end{equation} 
cf.~\cite[][Equation~(32) on p.~49]{stein99}. 
Assume that $\widetilde{\varrho}$ is a further Mat\'ern covariance function 
with parameters $\widetilde{\sigma},\widetilde{\nu},\widetilde{\kappa}\in(0,\infty)$  
and corresponding spectral density $\widetilde{f}_{\mathsf{M}}$. 
Since   
\[
\frac{\widetilde{f}_{\mathsf{M}}(\omega)}{f_{\mathsf{M}}(\omega)} 
= 
\frac{ \Gamma(\widetilde{\nu} + d/2) }{ \Gamma(\widetilde{\nu}) } 
\frac{ \Gamma(\nu) }{ \Gamma(\nu + d/2) } 
\frac{ \widetilde{\sigma}^2 \widetilde{\kappa}^{2\widetilde{\nu}} }{ 
\sigma^2 \kappa^{2\nu} } 
\frac{ ( \kappa^2 + \| \omega \|_{\bbR^d}^2 )^{\nu + d/2} }{ 
( \widetilde{\kappa}^2 + \| \omega \|_{\bbR^d}^2 )^{\widetilde{\nu} + d/2} } , 
\qquad 
\omega \in \bbR^d, 
\]
we conclude with Corollary~\ref{cor:spectral-density-I}\ref{cor:spectral-density-I-neg} 
that Assumption~\ref{ass:kriging}.I can only be satisfied if $\widetilde{\nu}=\nu$. 
In this case, $f_{\mathsf{M}} \asymp \widetilde{f}_{\mathsf{M}}$ and, 
since by \cite[][Remark~4.1]{Krasnitskii1990} 
also $f_{\mathsf{M}} \asymp |\varphi_0|^2$ holds 
for some $\varphi_0\in\cW_{\Pi_{\mathbf{R}}}$ 
and some parallelepiped $\mathbf{R}$, 
Corollary~\ref{cor:spectral-density-III} is applicable and shows that 
Assumptions~\ref{ass:kriging}.I and~\ref{ass:kriging}.III hold, 
with  
$a= 
\frac{ \widetilde{\sigma}^2 \widetilde{\kappa}^{2\nu} }{ 
\sigma^2 \kappa^{2\nu} }$ 
in~\eqref{eq:ass:Ta}. 
Thus, misspecifying the second order structure 
$(0, \varrho)$ by 
$(0,\widetilde{\varrho})$ yields  
uniformly asymptotically optimal linear prediction 
if and only if ${\nu = \widetilde{\nu}}$. 
For equivalence of the corresponding Gaussian measures, 
$a=1$ is necessary, i.e., 
the microergodic parameter $\sigma^2 \kappa^{2\nu}$ 
has to coincide for the two models. 
This is in accordance with the identifiability of this parameter 
under infill asymptotics, see \cite{Zhang2004}.

%====================================================================
\subsection{Periodic random fields}
\label{subsec:periodic}
%====================================================================

A stochastic process $\{\GP(x)\}_{x\in\cX}$ 
indexed by $\cX:=[0,1]^d$ 
is said to be weakly periodic 
if its mean value function $\pE[\GP]\equiv m$ 
is constant on $[0,1]^d$ 
and, in addition, its covariance function 
$\varrho(x,x')$ only depends on the difference 
$x - x'$, where the difference is taken modulo $1$ 
in each coordinate (see \cite{stein97}). 
Let  $\bbZ_+^d$ denote all elements $\kvec = (k_1,\ldots,k_d)^\top \in \bbZ^d$ 
such that at least one element in the vector is nonzero, 
and the first nonzero component is positive. 
A weakly periodic process 
admits the series expansion 
\[
\textstyle 
\GP(x)  
= 
X_0 
+ 
\sum\limits_{\kvec\in\bbZ_+^d} 
\left[ X^{\rm c}_{\kvec} \cos(2\pi \kvec \cdot x) 
+  X^{\rm s}_{\kvec}  \sin(2\pi \kvec \cdot x) \right], 
\]
where $X_0, X^{\rm c}_{\kvec}, X^{\rm s}_\kvec$ are pairwise 
uncorrelated random variables such that   
$\pE[ X_0 ]= \pE[ Z ] = m$ and, 
for all $\kvec\in\bbZ^d_+$, one has  
$\pE[ X^{\rm c}_{\kvec} ] = \pE[ X^{\rm s}_{\kvec} ] = 0$ 
as well as  
$\pE\bigl[ | X^{\rm c}_{\kvec} |^2 \bigr] = \pE\bigl[ | X^{\rm s}_{\kvec} |^2 \bigr]$. 
Define 
$f\from\bbZ^d\to[0,\infty)$ 
by $f(\mathbf{0}) := \pV[ X_0 ]$, 
$f(\kvec) := \frac{1}{2}\pV[ X^{\rm c}_{\kvec} ]$ for $\kvec\in\bbZ_+^d$, 
and $f(-\kvec) = f(\kvec)$. 
Then, we can represent the covariance function of $Z$ as 
\begin{align*} 
\varrho(x,x') 
&\textstyle = 
\sum\limits_{\kvec\in\bbZ^d} 
f(\kvec) \left[ \cos(2\pi \kvec \cdot x) \cos(2\pi \kvec \cdot x')
+ \sin(2\pi \kvec \cdot x) \sin(2\pi \kvec \cdot x') \right] 
\\
&\textstyle = 
\sum\limits_{\kvec\in\bbZ^d} f(\kvec) \cos(2\pi \kvec \cdot (x-x')) 
=: \varrho_0(x-x'). 
\end{align*} 
For this reason, $f$ can be viewed as the spectral density 
with respect to the counting measure on $\bbZ^d$. 
It is not difficult to show that the set 
\[ 
\bigl\{1, e^{\rm c}_{\kvec}, e^{\rm s}_{\kvec} : \kvec\in\bbZ_+^d \bigr\}, 
\qquad 
e^{\rm c}_{\kvec}(x) := \sqrt{2}\cos(2\pi \kvec \cdot x), 
\qquad
e^{\rm s}_{\kvec}(x) := \sqrt{2}\sin(2\pi \kvec\cdot x), 
\] 
forms an orthonormal basis for $L_2\bigl( [0,1]^d \bigr)$. 
Moreover, it 
is an eigenbasis of 
the covariance operator 
with kernel $\varrho$. 
Indeed, $\int_{\cX} \varrho(x,x') \, \rd x' = f(\mathbf{0})$ and 
\[ 
\forall \kvec\in\bbZ_+^d : 
\quad 
\int_{\cX} \varrho(x,x') e^{\iota}_{\kvec}(x') \, \rd x'  
=  
f(\kvec) e^{\iota}_{\kvec}(x),  
\quad 
\iota \in \{ \mathrm{c}, \mathrm{s} \}. 
\] 
Since $\varrho_0(0) = \sum_{\kvec\in \bbZ^d} f(\kvec) < \infty$  
and $f(\kvec)\geq 0$, it is clear that $f(\kvec)$ accumulates only at zero. 
Thus, for any two weakly periodic random fields on~$[0,1]^d$ 
with corresponding spectral densities  
$f, \widetilde{f}\from\bbZ^d\to[0,\infty)$ defined as above, 
we are in the setting of Corollary~\ref{cor:equal_ef}: 
Assumptions~\ref{ass:kriging}.I and \ref{ass:kriging}.III are satisfied  
if and only if $\widetilde{f}(\kvec)/f(\kvec) \rightarrow a$ for some $a\in (0,\infty)$ 
as $|\kvec|\rightarrow \infty$. 
This result holds without any further assumptions 
on the spectral densities, and can be viewed as a version of 
\cite[][p.~102, Theorem~10]{stein99} for periodic random fields. 

%====================================================================
\subsection{Random fields on the sphere}
\label{subsec:sphere}
%====================================================================

Due to the popularity of 
the Mat\'ern covariance family on $\bbR^d$ 
(see Example~\ref{ex:cov-functions}\ref{ex:cov-euclid} 
and Section~\ref{subsec:matern})
it is highly desirable to have  
a corresponding covariance model 
also on the sphere $\bbS^2$. 
A simple remedy for this is 
to define the covariance function 
as in Example~\ref{ex:cov-functions}\ref{ex:cov-manifold}, i.e., 
via the \emph{chordal distance}  
$d_{\bbR^3}(x,x') = \|x-x'\|_{\bbR^3}$. 
One reason for why this is a common choice 
is that the (more suitable) great circle distance  
$d_{\bbS^2}(x,x') = \arccos\bigl( \scalar{x, x'}{\bbR^3} \bigr)$   
results in a kernel $\varrho$ 
which is (strictly) positive definite only 
for $\nu \leq 1/2$, 
see Example~\ref{ex:cov-functions}\ref{ex:cov-sphere} 
and \cite{gneiting2013}. 
As this severely limits the flexibility 
of the model,  
several authors have suggested  
alternative ``Mat\'ern-like'' covariances 
on $\mathbb{S}^2$. \pagebreak

Guinness and Fuentes~\cite{Guinness2016} 
proposed the 
\emph{Legendre--Mat\'ern} covariance, 
\begin{equation}\label{eq:appl:varrho1}   
\varrho_1(x,x') 
:= 
\sum\limits_{\ell=0}^{\infty} 
\frac{\sigma_1^2}{(\kappa_1^2+ \ell^2)^{\nu_1+ 1/2}} 
P_\ell\bigl( \cos d_{\bbS^2}(x,x') \bigr),
\quad 
x, x' \in\bbS^2, 
\end{equation}
where $\sigma_1,\nu_1,\kappa_1\in(0,\infty)$ 
are model parameters and 
$P_\ell\from[-1,1]\to\bbR$ is the 
$\ell$-th Legendre polynomial, i.e., 
\[ 
\textstyle 
P_\ell(y) 
= 
2^{-\ell} 
\frac1{\ell!}
\frac{\rd^\ell}{\rd y^\ell} 
\left( y^2 -1 \right)^\ell, 
\quad 
y\in[-1,1], 
\qquad 
\ell\in\bbN_0:=\{0,1,2,\ldots\}. 
\]
This choice is motivated  
firstly by the Legendre polynomial representation 
of positive definite functions on $\bbS^2$ 
(see \cite{schoenberg1942}) and 
secondly by the fact that the spectral density %$R(\omega)$ 
$f_{\mathsf{M}}(\omega)$
for the Mat\'ern covariance on $\bbR^d$ 
is~proportional to 
$\sigma^2(\kappa^2 + \|\omega\|_{\bbR^d}^2)^{-(\nu + d/2)}$, 
see~\eqref{eq:matern-spectral-density}. 
However, note that the parameter $\sigma_1^2$ 
in \eqref{eq:appl:varrho1} is not the variance 
since 
$\varrho_1(x,x) 
= 
\sigma_1^2 \sum_{\ell= 0}^{\infty} (\kappa_1^2 + \ell^2)^{-\nu_1-1/2}$.

Another plausible way of defining a Mat\'ern model on $\bbS^2$  
is to use the stochastic partial differential equation (SPDE) 
representation of Gaussian Mat\'ern fields derived by 
Whittle~\cite{whittle63}, 
according to which a centered Gaussian Mat\'ern field 
$\bigl\{ \GP^0(x) : x\in \bbR^d \bigr\}$ can be viewed as a 
solution to the SPDE 
\begin{equation}\label{eq:spde}
(\kappa^2 - \Delta)^{(\nu + d/2)/2} \left(\tau \GP^0 \right) = \cW 
\quad 
\text{on}\;\; \bbR^d,  
\end{equation}
where the parameter $\tau\in(0,\infty)$ controls  
the variance of $\GP^0$, 
$\cW$ is Gaussian 
white noise, and 
$\Delta$ is the Laplacian. 
Lindgren, Rue and Lindstr\"om~\cite{lindgren11} 
proposed Gaussian Mat\'ern fields 
on the sphere as solutions to \eqref{eq:spde} 
formulated on $\bbS^2$ 
instead of $\bbR^d$. 
In this case, $\Delta$ is the Laplace--Beltrami operator. 

In order to state the corresponding 
covariance function 
$\varrho_2\from \bbS^2\times\bbS^2\to\bbR$, 
we introduce the spherical coordinates 
$(\vartheta,\varphi)\in[0,\pi]\times[0,2\pi)$ 
of a point $(x_1,x_2,x_3)^\trsp\in \bbR^3$ on $\bbS^2$  
by $\vartheta = \arccos(x_3)$ 
and $\varphi = \arccos\bigl( x_1 (x_1^2 + x_2^2)^{-1/2} \bigr)$. 
For all $\ell\in\bbN_0$ 
and $m \in\{ -\ell, \ldots, \ell\}$, we then  
define the (complex-valued) spherical harmonic 
$Y_{\ell,m} : \bbS^2 \to \bbC$  
as (see \cite[][p.~64]{marinucci2011}) 
\begin{align*}
\hspace*{2cm}   
Y_{\ell,m}(\vartheta, \varphi) 
&= 
C_{\ell,m}
P_{\ell,m}(\cos\vartheta) e^{im\varphi}, 
& m &\geq 0, 
\hspace*{2cm}  
\\
Y_{\ell,m}(\vartheta, \varphi) 
&=  
(-1)^m \overline{Y}_{\ell,-m}(\vartheta,\varphi), 
& m &< 0, 
\end{align*}
where, for $\ell\in\bbN_0$ and 
$m\in\{0,\ldots,\ell\}$, 
we set 
$C_{\ell,m} 
:=
\sqrt{\frac{2\ell + 1}{4\pi} \frac{(\ell - m)!}{(\ell + m)!}}$ 
and 
$P_{\ell,m}\from[-1,1]\to\bbR$ denotes   
the associated Legendre polynomial,    
given by 
\[
\textstyle 
P_{\ell,m}(y) 
= 
(-1)^m \left(1-y^2 \right)^{m/2} 
\frac{\rd^m}{\rd y^m} P_{\ell}(y), 
\quad 
y\in[-1,1].
\]
The spherical harmonics 
$\{Y_{\ell,m} :  \ell \in\bbN_0, \, m = -\ell,\ldots, \ell \}$ 
are eigenfunctions of 
the Laplace--Beltrami operator, 
with corresponding eigenvalues given by 
$\lambda_{\ell,m} = -\ell(\ell+1)$. 
In addition, they 
form an orthonormal basis of the complex-valued   
Lebesgue space 
$L_2\bigl( \bbS^2, \nu_{\bbS^2}; \bbC \bigr)$, 
see 
\cite[][Proposition 3.29]{marinucci2011}.  
Here, $\nu_{\bbS^2}$ denotes the 
Lebesgue measure on the sphere which,  
in spherical coordinates, 
can be expressed as 
$\rd\nu_{\bbS^2}(x) 
= 
\sin\vartheta \,\rd\vartheta \,\rd\varphi$. 

The covariance function of the solution 
$\GP^0$ to the SPDE \eqref{eq:spde} on $\bbS^2$ 
can thus be represented using 
the spherical harmonics 
via the series expansion 
(cf.~\cite[][Theorem 5.13 and p.~125]{marinucci2011}) 
\[ 
\varrho_2(x,x') 
= 
\sum_{\ell=0}^{\infty}  
\frac{\tau^{-2}}{(\kappa^2+ \ell(\ell+1))^{\nu + 1}} 
\sum_{m=-\ell}^{\ell} 
Y_{\ell,m}(\vartheta,\varphi) 
\overline{Y}_{\ell,m}(\vartheta',\varphi') , 
\] 
where $(\vartheta,\varphi)$, $(\vartheta',\varphi')$ 
are the spherical coordinates of $x$ and $x'$, respectively. 
Then, by expressing also the 
Legendre--Mat\'ern 
covariance function in \eqref{eq:appl:varrho1} 
in spherical coordinates and 
by using the addition formula for 
the spherical harmonics 
\cite[][Equation (3.42)]{marinucci2011}, 
we find that 
\begin{align*} 
\varrho_1(x,x')
&=
\sum_{\ell=0}^{\infty} 
\frac{\sigma_1^2}{(\kappa_1^2+ \ell^2)^{\nu_1+ 1/2}} 
P_\ell\bigl(\scalar{ x, x' }{\bbR^3} \bigr) 
\\
&= 
\sum_{\ell=0}^{\infty} 
\frac{\sigma_1^2}{(\kappa_1^2+ \ell^2)^{\nu_1+ 1/2}} 
\frac{4\pi}{2\ell + 1}
\sum_{m=-\ell}^\ell 
Y_{\ell,m}(\vartheta, \varphi) 
\overline{Y}_{\ell,m}(\vartheta', \varphi') . 
\end{align*} 
Thus, the covariance functions 
$\varrho_1,\varrho_2$ 
are similar, 
but not identical. 
Due to 
the SPDE representation of $\varrho_2$, 
we believe that this is the 
preferable model. 
However, an immediate question is now if  
the two models provide similar kriging predictions. 
The answer to this is given by Corollary~\ref{cor:equal_ef}: 
Since 
$\sum_{m=-\ell}^{\ell} 
Y_{\ell,m}(\vartheta,\varphi) 
\overline{Y}_{\ell,m}(\vartheta',\varphi') 
= 
\sum_{m=-\ell}^{\ell} 
v_{\ell,m}(\vartheta,\varphi) 
v_{\ell,m}(\vartheta',\varphi')$, 
where 
\[
v_{\ell,m}(\vartheta,\varphi)  
:= 
\begin{cases} 
\sqrt{2} \, C_{\ell,-m} P_{\ell,-m}(\cos\vartheta) \cos(m\varphi) 
& 
\text{if } m < 0 ,
\\
(1/\sqrt{4\pi}) P_\ell(\cos\vartheta) 
& 
\text{if } m = 0 ,
\\
\sqrt{2} \, C_{\ell,m} P_{\ell,m}(\cos\vartheta) \sin(m\varphi) 
& 
\text{if } m > 0 ,
\end{cases} 
\]
the two covariance operators have the same 
(orthonormal, real-valued) eigenfunctions 
in $L_2\bigl(\bbS^2,\nu_{\bbS^2};\bbR\bigr)$. 
Thus, we are in the setting of Corollary~\ref{cor:equal_ef} 
and consider the 
limit of the ratio of the corresponding eigenvalues:   
\[ 
\lim_{\ell\rightarrow\infty} 
\frac{(\kappa_1^2+ \ell^2)^{\nu_1 + 1/2}(2\ell+1)}{ 
(\kappa^2+ \ell(\ell+1))^{\nu + 1}} 
\frac1{\tau^2\sigma_1^24\pi}  
= \begin{cases}
0 		& \text{if $\nu_1 < \nu$}, \\
\infty  & \text{if $\nu_1 > \nu$}, \\
\frac1{\tau^2\sigma_1^2 2\pi} 
& \mbox{if $\nu_1 = \nu$}.
\end{cases}
\]
We conclude that the models will provide  
asymptotically equivalent kriging prediction  
as long as they have the same smoothness parameter $\nu$ 
(and positive, finite variance parameters). 
By the same reasoning, it is easy to see that 
one may misspecify both $\tau$ and $\kappa$ 
as well as $\sigma_1$ and $\kappa_1$
for the two covariance models and 
still obtain asymptotically optimal linear prediction.

%====================================================================
\section{Discussion}
\label{section:discussion}
%====================================================================

For statistical applications it is crucial 
to understand the effect that misspecifying 
the mean or the covariance function 
has on linear prediction. 
We have addressed this  
by providing three necessary and sufficient conditions, 
Assumptions~\ref{ass:kriging}.I--III, 
for uniformly asymptotically optimal linear prediction 
of random fields on compact metric spaces.   

There are several 
directions 
in which this work can be continued in the future. 
An interesting question is  
whether Assumptions~\ref{ass:kriging}.I--III 
can be relaxed if the uniformity requirement on the optimality is dropped. 
Furthermore, the results of Section~\ref{subsec:spectral-density} 
can likely be refined 
to obtain \emph{necessary and sufficient} conditions 
on the spectral densities $f$ and $\widetilde{f}$. 
This should be possible at least in the case 
that $f\asymp |\varphi_0|^2$ holds 
for some $\varphi_0\in \cW_{\Pi_{\mathbf{R}}}$. 

A more challenging problem would be to generalize our results 
to the setting of \emph{locally compact} spaces. 
This extension is conceivable, 
but it would require substantial changes 
to both the problem formulation and the methods of proving. 
For the current setting of compact metric spaces, 
there are several additional applications that can be considered. 
For example, the application to Gaussian Mat\'ern fields on the sphere 
in Section~\ref{subsec:sphere}
can easily be extended to SPDE-based 
Gaussian Mat\'ern fields on more general domains, 
since our arguments  
depend only on the asymptotic behavior 
of the eigenvalues of the Laplace--Beltrami operator 
which is also known, 
for instance, on compact Riemannian manifolds, 
see e.g.\ \citep[][Theorem 15.2]{Shubin2001}.

%====================================================================
% ACKNOWLEDGMENTS 
%====================================================================

\begin{acks}[Acknowledgments]
The authors thank  
S.G.~Cox and 
J.M.A.M.~van Neerven   
for fruitful discussions on spectral 
theory which considerably 
contributed to the 
proof of Lemma~B.2, 
see Appendix~B in the Supplementary Material~\citep{kbsup}. 
In addition, we thank the reviewer 
and the editor for their valuable comments. 
\end{acks}

% AOS,AOAS: If there are supplements please fill:
\begin{supplement}
%	\sname{}
\stitle{Supplement to ``Necessary and sufficient conditions for asymptotically optimal 
	linear prediction of random fields on compact metric spaces''}
\sdescription{Three appendices of the manuscript.}
\end{supplement}

%%  MR numbers will be added by VTeX.                      %%
%%                                                         %%
%%  Use \cite{...} to cite references in text.             %%
%%                                                         %%
%%%%%%%%%%%%%%%%%%%%%%%%%%%%%%%%%%%%%%%%%%%%%%%%%%%%%%%%%%%%%

%====================================================================

\end{cbunit}
%====================================================================

%% if your bibliography is in bibtex format, uncomment commands:
%\bibliographystyle{imsart-number} % Style BST file (imsart-number.bst or imsart-nameyear.bst)
%\bibliography{bibliography}       % Bibliography file (usually '*.bib')

\newpage
\renewcommand*\footnoterule{}

	\begin{frontmatter}
	
	\title{Supplement to ``Necessary and sufficient conditions \\ 
		for asymptotically optimal 
		linear prediction \\ 
		of random fields on compact metric spaces''}
	\runtitle{Linear prediction for misspecified random fields}
	%\thankstext{T1}{A sample additional note to the title.}
	
	\begin{aug}
		%%%%%%%%%%%%%%%%%%%%%%%%%%%%%%%%%%%%%%%%%%%%%%%
		%% Only one address is permitted per author. %%
		%% Only division, organization and e-mail is %%
		%% included in the address.                  %%
		%% Additional information can be included in %%
		\author[A]{\fnms{Kristin} \snm{Kirchner}\ead[label=e1]{k.kirchner@tudelft.nl}}
		\and  
		\author[B]{\fnms{David} \snm{Bolin}\ead[label=e2]{david.bolin@kaust.edu.sa}} 
		%%%%%%%%%%%%%%%%%%%%%%%%%%%%%%%%%%%%%%%%%%%%%%
		%% Addresses                                %%
		%%%%%%%%%%%%%%%%%%%%%%%%%%%%%%%%%%%%%%%%%%%%%%
		\address[A]{Delft Institute of Applied Mathematics, 
			Delft University of Technology,
			\printead{e1}}
		
		\address[B]{CEMSE Division, 
			King Abdullah University of Science and Technology, 
			\printead{e2}}
	\end{aug}

\end{frontmatter}
%====================================================================
% APPENDIX 
%====================================================================

\begin{appendix}
	
	In this supplement 
	all references to sections or equations 
	of the form ``Section~X'' or ``(X.Y)'', 
	where~X is a number (and not a letter), 
	refer to the main file of the article.

	%====================================================================
	\section{Equivalence of Gaussian measures}\label{appendix:feldman-hajek}
	%====================================================================
	
	In this section we 
	recall several notions 
	from operator theory 
	and we state 
	the Feldman--H\'ajek theorem 
	as formulated by 
	Da Prato and Zabczyk~\cite{daPrato2014}, 
	see Theorem~\ref{thm:feldman-hajek} below. 
	This theorem characterizes equivalence 
	of two Gaussian measures on a Hilbert space 
	by three necessary and sufficient 
	conditions. 
	
	Assume that $\bigl( E,\scalar{\,\cdot\,,\,\cdot\,}{E} \bigr)$ 
	is a separable Hilbert space 
	and let $\{e_j\}_{j\in\bbN}$ 
	be an orthonormal basis for $E$. 
	The space of bounded 
	linear operators on~$E$ 
	is denoted by~$\cL(E)$, 
	and we write $\dual{T}\from E \to E$ for 
	the adjoint operator of $T\in\cL(E)$. 
	We recall that a linear operator~${T\from E \to E}$ 
	is compact, denoted by $T\in\cK(E)$,  
	if and only if it is the limit 
	(in the space $\cL(E)$) 
	of finite-rank operators. 
	Besides, $T$ is said to be a trace-class operator 
	or a Hilbert--Schmidt operator 
	provided that 
	$\tr(T) := \sum_{j\in\bbN} \scalar{ T e_j, e_j }{E} < \infty$ 
	or    
	$\tr(\dual{T} T) = \sum_{j\in\bbN} \norm{ T e_j }{E}^2 < \infty$, 
	respectively.
	In the case that $T$ is self-adjoint (i.e., $\dual{T}=T$), 
	these conditions are equivalent to  
	the (real-valued) eigenvalues of $T$ 
	being summable and square-summable, respectively. 
	
	Let  
	$\mu = \normal(m,\cC)$ 
	be a Gaussian measure on~$E$ 
	with mean $m\in E$ and self-adjoint, 
	positive definite, trace-class 
	covariance operator $\cC\from E\to E$. 
	The Cameron--Martin space $\cC^{1/2}(E)$ associated with 
	$\mu$ on $E$ 
	(also known as reproducing kernel Hilbert space of $\mu$) 
	is the Hilbert space which is defined as 
	the range 
	of~$\mathcal{C}^{1/2}$ in~$E$ 
	and which is equipped with the inner product
	$\scalar{\cC^{-1}\,\cdot\,, \,\cdot\,}{E}$. 
	As opposed to 
	Da Prato and Zabczyk~\cite{daPrato2014}, 
	we always assume that $\cC$ 
	is strictly positive definite  
	so that this definition requires no pseudo-inverse. 
	
	Let $\MU=\normal(\widetilde{m},\CC)$ 
	be a second Gaussian measure on $E$. 
	Then, $\mu$ and $\MU$ 
	are said to be equivalent 
	if they are mutually absolutely continuous. 
	That is, for all sets $A$ in the 
	Borel $\sigma$-algebra $\cB(E)$, 
	one has $\mu(A) = 0$ if and only if $\MU(A) = 0$. 
	In contrast, $\mu$ and $\MU$ 
	are called singular (or orthogonal) if there exists 
	some $A\in\cB(E)$ such that $\mu(A) = 0$ and $\MU(A) = 1$. 
	The next theorem is taken from 
	\cite[][Theorem~2.25]{daPrato2014}. 
	
	\begin{theorem}[Feldman--H\'ajek]\label{thm:feldman-hajek}
		Two Gaussian measures $\mu = \normal(m,\cC)$, 
		$\MU=\normal(\widetilde{m},\CC)$ 
		on the Hilbert space~$E$ 
		are either singular or equivalent. 
		They are equivalent 
		if and only if 
		the following three conditions 
		are satisfied:
		\begin{enumerate}[label = \normalfont(\roman*), labelsep=5mm, topsep=4pt]
			\item\label{fh-1}\hspace*{-2mm}  
			The Cameron--Martin spaces 
			$\cC^{1/2} (E)$ and $\widetilde{\cC}^{1/2} (E)$ 
			are norm equivalent. 
			\item\label{fh-2}\hspace*{-3mm} 
			The difference of the means satisfies 
			$m - \widetilde{m} \in \cC^{1/2}(E)$. 
			\item\label{fh-3}\hspace*{-4.5mm} 
			The operator 
			$T_1 :=  
			\cC^{-1/2} \widetilde{\cC} 
			\cC^{-1/2} - \cI_{E}$ 
			is a Hilbert--Schmidt operator on~$E$. 
		\end{enumerate}
	\end{theorem} 
	
	%====================================================================
	\section{Auxiliary results}  
	\label{appendix:auxiliary}
	%====================================================================
	
	The following 
	Lemmas~\ref{lem:Ts-compact},~\ref{lem:spectral} and~\ref{lem:int-H0} 
	contain auxiliary results which are used  
	in the proofs of Section~4.  
	Throughout this section, let $( E, \scalar{\,\cdot\,, \,\cdot\,}{E} )$ be a separable Hilbert space 
	with $\dim(E)=\infty$, 
	and identity $\cI_E\from E\to E$.\vspace*{-4.7mm}\pagebreak
	
	\begin{lemma}\label{lem:Ts-compact} 
		Let 
		$\cC,\CC\from E \to E$ be  
		self-adjoint and positive definite 
		linear operators. 
		Suppose that 
		$\CC^{1/2}\cC^{-1/2}$ defines  
		an isomorphism on $E$ 
		and that there exists 
		a positive real number 
		$a\in(0,\infty)$ such that the operator 
		$T_a := \cC^{-1/2}\CC\cC^{-1/2} - a \cI_{E}$
		is compact on~$E$. 
		Then, the following operators are all compact 
		on $E$: 
		\[ 
		\CC^{-1/2} \cC \CC^{-1/2} 
		- a^{-1} \cI_{E}, 
		\qquad 
		\CC^{1/2} \cC^{-1} \CC^{1/2} 
		- a \cI_{E}, 
		\qquad 
		\cC^{1/2} \CC^{-1} \cC^{1/2} 
		- a^{-1} \cI_{E}. 
		\] 
	\end{lemma} 
	
	\begin{proof}
		By assumption  
		$T_a = \cC^{-1/2} \CC \cC^{-1/2} - a \cI_{E}$ 
		is compact on $E$ and  
		$S := \CC^{1/2} \cC^{-1/2}$ 
		is an isomorphism on $E$ and, thus,  
		$S,\dual{S},S^{-1}, \dual{(S^{-1})}\in\cL(E)$. 
		Since the space of compact operators 
		$\cK(E)$ forms a two-sided ideal in 
		$\cL(E)$, we conclude that 
		\[
		\widetilde{T}_a 
		:= 
		\CC^{-1/2} \cC \CC^{-1/2} 
		- a^{-1} \cI_{E}
		= 
		- a^{-1} \dual{(S^{-1})} T_a S^{-1} \in \cK(E). 
		\]
		Then, again by the ideal property of $\cK(E)$, 
		$\CC^{1/2} 
		\cC^{-1} 
		\CC^{1/2} 
		- 
		a \cI_E 
		= 
		S T_a S^{-1} 
		\in \cK(E)$  
		and 
		$\cC^{1/2} 
		\CC^{-1} 
		\cC^{1/2} 
		-a^{-1} \cI_E 
		= 
		S^{-1} \widetilde{T}_a S 
		\in \cK(E)$ follow. 
	\end{proof}  
	
	\begin{lemma}\label{lem:spectral} 
		Let $\{e_j\}_{j\in\bbN}$ be an orthonormal basis for $E$. 
		For $n\in\bbN$, define the $n$-dimen- sional subspace  
		$E_n := \operatorname{span}\{e_j\}_{j=1}^n$, 
		with $E$-orthogonal complement 
		$E_n^\perp=\operatorname{span}\{e_j\}_{j>n}$,  
		and let $P_n\from E \to E_n$ 
		as well as 
		$P_n^\perp := \cI_E - P_n$ 
		be the $E$-orthogonal 
		projections onto $E_n$ and $E_n^\perp$, respectively. 
		Assume that $S\from E \to E$ is a  
		self-adjoint, positive definite
		isomorphism 
		such that, for any $a\in(0,\infty)$, the operator 
		$T_a := S - a \cI_E$ is not compact on~$E$. 
		Then, there exist 
		$\{ \underline{a}_n \}_{n\in\bbN}, \{ \overline{a}_n \}_{n\in\bbN} 
		\subset \bigl[\norm{S^{-1}}{\cL(E)}^{-1}, \norm{S}{\cL(E)}\bigr]$,  
		$\delta \in (0,\infty)$  
		and, for every $n\in\bbN$, 
		$\underline{w}_n, \overline{w}_n\in E_n^\perp\setminus\{0\}$ 
		such that, for all $n\in\bbN$,  
		\begin{equation}\label{eq:lem:spectral} 
			\overline{a}_n - \underline{a}_n \geq\delta, 
			\qquad 
			\left| 
			\frac{ \scalar{S \underline{w}_n, \underline{w}_n }{E} }{ 
				\scalar{ \underline{w}_n, \underline{w}_n }{E} } 
			- 
			\underline{a}_n \right| 
			< \frac{\hat{\delta}}{3} , 
			\quad\; 
			\text{and}
			\quad\;\;  
			\left| 
			\frac{ \scalar{S \overline{w}_n, \overline{w}_n }{E} }{ 
				\scalar{ \overline{w}_n, \overline{w}_n }{E} } 
			- 
			\overline{a}_n \right| 
			< \frac{\hat{\delta}}{3}, 
		\end{equation}
		where 
		$\hat{\delta} := \delta\norm{S^{-1}}{\cL(E)}^{-2} \norm{S}{\cL(E)}^{-2} 
		\in(0,\delta]$. 
	\end{lemma} 
	
	\begin{proof} 
		We will show that 
		\begin{equation}\label{eq:lem:spectral:proof}  
			\underline{a}_n 
			:= 
			\inf_{u \in E_n^\perp\setminus\{0\} } 
			\frac{ \scalar{Su, u}{E} }{ \scalar{u, u}{E} }, 
			\qquad 
			\overline{a}_n 
			:= 
			\sup_{u \in E_n^\perp\setminus\{0\} }
			\frac{ \scalar{Su, u}{E} }{ \scalar{u, u}{E} }, 
			\qquad\;\;   
			n\in\bbN, 
		\end{equation}
		and 
		$\delta := \inf_{n\in\bbN} 
		\{ \overline{a}_n - \underline{a}_n \}$ 
		have the desired properties. 
		
		\textbf{Step 1:} $\delta>0$. 
		By the definitions in \eqref{eq:lem:spectral:proof}, 
		$\overline{a}_n 
		- 
		\underline{a}_n \geq 0$ for all $n\in\bbN$ and 
		$\delta\geq 0$ follows. 
		Since $S\from E \to E$ is an isomorphism, 
		the sequences $\{ \underline{a}_n \}_{n\in\bbN}$ and  
		$\{ \overline{a}_n \}_{n\in\bbN}$
		take values in the compact interval 
		$[\underline{\alpha},\overline{\alpha}]\subset (0,\infty)$, 
		where $\underline{\alpha}:=\norm{S^{-1}}{\cL(E)}^{-1}$, 
		$\overline{\alpha}:=\norm{S}{\cL(E)}$.  
		Furthermore, by definition they are monotone 
		increasing and decreasing, respectively. 
		For this reason they converge 
		$\underline{a}_n \uparrow \underline{a}_{*}$, 
		$\overline{a}_n \downarrow \overline{a}_{*}$ 
		as $n\to\infty$, 
		with limits $\underline{a}_*, \overline{a}_*\in[\underline{\alpha},\overline{\alpha}]$. 
		Assume that $\delta=0$. 
		This would~imply~that 
		$\underline{a}_* = \overline{a}_* = a_*$ 
		and, for all $\eps\in(0,\infty)$, there exists 
		$n_\eps\in\bbN$ such that 
		\begin{align*} 
			\forall n \geq n_\eps: 
			\quad\;\; 
			\eps 
			&> 
			|\underline{a}_n - a_*| 
			=  
			a_* 
			- 
			\inf_{u \in E_n^\perp\setminus\{0\} } 
			\tfrac{ \scalar{Su, u}{E} }{\scalar{u, u}{E}} 
			= 
			\sup_{u \in E_n^\perp\setminus\{0\} } 
			\tfrac{ \scalar{ (a_* \cI_E - S) u, u}{E} }{\scalar{u, u}{E}} , 
			\\
			\forall n \geq n_\eps: 
			\quad\;\; 
			\eps 
			&> 
			|\overline{a}_n - a_*| 
			=  
			\sup_{u \in E_n^\perp\setminus\{0\} } 
			\tfrac{ \scalar{Su, u}{E} }{\scalar{u, u}{E}} 
			- 
			a_*
			= 
			\sup_{u \in E_n^\perp\setminus\{0\} } 
			\tfrac{ \scalar{ (S - a_* \cI_E )u, u }{E} }{\scalar{u, u}{E}}. 
		\end{align*} 
		This shows that  
		$T_{a_*}$ is the limit in $\cL(E)$ 
		of finite-rank operators,  
		\[
		\bigl\| P_n^\perp T_{a_*} P_n^\perp \bigr\|_{\cL(E) } 
		= 
		\sup_{w \in E \setminus\{0\} } 
		\frac{ | \scalar{ P_n^\perp T_{a_*} P_n^\perp w, w}{E} | }{\scalar{w, w}{E}} 
		= 
		\sup_{u \in E_n^\perp \setminus\{0\} } 
		\frac{ | \scalar{ (S-a_* \cI_E ) u, u}{E} | }{\scalar{u, u}{E}} 
		< \eps, 
		\]
		for all $n\geq n_\eps$, 
		contradicting the assumption that $T_{a_*}$ is not compact on $E$. 
		
		\textbf{Step 2:} Since $\delta\in(0,\infty)$, 
		for every $n\in\bbN$, there exist vectors 
		$\underline{w}_n, \overline{w}_n\in E_n^\perp$ 
		such that 
		$\bigl| 
		\tfrac{ \scalar{S \underline{w}_n, \underline{w}_n }{E}}{ 
			\scalar{\underline{w}_n, \underline{w}_n}{E}} 
		- \underline{a}_n \bigr| < \tfrac{\hat{\delta}}{3}$ 
		and 
		$\bigl| 
		\tfrac{ \scalar{S \overline{w}_n, \overline{w}_n }{E}}{ 
			\scalar{\overline{w}_n, \overline{w}_n}{E}} 
		- \overline{a}_n \bigr| < \tfrac{\hat{\delta}}{3}$ 
		by definition of $\underline{a}_n, \overline{a}_n$ 
		as the accumulation points in~\eqref{eq:lem:spectral:proof}.   
	\end{proof}

	In what follows, 
	$C(\cX)$ 
	denotes the vector space 
	of all continuous functions 
	$v\from\cX\to\bbR$, 
	defined on the connected, compact metric space 
	$(\cX, d_\cX)$. 
	Let $\nu_\cX$ 
	be a strictly positive, finite Borel measure on $(\cX,\cB(\cX))$. 
	Recall that 
	the norm $\norm{v}{C(\cX)}:=\sup_{x\in\cX}|v(x)|$ 
	renders $C(\cX)$  
	a Banach space.

	\begin{lemma}\label{lem:int-H0} 
		Let 
		$\GP^0\from\cX\times\Omega\to\bbR$ 
		be a centered Gaussian random field 
		indexed by the compact metric space $(\cX,d_\cX)$ 
		with continuous, 
		strictly positive definite covariance function 
		$\varrho\from\cX\times\cX\to\bbR$, 
		and let $\cH^0$ be the Hilbert space  
		defined in Equation~(2.4) in Section~2.1. 
		Then, for every $v\in C(\cX)$, 
		the random variable 
		$\gp_v := \scalar{\GP^0, v}{L_2(\cX,\nu_\cX)}$ is 
		an element of~$\cH^0$. 
	\end{lemma} 
	
	\begin{proof} 
		We first note that 
		continuity of the inner product 
		$\scalar{\,\cdot\,, \,\cdot\,}{L_2(\cX,\nu_\cX)}$
		and measurability of the 
		process 
		$\GP^0 \from (\Omega,\cA)
		\to (L_2(\cX,\nu_\cX),\cB(L_2(\cX,\nu_\cX)))$ 
		imply that, for every $v\in C(\cX)$, the mapping 
		$\gp_v\from (\Omega,\cA)\to(\bbR,\cB(\bbR))$ 
		is measurable, i.e., 
		that $\gp_v = \scalar{\GP^0, v}{L_2(\cX,\nu_\cX)}$ is a well-defined 
		real-valued random variable. 
		We now show that, for any $\eps\in(0,\infty)$, 
		there exist an integer $K\in\bbN$, 
		real numbers $a_1,\ldots,a_K\in\bbR$   
		and points $x_1,\ldots,x_K\in\cX$, such that 
		\begin{equation}\label{eq:lem:int-H0-0}
			\textstyle 
			\biggl\| 
			\sum\limits_{k=1}^K a_k \GP^0(x_k) 
			- 
			\gp_v 
			\Bigr\|_{\cH^0}^2 
			= 
			\pE \biggl[
			\Bigl| 
			\sum\limits_{k=1}^K a_k \GP^0(x_k) 
			- 
			\gp_v
			\Bigr|^2 
			\biggr]
			<\eps,  
		\end{equation}
		i.e., that we can approximate 
		$\gp_v = \scalar{\GP^0,v}{L_2(\cX,\nu_\cX)}$ 
		arbitrarily well in $L_2(\Omega,\bbP)$
		by an element of the vector space~$\cZ^0$ 
		given in~(2.3). 
		This proves the assertion of the lemma. 
		
		The case $v=0$ is trivial. Thus, from now 
		on we assume that $v\in C(\cX)\setminus\{0\}$ 
		so that $\| v\|_{C(\cX)} \in (0,\infty)$. 
		We first note that by square-integrability of $\GP^0$ 
		Fubini's theorem allows us to exchange the order 
		of integration with respect to $(\Omega,\bbP)$ 
		and with respect to $(\cX,\nu_\cX)$. 
		Thus, we find that 
		\begin{align}
			{\textstyle\biggl\| 
				\sum\limits_{k=1}^K a_k \GP^0(x_k)} 
			- 
			\gp_v 
			\biggr\|_{\cH^0}^2 
			&= 
			\int_\cX\int_\cX \varrho(x, x') v(x) v(x') \, \rd\nu_\cX(x) \, \rd\nu_\cX(x') 
			\notag 
			\\
			&
			+ 	{\textstyle 
				\sum\limits_{k=1}^K 
				\sum\limits_{\ell=1}^K a_k a_\ell \varrho(x_k, x_\ell) 
				- 2 
				\sum\limits_{k=1}^K a_k }
			\int\nolimits_\cX \varrho(x_k, x') v(x') \, \rd\nu_\cX(x').  
			\label{eq:lem:int-H0-1} 
		\end{align} 

		Now fix $\eps\in(0,\infty)$. 
		Furthermore, let $\{e_j\}_{j\in\bbN}\subset C(\cX)$ 
		be an orthonormal basis for $L_2(\cX,\nu_\cX)$ 
		consisting of the continuous representatives 
		of eigenfunctions of the covariance operator $\cC$ 
		corresponding to $\varrho$, see~(2.1),  
		with corresponding positive eigenvalues $\{\gamma_j\}_{j\in\bbN}$. 
		Since the series in~(2.2) converges uniformly, 
		there exists a finite integer $N\in\bbN$, such that 
		\begin{equation}\label{eq:lem:int-H0-2}
			\textstyle 
			\sup\nolimits_{x,x'\in\cX}
			|\varrho(x,x') - \varrho_N(x,x')| 
			<
			\tfrac{\eps}{6} \, \norm{v}{C(\cX)}^{-2} |\cX|^{-2} ,  
		\end{equation} 
		where $\varrho_N(x,x') 
		:= \sum\nolimits_{j=1}^N \gamma_j e_j(x) e_j(x')$ 
		and $|\cX|:=\nu_\cX(\cX)\in(0,\infty)$. 
		
		By continuity of the eigenfunctions 
		$M_N := \max_{1\leq j \leq N}\norm{e_j}{C(\cX)} \in(0,\infty)$ and, 	
		for every $x\in\cX$, there exists 
		$\delta^*_x \in(0,\infty)$ 
		such that, for all $x'\in B(x,\delta^*_x)$,  
		\begin{equation}\label{eq:lem:int-H0-3} 
			\textstyle 
			\max_{1\leq j \leq N} 
			|e_j(x) - e_j(x')| 
			< 
			\tfrac{\eps}{6} \,    
			\bigl( \tr(\cC) M_N \norm{v}{C(\cX)}^2 |\cX|^2 \bigr)^{-1} . 
		\end{equation}
		Here,  
		$B(x,\delta):= 
		\{x'\in\cX : d_\cX(x,x')< \delta\}$ 
		denotes the open ball 
		in $(\cX,d_\cX)$ 
		centered at $x\in\cX$, 
		with radius 
		$\delta>0$. 
		An open cover of $\cX$ is then given by 
		$\cX \subseteq \bigcup_{x\in\cX} B(x,\delta^*_x)$. 
		By compactness of the metric space 
		$(\cX,d_\cX)$, it contains a finite subcover, 
		i.e., there are finitely many $x_1,\ldots,x_K\in\cX$ 
		such that 
		$\cX \subseteq \bigcup_{k=1}^K  B(x_k,\delta^*_{x_k})$. 
		Furthermore, there exists a corresponding partition of unity 
		$\{\chi_k\}_{k=1}^K$, i.e., 
		$\chi_k \from \cX \to [0,1]$ has support 
		$\operatorname{supp}(\chi_k)\subseteq B(x_k,\delta^*_{x_k})$ 
		and $\sum_{k=1}^K \chi_k(x) =1$ for all $x\in\cX$. 
		We set 
		$a_k := \int_\cX \chi_k(x) v(x) \, \rd\nu_\cX(x)$, 
		for $k=1,\ldots,K$. 
		Note that this definition implies the estimate 
		\begin{equation}\label{eq:lem:int-H0-4}
			{\textstyle 
				\sum\limits_{k=1}^K |a_k| 
				\leq 
				\sum\limits_{k=1}^K } 
			\int_\cX \chi_k(x) |v(x)| \, \rd\nu_\cX(x)
			\leq \norm{v}{C(\cX)} |\cX|. 
		\end{equation}
		The identity \eqref{eq:lem:int-H0-1} then yields 
		\begin{align*} 
			\biggl\| 
			{\textstyle\sum\limits_{k=1}^K a_k} 
			\GP^0(x_k) 
			- 
			\gp_v 
			&\biggr\|_{\cH^0}^2 
			= 
			{\textstyle 
				\sum\limits_{k=1}^K \sum\limits_{\ell=1}^K a_k a_\ell} 
			(\varrho(x_k, x_\ell) - \varrho_N(x_k, x_\ell))
			\\
			&- 
			2 
			{\textstyle \sum\limits_{k=1}^K  a_k} 
			\int_\cX 
			(\varrho(x_k, x') - \varrho_N(x_k, x')) v(x') \, \rd\nu_\cX(x') 
			\\
			&+ 
			\int_\cX \int_\cX (\varrho(x, x')-\varrho_N(x,x')) v(x) v(x') \, \rd\nu_\cX(x) \, \rd\nu_\cX(x') 
			\\
			&+
			{\textstyle \sum\limits_{k=1}^K a_k}  
			\biggl( 
			{\textstyle\sum\limits_{\ell=1}^K  a_\ell} 
			\varrho_N(x_k, x_\ell) 
			-  
			\int_\cX \varrho_N(x_k, x') v(x') \, \rd\nu_\cX(x') 
			\biggr) 
			\\
			&+ 
			\int_\cX 
			\biggl( \int_\cX 
			\varrho_N(x, x') v(x) \, \rd\nu_\cX(x) 
			- 
			{\textstyle \sum\limits_{k=1}^K a_k}  
			\varrho_N(x_k, x') 
			\biggr) 
			v(x') \, \rd\nu_\cX(x') 
			\\
			=:\,&
			\text{(I)} 
			- 2 \, 
			\text{(II)} 
			+
			\text{(III)} 
			+ 
			\text{(IV)} 
			+ 
			\text{(V)}. 
		\end{align*} 
		We now bound the absolute values  
		of these five expressions term by term. 
		Firstly, by \eqref{eq:lem:int-H0-2} 
		and \eqref{eq:lem:int-H0-4} 
		we obtain that 
		\begin{align*} 
			| \text{(I)} | 
			&\leq 
			\sup_{x,x'\in\cX}
			|\varrho(x,x') - \varrho_N(x,x')| 
			\biggl( {\textstyle \sum\limits_{k=1}^K |a_k| } \biggr)^2 
			< 
			\tfrac{\eps}{6}. 
			\intertext{Again by \eqref{eq:lem:int-H0-2} 
				and  \eqref{eq:lem:int-H0-4} 
				also the second and third terms can be bounded,}
			| \text{(II)} |
			&\leq  
			\sup_{x,x'\in\cX}
			|\varrho(x,x') - \varrho_N(x,x')| 
			{\textstyle \sum\limits_{k=1}^K |a_k| } 
			\int_{\cX} |v(x')| \, \rd\nu_\cX(x') 
			< 
			\tfrac{\eps}{6},  
			\\
			| \text{(III)} |
			&\leq 
			\sup_{x,x'\in\cX}
			|\varrho(x,x') - \varrho_N(x,x')| 
			\biggl( 
			\int_\cX |v(x)| \, \rd\nu_{\cX}(x) 
			\biggr)^2 
			< 
			\tfrac{\eps}{6}. 
			\intertext{For the remaining two terms, 
				we use the relation \eqref{eq:lem:int-H0-3} 
				and thus obtain that} 
			|\text{(IV)}| 
			&\leq 
			{\textstyle \sum\limits_{k=1}^K |a_k| }
			\biggl| 
			{\textstyle \sum\limits_{\ell=1}^K } 
			\int_\cX 
			\chi_\ell(x') v(x') (
			\varrho_N(x_k, x_\ell) 
			-  
			\varrho_N(x_k, x') ) \, \rd\nu_\cX(x') 
			\biggr|
			\\
			&\leq 
			{\textstyle 
				\sum\limits_{j=1}^N 
				\gamma_j 
				\sum\limits_{k=1}^K |a_k| 
				|e_j(x_k)| 
				\sum\limits_{\ell=1}^K } 
			\int_{B_\ell} 
			\chi_\ell(x') \, |v(x')| \, 
			|e_j(x_\ell) 
			-  
			e_j(x') | \, \rd\nu_\cX(x') 
			\\
			&\leq 
			\tr(\cC) M_N 
			\norm{v}{C(\cX)}^2 |\cX|^2 
			\max_{1\leq j \leq N}  
			\max_{1\leq \ell\leq K} 
			\sup_{x'\in B_\ell } 
			|e_j(x_\ell) 
			-  
			e_j(x') | 
			< 
			\tfrac{\eps}{6},  
			\intertext{where $B_\ell := B(x_\ell, \delta^*_{x_\ell})$ 
				for $\ell=1,\ldots,K$, 
				as well as} 
			|\text{(V)}| 
			&=
			\biggl| 
			\int_\cX 
			\biggl( 
			{\textstyle \sum\limits_{k=1}^K }
			\int_\cX 
			\chi_k(x) 
			v(x) 
			(\varrho_N(x,x') - \varrho_N(x_k,x')) 
			\, \rd\nu_\cX(x)
			\biggr) 
			v(x') \, \rd\nu_\cX(x') 
			\biggr|
			\\
			&  
			= 
			\biggl| 
			{\textstyle \sum\limits_{j=1}^N }
			\gamma_j 
			\int_\cX 
			e_j(x') 
			v(x') \, \rd\nu_\cX(x')  
			{\textstyle \sum\limits_{k=1}^K } 
			\int_\cX 
			\chi_k(x) 
			v(x) 
			(e_j(x) - e_j(x_k)) 
			\, \rd\nu_\cX(x) 
			\biggr| 
			\\
			& 
			\leq 
			{\textstyle \sum\limits_{j=1}^N } 
			\gamma_j 
			\norm{e_j}{C(\cX)} 
			\norm{v}{C(\cX)} 
			|\cX| 
			{\textstyle \sum\limits_{k=1}^K } 
			\int_{B_k} 
			\chi_k(x) 
			|v(x)| \, 
			|e_j(x) - e_j(x_k)| 
			\, \rd\nu_\cX(x) 
			\\
			&  
			\leq 
			\tr(\cC) M_N 
			\norm{v}{C(\cX)}^2 |\cX|^2  
			\max_{1\leq j \leq N}  
			\max_{1\leq k \leq K}
			\sup_{x\in B_k} 
			|e_j(x) - e_j(x_k)|  
			< 
			\tfrac{\eps}{6} . 
		\end{align*} 
		We conclude that, for this choice 
		of $K\in\bbN$, $a_1,\ldots,a_k\in\bbR$, 
		$x_1,\ldots,x_K\in\cX$ \eqref{eq:lem:int-H0-0} holds, 
		which completes the proof, 
		since $\eps\in(0,\infty)$ was arbitrary.
	\end{proof} 
	
	We close this section with a further auxiliary result which shows that 
	assuming boundedness (from below and above), uniformly with
	respect to $n$ and $h$, of any fraction in~(3.6) 
	for all admissible sequences of subspaces 
	$\{\cH_n\}_{n\in\bbN}\in\cS^\mu_{\mathrm{adm}}$ 
	implies 
	uniform boundedness (from below and above) 
	with respect to $\{\cH_n\}_{n\in\bbN}\in\cS^\mu_{\mathrm{adm}}$ , $n$ and $h$ 
	of all terms in~(3.6). 
	
	\begin{lemma}\label{lem:uniform-Smu}
		Let $\mu=\normal(m,\cC)$ 
		and $\MU=\normal(\widetilde{m},\CC)$. 
		In addition, let $\hn,\Hn$ denote the best linear predictors of $h$ 
		based on~$\cH_n$ and the measures~$\mu$ and~$\MU$, 
		respectively, and let 
		$\cS^\mu_{\mathrm{adm}}$ 
		be defined as in~(2.12). 
		In the case that 
		for every $\{\cH_n\}_{n\in\bbN}\in\cS^\mu_{\mathrm{adm}}$ 
		there exist $k,K\in(0,\infty)$ such that 
		one of the following fractions is bounded 
		from below by $k$ and from above by $K$, 
		uniformly with respect to $n\in\bbN$ and $h\in\cH_{-n}$, 
		\begin{equation}\label{eq:uniform-Smu} 
			\frac{\PV\bigl[ \hn - h \bigr]}{\pV\bigl[ \hn - h\bigr]}, 
			\qquad 
			\frac{\pV\bigl[ \Hn - h \bigr]}{\PV\bigl[ \Hn - h\bigr]}, 
			\qquad 
			\frac{\pV\bigl[ \Hn - h \bigr]}{\pV\bigl[ \hn - h\bigr]}, 
			\qquad 
			\frac{\PV\bigl[ \hn - h \bigr]}{\PV\bigl[ \Hn - h\bigr]},  
		\end{equation}
		then 
		all assertions (i)--(iv) of Proposition~3.5 hold and, 
		in particular, 
		all of the above expressions  
		are bounded 
		uniformly with respect to $\{\cH_n\}_{n\in\bbN}\in\cS^\mu_{\mathrm{adm}}$, 
		$n\in\bbN$ and $h\in\cH_{-n}$. 
	\end{lemma}
	
	\begin{proof} 
		We show that the assumptions of the lemma 
		imply validity of Proposition~3.5(iii). 
		The remaining claims follow by equivalence of the 
		four statements (i)--(iv) in Proposition~3.5. 
		To this end, 
		by \eqref{eq:Es-Vars-1}--\eqref{eq:Es-Vars-4},  
		see Proposition~\ref{prop:mean-shifting} below,  
		we may without loss of generality assume that 
		$\mu$ and $\MU$ are centered, 
		$m=\widetilde{m}=0$. 
		
		We start with proving necessity of~(iii)
		for uniform boundedness in $n$ and $h$ (from above and below) 
		of the first two fractions in~\eqref{eq:uniform-Smu} 
		holding for all $\{\cH_n\}_{n\in\bbN}\in\cS^\mu_{\mathrm{adm}}$. 
		Let $n\in\bbN$.
		If Proposition~3.5(iii) 
		does not hold, then the norms 
		$\|\cdot\|_{\cH^0}^2 = \pV[\,\cdot\,]$ 
		and 
		$\|\cdot\|_{\CH^0}^2 = \PV[\,\cdot\,]$ 
		cannot be equivalent  
		on the orthogonal complements 
		of $\cH^0_n:=\operatorname{span}\{\gp_1,\ldots,\gp_n\}$ 
		in~$\cH^0$ and in~$\CH^0$ 
		(due to their equivalence 
		on the finite-dimensional subspace~$\cH^0_n$), 
		where $\{\gp_j\}_{j\in\bbN} \subset\cH^0\cap\CH^0$ 
		is the orthonormal basis for $\cH^0$ from 
		Lemma~4.1(ii). 
		Thus, 
		for any fixed $n\in\bbN$, 
		there exist $h^{(n)}\in \cH^0$ and $g^{(n)}\in\cH^0$ 
		such that $h^{(n)}$ and $g^{(n)}$ 
		are orthogonal to 
		$\cH^0_n$ 
		in $\cH^0$ and in~$\CH^0$, respectively, 
		and such that 
		either the relations 
		$\frac{\PV[ h^{(n)} ]}{\pV[ h^{(n)} ]}
		\geq n$  
		and  
		$\frac{\PV[ g^{(n)} ]}{\pV[ g^{(n)} ]} 
		\geq n$ hold,  
		or we obtain  
		that 
		$\frac{\pV[ h^{(n)} ]}{\PV[ h^{(n)} ] } \geq n$  
		and 
		$\frac{\pV[ g^{(n)} ]}{\PV[ g^{(n)} ]}  
		\geq n$. 
		Note that in either case, $h^{(n)}_n = 0$ 
		and $\widetilde{g}^{(n)}_n = 0$ vanish, 
		where $h^{(n)}_n, \widetilde{g}^{(n)}_n$ 
		denote the best linear predictors 
		of $h^{(n)}, g^{(n)}$
		based on $\cH_n:=\bbR\oplus\cH_n^0$ 
		and the measures $\mu$ and $\MU$, respectively. 
		In the first case, we therefore obtain 
		\begin{equation}\label{eq:proof:uniform-Smu-1} 
			\frac{\PV\bigl[ h^{(n)}_n - h^{(n)} \bigr]}{
				\pV\bigl[ h^{(n)}_n - h^{(n)} \bigr]}
			=
			\frac{\PV\bigl[ h^{(n)} \bigr]}{
				\pV\bigl[ h^{(n)} \bigr]}
			\geq n 
			\quad\;
			\text{and} 
			\quad\;
			\frac{\pV\bigl[ \widetilde{g}^{(n)}_n - g^{(n)} \bigr]}{ 
				\PV\bigl[ \widetilde{g}^{(n)}_n - g^{(n)}  \bigr]} 
			=
			\frac{\pV\bigl[ g^{(n)} \bigr]}{\PV\bigl[ g^{(n)} \bigr]}
			\leq n^{-1}, 
		\end{equation}
		so that in the limit $n\to\infty$ 
		the first fraction is unbounded 
		and the second converges to zero. 
		Since $\{\gp_j\}_{j\in\bbN}$ 
		is an orthonormal basis for $\cH^0$, we have  
		$\{\cH_n\}_{n\in\bbN} \in\cS^\mu_{\mathrm{adm}}$ and 
		\eqref{eq:proof:uniform-Smu-1} 
		contradicts the 
		boundedness (from above and below) 
		of the first two fractions in \eqref{eq:uniform-Smu}. 
		In the second case, 
		we have the reverse situation: 
		$\frac{\PV[ h^{(n)}_n - h^{(n)} ]}{
			\pV[ h^{(n)}_n - h^{(n)} ]}
		\leq n^{-1}$   
		and 
		$\frac{\pV[ \widetilde{g}^{(n)}_n - g^{(n)} ]}{ 
			\PV[ \widetilde{g}^{(n)}_n - g^{(n)} ]} 
		\geq n$, again 
		a contradiction. 
		
		We proceed with the latter two terms  
		in~\eqref{eq:uniform-Smu}. 
		To this end, assume again 
		that Proposition~3.5(iii) does not hold and, 
		for any fixed $n\in\bbN$, define 
		$\widetilde{\gp}_n,\ldots,\widetilde{\gp}_1$ 
		as the Gram--Schmidt orthonormalization 
		of $\gp_n,\ldots,\gp_1$ in $\CH^0$. 
		In particular, we have 
		$\widetilde{\gp}_n = \gp_n / \sqrt{\zeta_n}$, 
		where we set $\zeta_n := \| \gp_n \|_{\CH^0}^2 = \PV[\gp_n] 
		= \gamma_n^{-1} \| \CC^{1/2} e_n \|^2_{L_2(\cX,\nu_\cX)} < \infty$. 
		If Proposition~3.5(iii)
		does not hold, then, 
		for any fixed $n\in\bbN$, 
		there exist 
		$\phi_n, \varphi_n \in \cH^0$ 
		such that $\phi_n$ and~$\varphi_n$ 
		are orthogonal to 
		$\operatorname{span}\{\gp_1,\ldots,\gp_{n}\}$ 
		in $\cH^0$ and in $\CH^0$, respectively, 
		and such that either 
		$\frac{\PE[ \phi_n^2 ]}{\pE[ \phi_n^2 ]}
		=
		\frac{\PV[\phi_n]}{\pV[\phi_n]} 
		\geq n \zeta_n$  
		and  
		$\frac{\PE[ \varphi_n^2 ]}{
			\pE[ \varphi_n^2 ]} 
		= 
		\frac{\PV[ \varphi_n ]}{ 
			\pV[ \varphi_n ]} 
		\geq n \zeta_n$,  
		or the two relations 
		$\frac{\pE[\phi_n^2]}{\PE[\phi_n^2]} \geq n \zeta_n^{-1}$  
		and 
		$\frac{\pE[\varphi_n^2]}{\PE[\varphi_n^2]}  
		\geq n \zeta_n^{-1}$ hold. 
		Define    
		
		\begin{align*} 
			\theta_n 
			&:= 
			\min\bigl\{  \PE \bigl[ h^2 \bigr] / \pE \bigl[ h^2 \bigr] : 
			h \in \operatorname{span}\{ \varphi_n, \widetilde{\gp}_n \}, \; 
			h \neq 0 
			\bigr\} , 
			\\
			\widetilde{\theta}_n 
			&:= 
			\min\bigl\{ \pE \bigl[ h^2 \bigr] / \PE \bigl[ h^2 \bigr] : 
			h \in \operatorname{span}\{ \phi_n, \gp_n \}, \; 
			h \neq 0 
			\bigr\} , 
		\end{align*} 
		and $\Theta_n, \widetilde{\Theta}_n\in(0,\infty)$ 
		are given in the same way with 
		$\min$ replaced by~$\max$. 
		Note that in both cases  
		$\min\bigl\{ 
		\Theta_n \theta_n^{-1} , 
		\widetilde{\Theta}_n \widetilde{\theta}_n^{-1} 
		\bigr\}
		\geq n$. 
		It then follows as in \cite[][Proof of Theorem 5]{CLEVELAND71}
		that there are    
		$g^{(n)}, \xi_n \in 
		\operatorname{span}\{ \varphi_n, \widetilde{\gp}_n \}$
		and 
		$h^{(n)}, \psi_n \in 
		\operatorname{span}\{ \phi_n, \gp_n \}$
		such that 
		\[ 
		\frac{ \pE\bigl[ (\widetilde{g}^{(n)}_1 - g^{(n)} )^2 \bigr] }{
			\pE\bigl[ ( g^{(n)}_1 - g^{(n)} )^2 \bigr] } 
		= 
		\frac{ (\theta_n + \Theta_n)^2 }{4 \theta_n \Theta_n} 
		\quad\;\; 
		\text{and} 
		\quad\;\; 
		\frac{ \PE\bigl[ (h^{(n)}_1 - h^{(n)} )^2 \bigr] }{
			\PE\bigl[ ( \widetilde{h}^{(n)}_1 - h^{(n)} )^2 \bigr] } 
		= 
		\frac{ (\widetilde{\theta}_n + \widetilde{\Theta}_n )^2 }{
			4 \widetilde{\theta}_n \widetilde{\Theta}_n},  
		\] 
		where $g^{(n)}_1, \widetilde{g}^{(n)}_1$ 
		($h^{(n)}_1, \widetilde{h}^{(n)}_1$)
		denote 
		the best linear predictors of $g^{(n)}$ 
		(of $h^{(n)}$) based 
		on the subspace 
		$\cU_{n} := \bbR\oplus\operatorname{span}\{ \xi_n \}
		\subset\cH$ 
		($\cV_{n} := \bbR\oplus\operatorname{span}\{ \psi_n \}\subset\cH$) 
		and the measures 
		$\mu$ and $\MU$, 
		respectively.  
		We now define 
		$\cH^\star_1 := \cU_{1}$, 
		and, for $n\geq 2$,  
		$\cH^\star_n 
		:= 
		\cU_n 
		\oplus
		\operatorname{span}\{\widetilde{\gp}_1,
		\ldots, \widetilde{\gp}_{n-1} \}$.  
		Recall that $g^{(n)}  
		\in \operatorname{span}\{ \varphi_n, \widetilde{\gp}_n \}$ 
		and that  
		$\varphi_n$, $\widetilde{\gp}_n$  
		are $\CH^0$-orthogonal 
		to $\widetilde{\gp}_1, \ldots, \widetilde{\gp}_{n-1}$. 
		By letting $g^{(n)}_n, \widetilde{g}^{(n)}_n$ 
		denote the best linear predictors of $g^{(n)}$ 
		based 
		on $\cH^\star_n$ and the Gaussian measures 
		$\mu$ and $\MU$, respectively, 
		we therefore conclude that 
		$\widetilde{g}^{(n)}_n = \widetilde{g}^{(n)}_1$, 
		and, in addition, the estimate 
		$\pE \bigl[ (g^{(n)}_n - g^{(n)} )^2 \bigr] 
		\leq 
		\pE \bigl[ (g^{(n)}_1 - g^{(n)} )^2 \bigr]$
		holds, 
		since 
		$g^{(n)}_1 \in \cU_{n} \subset \cH^\star_n$. 
		Hence, 
		\[ 
		\frac{ \pE \bigl[ (\widetilde{g}^{(n)}_n - g^{(n)} )^2 \bigr] }{
			\pE \bigl[ ( g^{(n)}_n - g^{(n)} )^2 \bigr] } 
		=
		\frac{ \pE \bigl[ ( \widetilde{g}^{(n)}_1 - g^{(n)} )^2 \bigr] }{
			\pE \bigl[ ( g^{(n)}_n - g^{(n)} )^2 \bigr] } 
		\geq 
		\frac{ \pE \bigl[ ( \widetilde{g}^{(n)}_1 - g^{(n)} )^2 \bigr] }{
			\pE \bigl[ ( g^{(n)}_1 - g^{(n)} )^2 \bigr] } 
		=
		\frac{ (\theta_n + \Theta_n)^2 }{4 \theta_n \Theta_n} 
		\geq 
		\frac{1}{4} \frac{ \Theta_n }{ \theta_n } 
		\geq 
		\frac{n}{4}.    
		\]
		By defining 
		$\cH^\divideontimes_1 := \cV_{1}$, 
		and, for $n\geq 2$,  
		$\cH^\divideontimes_n 
		:= 
		\cV_n\oplus\operatorname{span}\{ \gp_1,
		\ldots, \gp_{n-1} \}$, 
		we furthermore  
		obtain that 
		$\frac{ \PE [ (h^{(n)}_n - h^{(n)} )^2 ] }{
			\PE [ ( \widetilde{h}^{(n)}_n - h^{(n)} )^2 ] } 
		\geq 
		\widetilde{\Theta}_n / (4 \widetilde{\theta}_n) 
		\geq 
		n/4$, 
		where $h^{(n)}_n$ and $\widetilde{h}^{(n)}_n$ 
		denote the best linear predictors of $h^{(n)}$ 
		based 
		on $\cH^\divideontimes_n$ and 
		$\mu$ resp.\ $\MU$.    
		We thus have constructed two  
		sequences  
		of subspaces 
		$\{\cH^\star_n\}_{n\in\bbN}$, 
		$\{\cH^\divideontimes_n\}_{n\in\bbN}$,   
		which are not necessarily nested, 
		such that %(here the predictors are defined as above)
		\[ 
		\forall n\in\bbN 
		\quad    
		\exists g^{(n)}, h^{(n)} \in\cH 
		: 
		\quad    
		\frac{ \pV \bigl[ \widetilde{g}^{(n)}_n - g^{(n)} \bigr] }{
			\pV \bigl[ g^{(n)}_n - g^{(n)} \bigr] }  
		\geq 
		\frac{n}{4} 
		\quad\;\; 
		\text{and} 
		\quad\;\; 
		\frac{ \PV \bigl[ h^{(n)}_n - h^{(n)} \bigr] }{
			\PV \bigl[ \widetilde{h}^{(n)}_n - h^{(n)} \bigr] }  
		\geq 
		\frac{n}{4}.   
		\] 
		Here, we used that 
		$m=\widetilde{m}=0$ which implies that    
		$\pV \bigl[ \widetilde{g}^{(n)}_n - g^{(n)} \bigr]
		= 
		\pE \bigl[ ( \widetilde{g}^{(n)}_n - g^{(n)} )^2 \bigr]$ 
		and   
		$\PV \bigl[ h^{(n)}_n - h^{(n)} \bigr]
		= 
		\PE \bigl[ ( h^{(n)}_n - h^{(n)} )^2 \bigr]$.  
		Finally, since, for all $n \geq 2$, 
		$\operatorname{span}\{ \gp_1,\ldots,\gp_{n-1} \} 
		= 
		\operatorname{span}\{ \widetilde{\gp}_1,\ldots,\widetilde{\gp}_{n-1} \} 
		\subset \cH^\star_n \cap \cH^\divideontimes_n$ 
		holds, we obtain 
		that $\{\cH^\star_n\}_{n\in\bbN} , 
		\{\cH^\divideontimes_n\}_{n\in\bbN} 
		\in\cS^\mu_{\mathrm{adm}}$  
		by the basis property of $\{\gp_j\}_{j\in\bbN}$ in $\cH^0$   
		and the desired contradiction %to~\eqref{eq:uniform-Smu} 
		follows. 
	\end{proof} 
	
	%====================================================================
	\section{Calculations for different means}\label{appendix:mean_centering}
	%====================================================================
	
	In this section we justify 
	that we may, 
	without loss of generality, assume that $m=0$ 
	when proving all the theorems  
	of Section~3. 
	We first show that the best linear predictors 
	based on two different Gaussian measures 
	with the same covariance operator  
	differ by a constant. 
	
	\begin{lemma}\label{lem:same-cov-kriging} 
		Let $\hat{\cC}\from L_2(\cX,\nu_\cX) \to L_2(\cX,\nu_\cX)$ 
		be a self-adjoint, positive definite, trace-class linear operator. 
		For $\widehat{m},\breve{m}\in L_2(\cX,\nu_\cX)$, 
		consider two Gaussian measures  
		$\widehat{\mu} := \normal(\widehat{m}, \hat{\cC})$ 
		and 
		$\breve{\mu}   := \normal(\breve{m}, \hat{\cC})$. 
		Then, 
		\begin{equation}\label{eq:same-cov-kriging}
			\widehat{h}_n  
			= 
			\breve{h}_n 
			-  
			\widehat{\pE}\bigl[ \breve{h}_n - h \bigr]   
			\quad
			\text{and} 
			\quad 
			\breve{h}_n 
			= 
			\widehat{h}_n 
			-  
			\breve{\pE}\bigl[ \widehat{h}_n - h \bigr],  
		\end{equation}
		where $\widehat{h}_n, \breve{h}_n$ 
		are the best linear predictors 
		based on a subspace $\cH_n$ of $\hat{\cH}$ 
		corresponding to $n$ linearly independent 
		observations (cf.~Equation~(2.9) in Section 2.2)  
		and the measures 
		$\widehat{\mu}$ resp.\ $\breve{\mu}$. 
		Furthermore, 
		$\widehat{\pE}$ and $\breve{\pE}$ denote the 
		expectation operators under 
		$\widehat{\mu}$ and $\breve{\mu}$, 
		respectively.  
	\end{lemma}  
	
	\begin{proof} 
		Firstly, unbiasedness  
		$\widehat{\pE}\bigl[\widehat{h}_n - h \bigr]= 0$ 
		for $\widehat{h}_n$ as in \eqref{eq:same-cov-kriging}
		is obvious. 
		If $\breve{h}_n$ is the best 
		linear predictor based on $\cH_n$ 
		and~$\breve{\mu}$, 
		we obtain 
		$\breve{\pE}\bigl[ \breve{h}_n - h \bigr]=0$ 
		and the identical covariances   
		of $\widehat{\mu}$ and $\breve{\mu}$ 
		show that, for all $g_n\in\cH_n$,  
		\begin{align*}
			0 
			&= 
			\breve{\pE}\bigl[ (\breve{h}_n - h) g_n \bigr] 
			= 
			\breve{\pE}\bigl[ 
			\bigl(\breve{h}_n - h - \breve{\pE}\bigl[ \breve{h}_n - h \bigr]\bigr) 
			\bigl(g_n - \breve{\pE}[g_n] \bigr) \bigr]
			\\ 
			&= 
			\widehat{\pE}\bigl[ 
			\bigl(\breve{h}_n - h - \widehat{\pE}\bigl[ \breve{h}_n - h \bigr]\bigr) 
			\bigl(g_n - \widehat{\pE}[g_n] \bigr) \bigr]
			= 
			\widehat{\pE}\bigl[ 
			\bigl(\breve{h}_n - \widehat{\pE}\bigl[ \breve{h}_n - h \bigr] - h\bigr) 
			g_n \bigr]. 
		\end{align*} 
		By uniqueness of the kriging 
		predictor defined in~(2.10), 
		we conclude that 
		$\widehat{h}_n$ in \eqref{eq:same-cov-kriging} 
		is the best linear predictor 
		based on $\cH_n$ and the measure $\widehat{\mu}$.   
		Changing the roles of $\widehat{\mu}, \breve{\mu}$
		proves the second assertion 
		in \eqref{eq:same-cov-kriging}. 
	\end{proof} 
	
	In addition to  
	$\mu=\normal(m,\cC)$ and 
	$\MU=\normal(\widetilde{m},\CC)$, 
	we consider the centered and mean-shifted 
	Gaussian measures 
	$\muc := \normal(0,\cC)$,  
	$\mus := \normal(\widetilde{m}-m, \cC)$,  
	$\MUc := \normal(0,\CC)$,  
	$\MUs := \normal(\widetilde{m}-m,\CC)$.
	Similarly as for $\mu$ and $\MU$,  
	$\pEc[\,\cdot\,]$, 
	$\pEs[\,\cdot\,]$, 
	$\PEc[\,\cdot\,]$, 
	$\PEs[\,\cdot\,]$
	and 
	$\pVc[\,\cdot\,]$, 
	$\pVs[\,\cdot\,]$, 
	$\PVc[\,\cdot\,]$, 
	$\PVs[\,\cdot\,]$
	denote the expectation and variance operators 
	under the measures $\muc$, $\mus$, $\MUc$ and $\MUs$, 
	respectively.  
	Furthermore, $\hnc, \hns, \Hnc$ 
	and $\Hns$ are the best linear predictors of $h$  
	based on a subspace~$\cH_n$, generated by $n$ observations,  
	and the measures $\muc, \mus, \MUc, \MUs$, cf.~(2.10).

	For fixed $n\in\bbN$, we now consider the corresponding errors 
	\begin{align*}
		\err &:= \hn - h, 
		& 
		\errc &:= \hnc - h, 
		& 
		\errs &:= \hns - h, 
		\\
		\ERR &:= \Hn - h, 
		& 
		\ERRc &:= \Hnc - h, 
		& 
		\ERRs &:= \Hns - h.
	\end{align*}

	\begin{proposition}\label{prop:mean-shifting} 
		With the above definitions, we have the identities: 
		\begin{align}
			\frac{\pE \bigl[ \ERR^2 \bigr]}{
				\pE \bigl[ \err^2 \bigr]} 
			&= 
			\frac{\pEc \bigl[ \ERRs^2 \bigr]}{
				\pEc \bigl[ \errc^2 \bigr]},  
			&
			\frac{\pV \bigl[ \ERR \bigr]}{
				\pV \bigl[ \err \bigr]} 
			&= 
			\frac{\pVc \bigl[ \ERRs \bigr]}{
				\pVc \bigl[ \errc \bigr]} 
			=
			\frac{\pVc \bigl[ \ERRc \bigr]}{
				\pVc \bigl[ \errc \bigr]} 
			=
			\frac{\pEc \bigl[ \ERRc^2 \bigr]}{
				\pEc \bigl[ \errc^2 \bigr]} , 
			\label{eq:Es-Vars-1} 
			\\
			\frac{\PE \bigl[ \err^2 \bigr]}{
				\PE \bigl[ \ERR^2 \bigr]} 
			&= 
			\frac{\PEs \bigl[ \errc^2 \bigr]}{
				\PEs \bigl[ \ERRs^2 \bigr]} , 
			&
			\frac{\PV \bigl[ \err \bigr]}{
				\PV \bigl[ \ERR \bigr]} 
			&= 
			\frac{\PVs \bigl[ \errc \bigr]}{
				\PVs \bigl[ \ERRs \bigr]} 
			=
			\frac{\PVc \bigl[ \errc \bigr]}{
				\PVc \bigl[ \ERRc \bigr]} 
			=
			\frac{\PEc \bigl[ \errc^2 \bigr]}{
				\PEc \bigl[ \ERRc^2 \bigr]} , 
			\label{eq:Es-Vars-2} 
			\\
			\frac{\PE \bigl[ \err^2 \bigr]}{
				\pE \bigl[ \err^2 \bigr]} 
			&= 
			\frac{\PEs \bigl[ \errc^2 \bigr]}{
				\pEc \bigl[ \errc^2 \bigr]} , 
			&
			\frac{\PV \bigl[ \err \bigr]}{
				\pV \bigl[ \err \bigr]} 
			&= 
			\frac{\PVs \bigl[ \errc \bigr]}{
				\pVc \bigl[ \errc \bigr]} 
			=
			\frac{\PVc \bigl[ \errc \bigr]}{
				\pVc \bigl[ \errc \bigr]} 
			=
			\frac{\PEc \bigl[ \errc^2 \bigr]}{
				\pEc \bigl[ \errc^2 \bigr]}, 
			\label{eq:Es-Vars-3} 
			\\
			\frac{\pE \bigl[ \ERR^2 \bigr]}{
				\PE \bigl[ \ERR^2 \bigr]} 
			&= 
			\frac{\pEc \bigl[ \ERRs^2 \bigr]}{
				\PEs \bigl[ \ERRs^2 \bigr]}, 
			&
			\frac{\pV \bigl[ \ERR \bigr]}{
				\PV \bigl[ \ERR \bigr]} 
			&= 
			\frac{\pVc \bigl[ \ERRs \bigr]}{
				\PVs \bigl[ \ERRs \bigr]} 
			=
			\frac{\pVc \bigl[ \ERRc \bigr]}{
				\PVc \bigl[ \ERRc \bigr]} 
			=
			\frac{\pEc \bigl[ \ERRc^2 \bigr]}{
				\PEc \bigl[ \ERRc^2 \bigr]}.  
			\label{eq:Es-Vars-4} 
		\end{align}

		Furthermore, we obtain the following relations: 
		\begin{align}  
			&\frac{\PEs \bigl[ \errc^2 \bigr]}{
				\pEc \bigl[ \errc^2 \bigr]} 
			= 
			\frac{ \PV \bigl[ \err \bigr]}{
				\pV \bigl[ \err \bigr]}
			+ 
			\frac{\bigl| \pEs[\errc] \bigr|^2}{
				\pEc \bigl[ \errc^2 \bigr]}, 
			&&
			\frac{\pEc \bigl[ \ERRs^2 \bigr]}{
				\PEs \bigl[ \ERRs^2 \bigr]} 
			= 
			\frac{ \pV \bigl[ \ERR \bigr]}{
				\PV \bigl[ \ERR \bigr]}
			+ 
			\frac{\bigl| \PEs[\ERRc] \bigr|^2}{
				\PEc \bigl[ \ERRc^2 \bigr]}, 
			\label{eq:decomp_means-0}
			\\
			&\frac{\pEc \bigl[ \ERRs^2 \bigr]}{
				\pEc \bigl[ \errc^2 \bigr]} - 1   
			\geq  
			\frac{ \pV \bigl[ \ERR \bigr]}{
				\pV \bigl[ \err \bigr]} - 1  , 
			&& 
			\frac{\PEs \bigl[ \errc^2 \bigr]}{
				\PEs \bigl[ \ERRs^2 \bigr]} 
			- 1
			\geq 
			\frac{ \PV \bigl[ \err \bigr]}{
				\PV \bigl[ \ERR \bigr]} 
			- 1 , 
			\label{eq:decomp_means-1}
			\\
			&\frac{\pEc \bigl[ \ERRs^2 \bigr]}{
				\pEc \bigl[ \errc^2 \bigr]} 
			- 1   
			\geq   
			c_1^{-2}
			\frac{\bigl| \PEs[\ERRc] \bigr|^2}{
				\PEc \bigl[ \ERRc^2 \bigr]} , 
			&& 
			\frac{\PEs \bigl[ \errc^2 \bigr]}{
				\PEs \bigl[ \ERRs^2 \bigr]} 
			- 1
			\geq 
			c_0^{-2}
			\frac{\bigl| \pEs[ \errc ] \bigr|^2}{
				\pEc \bigl[ \errc^2 \bigr]}, 
			\label{eq:decomp_means-2}
		\end{align}
		where \eqref{eq:decomp_means-2} holds, 
		provided that 
		\[
		c_0 := 
		\norm{\CC^{1/2} \cC^{-1/2}}{\cL(L_2(\cX,\nu_\cX))}
		\in (0,\infty) 
		\quad 
		\text{and} 
		\quad 
		c_1 := 
		\norm{\cC^{1/2} \CC^{-1/2}}{\cL(L_2(\cX,\nu_\cX))}
		\in (0,\infty).
		\] 
	\end{proposition}
	
	\begin{proof} 
		We first note that the covariance operators  
		of $\mu,\muc,\mus$ 
		are identical, 
		and the same holds 
		for those of $\MU,\MUc,\MUs$. 
		For this reason, the statements for 
		the variances in \eqref{eq:Es-Vars-1}--\eqref{eq:Es-Vars-4} 
		readily follow from 
		Lemma~\ref{lem:same-cov-kriging} 
		which shows that the errors 
		of best linear predictors 
		for Gaussian measures with the same covariance 
		operator differ by a constant.  	
		By the same argument we find that 
		\begin{equation}\label{eq:proof:mean-shifting-1}
			\pE \bigl[ \err^2 \bigr]  
			= \pV [ \err ] 
			=  \pV [ \errc ]  
			=  \pVc[ \errc ] 
			=  \pEc \bigl[ \errc^2 \bigr]. 
		\end{equation}
		Furthermore, 
		$\PE[e] 
		= 
		\PE \bigl[ 
		\errc 
		- 
		\pE[\errc]
		\bigr]
		= 
		\PE[\errc] 
		- 
		\pE[\errc] 
		=
		\PEs [ \errc ]$, 
		where we used that, 
		for all $h\in\cH$, 
		$\PEs[h] 
		= 
		\PE[h] - \pE[h]$. 
		Therefore, 
		\begin{equation}\label{eq:proof:mean-shifting-2}
			\PE\bigl[ \err^2 \bigr] 
			= 
			\PV[ \err ] 
			+ 
			\bigl|\PE[ \err ] \bigr|^2  
			= 
			\PVs [ \errc ] 
			+ 
			\bigl| \PEs[ \errc ] \bigr|^2  
			= \PEs \bigl[ \errc^2 \bigr] . 
		\end{equation}
		Combing the two equalities 
		\eqref{eq:proof:mean-shifting-1} 
		and 
		\eqref{eq:proof:mean-shifting-2} 
		proves the identity for the second moments 
		in \eqref{eq:Es-Vars-3}. 
		Next, we prove the corresponding 
		statement in~\eqref{eq:Es-Vars-1}. 
		Lemma~\ref{lem:same-cov-kriging} 
		yields the identities 
		$\Hn = \Hnc - \PE\bigl[ \Hnc - h \bigr]$ 
		and    
		$\Hns = \Hnc - \PEs\bigl[ \Hnc - h \bigr]$.  
		Thus,  
		\begin{align*}
			\pE[\ERR]  
			&= 
			\pE \bigl[  
			\Hnc - h
			- \PE\bigl[ \Hnc - h \bigr]  
			\bigr]   
			= -\PE\bigl[ \ERRc \bigr] + \pE\bigl[ \ERRc \bigr]  
			\\
			&= 
			- \PEs\bigl[ \ERRc \bigr] 
			= \PEc \bigl[  
			\Hnc - h 
			- 
			\PEs\bigl[ \Hnc - h \bigr]
			\bigr]
			= \PEc \bigl[ \ERRs \bigr] = \pEc \bigl[ \ERRs \bigr]. 
		\end{align*}
		Since 
		$\mu$ and $\muc$ have identical covariance operators, 
		this  
		implies that  
		\begin{equation}\label{eq:proof:mean-shifting-3}
			\pE\bigl[ \ERR^2 \bigr]
			= 
			\pV[\ERR] 
			+ 
			\bigl| \pE[\ERR] \bigr|^2
			= 
			\pVc[\ERRs] 
			+ 
			\bigl| \pEc[\ERRs] \bigr|^2
			= 
			\pEc\bigl[ \ERRs^2 \bigr]. 
		\end{equation}
		Combining \eqref{eq:proof:mean-shifting-1} 
		and \eqref{eq:proof:mean-shifting-3} 
		shows the first equality in \eqref{eq:Es-Vars-1}. 
		For \eqref{eq:Es-Vars-2} and   
		\eqref{eq:Es-Vars-4} the second moment identities can 
		be derived similarly: 
		By changing the roles of 
		$\mu$ and $\MU$, 
		\eqref{eq:proof:mean-shifting-1}--\eqref{eq:proof:mean-shifting-3}
		above show that 
		\begin{equation}\label{eq:proof:mean-shifting-4}
			\frac{\PE\bigl[ \err^2 \bigr]}{\PE\bigl[ \ERR^2 \bigr]} 
			= 
			\frac{\PEc\bigl[ \errs^2 \bigr]}{\PEc\bigl[ \ERRc^2 \bigr]} 
			\qquad 
			\text{and} 
			\qquad 
			\frac{\pE\bigl[ \ERR^2 \bigr]}{\PE\bigl[ \ERR^2 \bigr]} 
			= 
			\frac{\pEs\bigl[ \ERRc^2 \bigr]}{\PEc\bigl[ \ERRc^2 \bigr]} . 
		\end{equation}
		In addition, we obtain  
		\begin{equation}\label{eq:proof:mean-shifting-5} 
			\PEc\bigl[ \ERRc^2 \bigr] 
			= 
			\PVc[\ERRc]
			= 
			\PVs[\ERRs] 
			= 
			\PEs\bigl[ \ERRs^2 \bigr],
		\end{equation} 
		as well as 
		\begin{align} 
			\PEc\bigl[ \errs^2 \bigr] 
			&= 
			\PVc[\errs] + \bigl| \PEc[\errs] \bigr|^2
			= 
			\PVs[\errc] + \bigl| \PEs[\errc] \bigr|^2
			= 
			\PEs\bigl[ \errc^2 \bigr], 
			\label{eq:proof:mean-shifting-6}
			\\
			\pEs\bigl[ \ERRc^2 \bigr] 
			&= 
			\pVs[\ERRc] + \bigl| \pEs[\ERRc] \bigr|^2
			= 
			\pVc[\ERRs] + \bigl| \pEc[\ERRs] \bigr|^2
			= 
			\pEc\bigl[ \ERRs^2 \bigr].
			\label{eq:proof:mean-shifting-7}  
		\end{align}
		Here, we have used that 
		$\errs = \errc - \pEs[\errc]$, 
		$\PEc[\errc] = \pEc[\errc]=0$,   
		$\pEs[\errc] = \PEs[\errc]$,
		as well as  
		$\ERRc = \ERRs - \PEc[\ERRs]$,
		$\pEs[\ERRs] = \PEs[\ERRs]=0$,    
		and 
		$\PEc[\ERRs] = \pEc[\ERRs]$. 
		Combining the identities 
		\eqref{eq:proof:mean-shifting-4},  
		\eqref{eq:proof:mean-shifting-5},  
		\eqref{eq:proof:mean-shifting-6},     
		\eqref{eq:proof:mean-shifting-7}  
		finishes the proof 
		of \eqref{eq:Es-Vars-1}--\eqref{eq:Es-Vars-4}. 
		
		It remains to derive \eqref{eq:decomp_means-0}--\eqref{eq:decomp_means-2}.  
		To this end, 
		we again exploit Lemma~\ref{lem:same-cov-kriging}, 
		which gives that   
		$\errc = \err - \pEc[ \err ]$, 
		$\ERRs = \ERRc - \PEs[ \ERRc ]$
		and 
		$\ERRc = \ERR - \PEc[ \ERR ]$.    
		Then,  
		\begin{align*} 
			\PEs \bigl[ \errc^2 \bigr]   
			&= 
			\PVs\bigl[ \errc \bigr] 
			+ 
			\bigl| \PEs[\errc] \bigr|^2 
			=
			\PV \bigl[ \err \bigr] 
			+ 
			\bigl| \pEs[\errc] \bigr|^2, 
			\\
			\pEc \bigl[ \ERRs^2 \bigr]   
			&= 
			\pEc \bigl[ \ERRc^2 \bigr] 
			+ 
			\bigl| \PEs[\ERRc] \bigr|^2
			- 
			2 
			\pEc[\ERRc] 
			\PEs[\ERRc]
			= 
			\pV \bigl[ \ERR \bigr] 
			+ 
			\bigl| \PEs[\ERRc] \bigr|^2, 
		\end{align*}
		where the last equality holds, 
		since $\pEc[\ERRc] 
		= 
		\PEc[\ERRc] = 0$. 
		Combining these identities 
		with \eqref{eq:proof:mean-shifting-1} shows  
		the first relations in 
		\eqref{eq:decomp_means-0} and \eqref{eq:decomp_means-1}, 
		\[
		\frac{\PEs \bigl[ \errc^2 \bigr]  }{
			\pEc\bigl[ \errc^2 \bigr]}
		= 
		\frac{ \PV \bigl[ \err \bigr]}{
			\pV \bigl[ \err \bigr]}
		+ 
		\frac{\bigl| \pEs[\errc] \bigr|^2}{
			\pEc \bigl[ \errc^2 \bigr]}, 
		\qquad   
		\frac{\pEc \bigl[ \ERRs^2 \bigr]}{
			\pEc \bigl[ \errc^2 \bigr]} - 1   
		=  
		\frac{\pV \bigl[ \ERR \bigr]}{
			\pV \bigl[ \err \bigr]} - 1   
		+ 
		\frac{\bigl| \PEs[\ERRc] \bigr|^2}{
			\pEc \bigl[ \errc^2 \bigr]}
		\geq 
		\frac{\pV \bigl[ \ERR \bigr]}{
			\pV \bigl[ \err \bigr]} - 1 
		.  
		\]
		By changing the roles of $\mu,\MU$, 
		these two relations  
		imply that 
		\[
		\frac{\pEs \bigl[ \ERRc^2 \bigr]}{
			\PEc \bigl[ \ERRc^2 \bigr]} 
		= 
		\frac{ \pV \bigl[ \ERR \bigr]}{
			\PV \bigl[ \ERR \bigr]}
		+ 
		\frac{\bigl| \PEs[ \ERRc ] \bigr|^2}{
			\PEc \bigl[ \ERRc^2 \bigr]}, 
		\qquad\quad  
		\frac{\PEc \bigl[ \errs^2 \bigr]}{
			\PEc \bigl[ \ERRc^2 \bigr]} - 1   
		\geq  
		\frac{ \PV \bigl[ \err \bigr]}{
			\PV \bigl[ \ERR \bigr]} - 1  , 
		\]
		which combined with 
		\eqref{eq:proof:mean-shifting-5}, 
		\eqref{eq:proof:mean-shifting-6} and 
		\eqref{eq:proof:mean-shifting-7} 
		yield the remaining 
		relations in 
		\eqref{eq:decomp_means-0}, \eqref{eq:decomp_means-1}.   
		
		To derive \eqref{eq:decomp_means-2}, 
		we recall that 
		$\frac{\pV [ \ERR ]}{
			\pV [ \err ]} - 1 \geq 0$. 
		By invoking 
		that 
		\[
		\frac{\pEc[ \ERRc^2 ]}{\PEc[ \ERRc^2 ]} 
		\leq c_1^2  
		\qquad 
		\text{and} 
		\qquad 
		\frac{\pEc[ \ERRc^2 ]}{\pEc[ \errc^2 ]} 
		\geq 1, 
		\] 
		we therefore obtain 
		\begin{align*} 
			\frac{\pEc \bigl[ \ERRs^2 \bigr]}{
				\pEc \bigl[ \errc^2 \bigr]} - 1   
			= 
			\frac{\pV \bigl[ \ERR \bigr]}{
				\pV \bigl[ \err \bigr]} - 1   
			+ 
			\frac{\bigl| \PEs[\ERRc] \bigr|^2}{
				\pEc \bigl[ \errc^2 \bigr]}
			\geq 
			\frac{\bigl| \PEs[\ERRc] \bigr|^2}{
				\pEc \bigl[ \errc^2 \bigr]}
			&=  
			\frac{\bigl| \PEs[\ERRc] \bigr|^2} 
			{\PEc\bigl[ \ERRc^2 \bigr]}
			\, 
			\frac{\PEc\bigl[ \ERRc^2 \bigr]}
			{\pEc\bigl[ \ERRc^2 \bigr]}
			\, 
			\frac{\pEc\bigl[ \ERRc^2 \bigr]}{
				\pEc\bigl[ \errc^2 \bigr]}
			\\
			&\geq 
			c_1^{-2} 
			\frac{\bigl| \PEs[\ERRc] \bigr|^2} 
			{\PEc\bigl[ \ERRc^2 \bigr]}. 
		\end{align*} 
		This shows the first inequality in 
		\eqref{eq:decomp_means-2}. 
		Since $\frac{\PEc[\errc^2]}{\pEc[\errc^2]} \leq c_0^2$
		is also bounded, we can exchange the roles of $\mu,\MU$ 
		which combined with 
		\eqref{eq:proof:mean-shifting-5} and 
		\eqref{eq:proof:mean-shifting-6} 
		show the second inequality in  
		\eqref{eq:decomp_means-2}. 
		This completes the proof. 
	\end{proof} 
	
	\begin{remark}\label{rem:uniformly-bounded} 
		Note that, if for a sequence 
		$\{\cH_n\}_{n\in\bbN} \in \cS^{\mu}_{\mathrm{adm}}$ 
		(with $\cS^{\mu}_{\mathrm{adm}}$ defined as  
		in Equation~(2.12) in Section~2.3) any of 
		the left-hand sides in \eqref{eq:decomp_means-0} 
		is bounded uniformly in $n\in\bbN$ and in $h\in\cH_{-n}$, 
		then so are both terms on the corresponding right-hand side, 
		because all arising terms are nonnegative. 
	\end{remark}

\end{appendix}

%====================================================================

%====================================================================

% AOS,AOAS: If there are supplements please fill:
%\begin{supplement}[id=suppA]
%  \sname{Supplement A}
%  \stitle{Title}
%  \slink[doi]{10.1214/00-AOASXXXXSUPP}
%  \sdatatype{.pdf}" 
%  \sdescription{Some text}
%\end{supplement}

\end{document}